\documentclass[makeidx]{amsart}
\usepackage{amssymb}
\usepackage{latexsym}
\usepackage{amsmath}
\usepackage{euscript}
\usepackage[hypertex]{hyperref}

     \usepackage{color,colortbl}

   \def\sH{{\mathfrak H}}   
   \def\sK{{\mathfrak K}}   \def\sL{{\mathfrak L}}
\def\sM{{\mathfrak M}}   \def\sN{{\mathfrak N}}

      \def\sX{{\mathfrak X}}
\def\sY{{\mathfrak Y}}

      \def\dC{{\mathbb C}}

   \def\dN{{\mathbb N}}   
      \def\dR{{\mathbb R}}

   \def\cW{{\mathcal W}}

\def\bm\chi{\mbox{\boldmath$\chi$}}
\def\half{{\frac{1}{2}}}

\def\ran{{\rm ran\,}}
\def\cran{{\rm \overline{ran}\,}}
\def\dom{{\rm dom\,}}
\def\dist{{\rm dist\,}}
\def\mul{{\rm mul\,}}
\def\cker{{\rm \overline{ker}\,}}
\def\cmul{{\rm \overline{mul}\,}}
\def\cdom{{\rm \overline{dom}\,}}
\def\clos{{\rm clos\,}}

\def\dim{{\rm dim\,}}

\let\xker=\ker \def\ker{{\xker\,}}
\def\span{{\rm span\,}}
\def\cspan{{\rm \overline{span}\, }}

\def\cmr{{\dC \setminus \dR}}

\def\liczp#1{{${#1}^{\text {\rm o}}$}}

\def\sbar#1{\,\overline{\!#1}}

\DeclareMathOperator{\lin}{span\,}
\DeclareMathOperator{\I}{i}
\DeclareMathOperator{\RE}{\mathfrak{re\,}}
\DeclareMathOperator{\IM}{\mathfrak{im\,}}
\DeclareMathOperator{\okr}{{\stackrel{{\scriptscriptstyle{def}}}{=}}}
\def\zb#1#2{\{#1\,;\;#2\}}
\DeclareMathOperator{\hplus}{\, \widehat{\,+} \,}
\DeclareMathOperator{\hoplus}{\, \widehat{\, \oplus} \,}

\newtheorem{theorem}{Theorem}[section]
\newtheorem{proposition}[theorem]{Proposition}
\newtheorem{corollary}[theorem]{Corollary}
\newtheorem{lemma}[theorem]{Lemma}

\theoremstyle{definition}
\newtheorem{example}[theorem]{Example}
\newtheorem{remark}[theorem]{Remark}
\numberwithin{equation}{section}

\makeindex

\begin{document}
\title[Decompositions of linear relations]
{Componentwise and Cartesian decompositions \\ of linear relations}
\author{S. Hassi}
\author{H.S.V. de Snoo}
\author{F.H. Szafraniec}
\address{Department of Mathematics and Statistics \\
University of Vaasa \\
P.O. Box 700, 65101 Vaasa \\
Finland}
\email{sha@uwasa.fi}
\address{Department of Mathematics and Computing Science\\
University of Groningen \\
P.O. Box 407, 9700 AK Groningen \\
Nederland}
\email{desnoo@math.rug.nl}
\address{Instytut        Matematyki,         Uniwersytet
   Jagiello\'nski, ul. \L ojasiewicza 6, 30 348 Krak\'ow, Poland}
\email{fhszafra@im.uj.edu.pl}


\dedicatory{Dedicated to Sch\^oichi \^Ota on the occasion of his
sixtieth birthday}

\date{relation26a.tex; May 28, 2009}

\thanks{The first and second author were supported by the V\"ais\"al\"a
Foundation of the Finnish Academy of Science and Letters. The
third author was supported by the Dutch Organization for
Scientific Research NWO and partially supported by the MNiSzW grant N201 026
32/1350. He also would like to acknowledge assistance of the EU
Sixth Framework Programme for the Transfer of Knowledge ``Operator
theory methods for differential equations'' (TODEQ) \#
MTKD-CT-2005-030042.}




\begin{abstract}
Let $A$ be a, not necessarily closed, linear relation in a Hilbert
space $\sH$ with a multivalued part $\mul A$. An operator $B$ in
$\sH$ with $\ran B\perp\mul A^{**}$ is said to be an operator part
of $A$ when $A=B \hplus (\{0\}\times \mul A)$, where the sum is
componentwise (i.e. span of the graphs). This decomposition provides
a counterpart and an extension for the notion of closability of
(unbounded) operators to the setting of linear relations. Existence
and uniqueness criteria for the existence of an operator part are
established via the so-called canonical decomposition of $A$. In
addition, conditions are developed for the decomposition to be
orthogonal (components defined in orthogonal subspaces of the
underlying space). Such orthogonal decompositions are shown to be
valid for several classes of relations. The relation $A$ is said to
have a Cartesian decomposition if $A=U+\I V$, where $U$ and $V$ are
symmetric relations and the sum is operatorwise. The connection
between a Cartesian decomposition of $A$ and the real and imaginary
parts of $A$ is investigated.
\end{abstract}

\maketitle \tableofcontents
\section{Introduction}

\subsection{Some background}

A linear \textit{relation} \index{relation} $A$ in a Hilbert space
$\sH$ is by definition a linear subspace of the product space $\sH
\times \sH$. A linear relation $A$ is (the graph of) a linear
operator if and only if $\mul A=\{0\}$, where the
\textit{multivalued} part \index{part of relation!multivalued
$\mul A$}$\mul A$ of $A$ is defined as $\{g\in \sH ;\, \{0,g\} \in
A\}$. The formal \textit{inverse} $A^{-1}$ \index{relation!inverse
$A^{-1}$}of a linear relation $A$ is given by $A^{-1}=\{
\{k,h\};\,\{h,k\}\in A\}$, so that $\dom A^{-1}=\ran A$, $\ran
A^{-1}=\dom A$, $\ker A^{-1}=\mul A$, and $\mul A^{-1}=\ker A$.
The \textit{closure} \index{relation!closure $\sbar A$} of a
linear relation is a linear relation which is obtained by taking
the closure of the corresponding subspace in $\sH \times \sH$. The
linear relation $A$ is called \textit{closed}
\index{relation!closed} as a relation in $\sH$ if the subspace is
closed in $\sH \times \sH$. If $A$ is (the graph of) a linear
operator, then $A$ is said to be closable if the closure of $A$ is
(the graph of) a linear operator. The \textit{adjoint}
\index{relation!adjoint $A^*$} $A^*=JA^\perp=(JA)^\perp$, with the
operator $J$ defined by $J\{f,f'\}=\{f',-f\}$, $\{f,f'\} \in \sH
\times \sH$, is automatically a closed linear relation in $\sH$.
Then the second adjoint $A^{**}$ is equal to the closure
$\sbar{A}$ of $A$. A relation is said to be \textit{symmetric}
\index{relation!symmetric} if $A \subset A^*$ and
\textit{selfadjoint} \index{relation!selfadjoint} if $A=A^*$. The
study of general relations was initiated by R.~Arens \cite{Ar61}.
Further work has been concerned with symmetric and selfadjoint
relations and, more generally, with normal, accretive,
dissipative, and sectorial relations, see for instance \cite{Ar2},
\cite{Yury}, \cite{Co73}, \cite{DS}, \cite{RB}.

Linear relations can be viewed as multivalued linear operators.
They show up in a natural way in a variety of problems. Some of
these will be presented for the convenience of the reader.

The first example shows the usefulness of relations by relating
results for $A$ to those for the formal inverse $A^{-1}$.

\begin{example}\label{adjex}
Let $A$ be a linear operator or a linear relation
in a Hilbert space $\sH$, which is not
necessarily closed or densely defined. An element $h \in
\sH$ belongs to $\dom A^*$ if and only if
\begin{equation}\label{Een}
  \sup \left\{\,(h,g)+(g,h)-(f,f)\,;\, \{f,g\} \in A\,\right\}<\infty,
\end{equation}
and an element $k \in \sH$ belongs to $\ran A^*$ if and only if
\begin{equation}\label{Twee}
  \sup\left\{\,(f,k)+(k,f)-(g,g)\,;\, \{f,g\} \in A\,\right\}<\infty.
\end{equation}
The formulas \eqref{Een} and \eqref{Twee} show the advantage of
the language of relations: the formula \eqref{Twee} is in fact the
same as the formula \eqref{Een} when the relation $A$ is replaced
by its formal inverse $A^{-1}$. Moreover,  \eqref{Een} is equivalent to
\begin{equation}\label{Drie}
  \sup \left\{\,|(g,h)|^2\,;\, \{f,g\} \in A, \, (f,f) \le
1\,\right\} <\infty,
\end{equation}
and  \eqref{Twee} is equivalent to
\begin{equation}\label{Vier}
  \sup\left\{\,|(f,k)|^2\, ;\, \{f,g\} \in A, \,(g,g) \le 1\,
\right\} <\infty.
\end{equation}
Again the relation between \eqref{Drie} and \eqref{Vier} via the
formal inverse of $A$ is evident. The last two characterizations are
versions of results which go back to Shmulyan for bounded operators $A$;
for more details see \cite{HSeS}.
\end{example}

As a second example
it is shown that under very general conditions
a densely defined closable operator can be decomposed as
the sum of a closable operator and
a singular operator (whose closure is a Cartesian product).

\begin{example}
Let $A$ be a densely defined closable operator in a Hilbert space $\sH$;
i.e., the closure $\overline{A}$ of $A$ in $\sH \times \sH$ is the graph
of a linear operator. Let $\varphi \in \sH$ and let $P_{\varphi}$ be
the orthogonal projection
from $\sH$ onto the linear space spanned by $\varphi$. Then the
operator $A$ admits the decomposition
\begin{equation}\label{decomp}
 A=B+C,
\end{equation}
with the densely defined operators $A$ and $C$ defined by
\begin{equation}\label{decomp1}
 B=(I-P_\varphi)A, \quad C=P_\varphi  A.
\end{equation}
Then the operator $B$ is closable for any choice of $\varphi \in
\sH$, but  the behaviour of the operator $C$ depends on the choice
of $\varphi \in \sH$.  If $\varphi \in \dom A^*$, then $\overline{C}
\in \boldsymbol{B}(\sH)$ and  $\overline{C} h=(h,A^* \varphi) \varphi$
for $h \in \sH$. However,  if $\varphi \in \sH \setminus \dom A^*$,
then $C$ is a so-called singular operator, i.e., $\ran C\subset \mul
\overline{C}$ and $\overline{C}=\sH \times \lin \{\varphi\}$. For
more details and the connection with Lebesgue type decompositions,
see \cite{HSS??}.
 \end{example}

As a third example consider the case of a monotonically increasing
sequence of bounded linear operators in the absence of a uniform
upper bound.

\begin{example}
Let $\sH$ be a Hilbert space and let $A_n \in \boldsymbol{B}(\sH)$
(bounded linear operators on $\sH$) be a nondecreasing sequence of
nonnegative operators, i.e., $0\le (A_{ m} h,h) \le (A_n h,h)$, $h
\in \sH$, for $n \ge m$. If the sequence $A_n$ is bounded from
above, i.e., $(A_n h,h) \le M (h,h)$, $h \in \sH$, for some $M \ge
0$, then it is known that there exists a strong limit $A_\infty
\in \boldsymbol{B}(\sH)$, i.e., $\|A_n h-A_\infty h\|\to 0$, $h
\in \sH$, and $A_\infty$ \index{relation!$A_\infty$}has the same
upper bound. The situation is different when the family $A_n$ does
not have an upper bound. The absence of a uniform bound leads to
phenomena, which involve unbounded operators and relations. In
fact there exists a selfadjoint relation $A_\infty$, which is
nonnegative, i.e., $(f',f) \ge 0$, $\{f,f'\}\in A_\infty$, such
that $A_n$ converges to $A_\infty$ in the strong resolvent sense,
i.e.,
\[
 (A_n-\lambda)^{-1} h \to (A_\infty-\lambda)^{-1}h, \quad h \in \sH, \quad
 \lambda \in \cmr.
\]
Morover, the domain of the square root of
$A_\infty$\index{relation!$A_\infty$} is given by
\begin{equation*}\label{quatre}
\dom
A_\infty^\half=\{\, h \in \sH\, ; \,  \sup_{n \in \dN} \, (A_n h,h) < \infty\,\}.
\end{equation*}
 For more details and the connection with monotone sequences of
semibounded closed forms, see \cite{BHSW}.
\end{example}

Often multivalued operators appear as extensions of symmetric
operators, like in boundary value problems for differential
operators. Boundary conditions impose restrictions in such a way
that an underlying symmetric operator becomes nondensely defined;
cf. \cite{CoDij}.
The situation can often be formalized in simple terms as follows.

\begin{example}
Let $A$ be a selfadjoint operator in a Hilbert space $\sH$ and let
$Z$ be, for simplicity, a finite-dimensional subspace of $\sH \times
\sH$. Then the intersection $A \cap Z^*$ is a symmetric restriction
of $A$, which may be nondensely defined, so that its adjoint $A
\hplus Z$, a componentwise sum, may be multivalued. In this case,
among the selfadjoint extensions of $A \cap Z^*$, there also occur
multivalued operators. In the connection of differential operators
this construction gives rise to nonstandard boundary conditions. If,
for instance, $A$ is a selfadjoint Sturm-Liouville operator, then
integral boundary conditions or perturbations via delta functions or
their derivatives fit into this framework with a proper choice of
the subspace $Z$; see \cite{HSSW} for more details
\end{example}

In general, the spectral theory of differential equations offers
many examples of multivalued operators. Linear relations provide the
natural context for the study of general selfadjoint boundary value
problems involving systems of differential equations; cf. \cite{BHSW1}.
In fact, the theory of boundary triplets and boundary relations has been
formulated to discuss all extensions (single-valued and multivalued)
of symmetric relations; see \cite{DHMS},  \cite{DM}.
For instance,  the description of selfadjoint extensions of a symmetric
operator or relation is always in terms of selfadjoint relations in a parameter
space; such selfadjoint relations also appear in Kre\u{\i}n's formula.

\subsection{Decomposition of relations}\label{decc}   \index{decomposition of relations}

There are many kinds of  decompositions of linear relations, for
instance, for semi-Fredholm relations and for quasi-Fredholm
relations there is a so-called Kato decomposition, see \cite{LSSW},
or Stone decomposition for closed linear relations, see
\cite{HSSS07}, \cite{Mez}. The decompositions appearing in the present
paper are concerned with splitting linear operators and relations
via components, that are closable, nonclosable, or purely
multivalued, and components involving the real and imaginary parts
of relations.

It is necessary to begin by explaining the so-called
\textit{canonical decomposition} \index{decomposition of
relations!canonical} of linear relations which has been studied
recently in \cite{HSSS07}. Let $A$ be a relation in a Hilbert
space and let $A^{**}$ be its closure. Let $P$ the orthogonal
projection from $\sH$ onto $\mul A^{**}$ and define the relations
$A_{\rm reg}$ and $A_{\rm sing}$, the \textit{regular} part and
the \textit{singular} part of $A$ respectively, \index{part of
relation!regular $A_{\rm reg}$} \index{part of relation!singular
$A_{\rm sing}$} by
\[
A_{\rm reg}=\{\,\{f,(I-P)f'\};\,\{f,f'\} \in A\,\},
\quad A_{\rm sing}=\{\,\{f,Pf'\};\,\{f,f'\} \in
A\,\}.
\]
Then $A$ admits the decomposition
\[
 A=A_{\rm reg}+A_{\rm sing}
 =\left\{\, \{f,h+k\};\,\{f,h\} \in A_{\rm reg}, \{f,k\} \in A_{\rm sing}\,\right\}.
\]
The regular part  $A_{\rm reg}$ is actually a closable operator,
whereas the singular part $A_{\rm sing}$ is a singular relation,
i.e., its closure is a Cartesian product, cf. \cite{HSSS07}. The
canonical decomposition of $A$ above is strongly related to the
Lebesgue decompositions of forms, see \cite{HSSS07}. The canonical
decomposition of a relation is an example of a decomposition as an
operatorwise sum. However, relations also admit componentwise
decompositions. The aim of this paper is to present several
decompositions of linear relations as operatorlike sums and as
componentwise sums.

The second type of decomposition introduced in the present paper
for general, not necessarily closed, linear relations is a
\textit{componentwise decomposition} \index{decomposition of
relations!componentwise $\hplus$} of a relation $A$ in an operator
part and a multivalued part of the form
\begin{equation}\label{een}
A=B \hplus A_{\rm mul},
\end{equation}
where the operator part $B$ is (the graph of) an operator in
$\sH$, $A_{\rm mul}=\{0\} \times \mul A$, and the sum in
\eqref{een} is componentwise (as indicated by $\hplus$). To make
the decomposition somewhat reasonable or unique it is necessary to
impose some additional assumptions on \eqref{een}. Assume for the
moment that the relation $A$ is closed. Then one possible choice
is $B=A_{\rm op}$ where $A_{\rm op}=\{\{f,f'\} \in A ;\, f' \perp
\mul A\}$, \index{part of relation!minimal operator $A_{\rm
op}$}so that $B$ is a closed operator. Since $\mul A$ is closed
and $\sH=\cdom A^* \oplus \mul A$, the identity \eqref{een}
follows. This motivates the construction in the general case. The
extra assumption that $\ran B \subset \cdom A^*=(\mul
A^{**})^\perp$ makes $B$ unique, namely $B=A_{\rm op}$, where now
\[
A_{\rm op}=\{\{f,f'\} \in A ;\, f' \perp \mul A^{**} \}
\]
is a closable operator. Observe that $A_{\rm op} \subset A_{\rm
reg}$. It will be shown that $B=A_{\rm op}$ satisfies \eqref{een}
precisely when $A_{\rm op} = A_{\rm reg}$. A relation $A$ which
allows a decomposition \eqref{een} with $\ran B \subset \cdom A^*$
will be called \textit{decomposable}\index{relation!decomposable}.

The third decomposition is related to the second type of
decomposition, so it is again componentwise. Assuming that $A$ is
decomposable the question is when the decomposition \eqref{een} is
orthogonal with regard to the \textit{orthogonal splitting}
\index{orthogonal splitting} of the Hilbert space
\[
\sH=\cdom A^* \oplus \mul A^{**}.
\]
A necessary and sufficient additional condition that appears now for
$A$ is
\[
\dom A \subset \cdom A^* \quad \mbox{or, equivalently,} \quad \mul
A^{**} \subset \mul A^*.
\]
Particular cases are studied for
decomposable relations $A$ which are in addition formally domain
tight and domain tight, i.e., satisfy
\[
\dom A \subset \dom A^* \quad \mbox{and} \quad \dom A=\dom A^*,
\]
respectively. Furthermore, decomposable relations are studied under
the condition that their numerical range is a proper subset of
$\dC$. Orthogonal decompositions for normal, selfadjoint, and, for
instance, maximal sectorial relations are obtained as byproducts.

The fourth type of decomposition to be studied in the present
paper is the Cartesian decomposition of a relation. By definition
a \textit{Cartesian decomposition} \index{decomposition of
relations!Cartesian}of a relation $A$ is of the form
\begin{equation}\label{twee}
A=\RE A+\I \IM A,
\end{equation}
where $\RE A$ and $\IM A$ \index{part of relation!real $\RE A$}
\index{part of relation!imaginary $\IM A$}are symmetric relations
in $\sH$, i.e.,
\[
\RE A\subset (\RE A)^*, \quad \IM A\subset (\IM A)^*,
\]
and where the sum in \eqref{twee} is now again operatorwise; see
\cite{StSz} for the operator case. It is a consequence of the
Cartesian decomposition \eqref{twee} that $A$ satisfies the
condition $\dom A \subset \dom A^*$, and it will be shown that
this is also a sufficient condition for the existence of a
Cartesian decomposition. The connection between the components
$\RE A$ and $\IM A$ of a Cartesian decomposition \eqref{twee} of
$A$ and the real and imaginary parts of $A$ is clear if $A$ is a
densely defined normal operator, cf. \cite{StSz}. In the general
case, the connection is vague, but the situation becomes clear
when the following extension of $A$ is introduced:
\index{relation!$A_\infty$}
\[
A_\infty=A \hplus (\{0\}\times \mul A^*).
\]
The special situation of Cartesian
decompositions for normal relations will be treated in
\cite{HSSz??}.

\subsection{Brief description}
Here is a brief review of the contents of the paper. Section~2
contains a number of preliminary definitions and facts concerning
linear relations. A number of results which are known for linear
operators are stated for the case of linear relations; for
completeness proofs are included. The notions of formally domain
tight and domain tight relations are introduced. Canonical
decompositions and decompositions of linear relations of the form
\eqref{een} are taken up in Section~3. The notion of decomposable
relation is characterized in various ways. A number of examples is
included illustrating relations which are not decomposable. The
question of the orthogonality of such decompositions is taken up in
Section~4. In particular, relations whose numerical range is a
proper subset of $\dC$ are treated. Cartesian decompositions of the
form \eqref{twee} are treated in Section~5. This section also
contains a treatment of the real and imaginary parts of a linear
relation.

\tableofcontents

\section{Preliminaries}

This section contains a number of basic definitions and results concerning
linear relations in a Hilbert space.
These results are analogs or natural extensions of results which are better
known in the case of operators.
It should be mentioned that many of the stated results have their analogs also
for linear relations acting from one Hilbert space to another Hilbert space.
However, for simplicity all the statements are formulated here for the case
of linear relations from a given Hilbert space back to itself.

\subsection{Linear relations in a Hilbert space}

Let $\sH$ be a Hilbert space with inner product $(\cdot,\cdot)$.
The Cartesian product $\sH \times \sH$ of $\sH$ with itself, will
be provided with the usual inner product. A linear
\textit{relation} (or relation, for short) $A$ in $\sH$ is by
definition a linear subspace of the Hilbert space $\sH \times
\sH$. The \textit{domain}\index{relation!domain $\dom A$},
\textit{range}\index{relation!range $\ran A$},
\textit{kernel}\index{relation!kernel $\ker A$}, and
\textit{multivalued part}\index{part of relation!multivalued $\mul
A$} of $A$ are denoted by $\dom A$, $\ran A$, $\ker A$, and $\mul
A$:
\[
\begin{array}{ll}
 \dom A \okr \zb{f}{\{f,f'\} \in A},
 &\ker A \okr \zb{f}{\{f,0\} \in A}, \\
 \ran A \okr \zb{f'}{\{f,f'\} \in A},
 &\mul A \okr \zb{f'}{\{0,f'\} \in A};
\end{array}
\]
they are linear subspaces of $\sH$. An operator\index{operator} is
a relation when it is identified with it graph. Clearly in this
sense A relation $A$ is an operator precisely when $\mul A
=\{0\}$. Define the \textit{inverse}\index{relation!inverse
$A^{-1}$} of $A$ by
\[
 A^{-1} \okr \{\,\{f',f\}\, ;\; \{f,f'\}\in A\,\},
\]
then, by complete symmetry,
\[
\begin{array}{ll}
 \dom A^{-1}=\ran A, & \ker A^{-1}=\mul A, \\
 \ran A^{-1}=\dom A, & \mul A^{-1}=\ker A.
\end{array}
\]
A relation $A$ is \textit{closed} if it is closed as a subspace of
$\sH\times \sH$, in which case $\ker A$ and $\mul A$ are closed
subspaces of $\sH$. The closure of a relation $A$ in $\sH \times
\sH$ is denoted by $\clos A$; the notations $\cdom A$ and $\cran
A$ indicate the closures of $\dom A$ and $\ran A$ in $\sH$,
respectively. The closure of (the graph of) an operator is a
closed relation which is not necessarily (the graph of) an
operator. An operator is said to be
\textit{closable}\index{operator!closable} if the closure of its
graph is (the graph of an) operator. In what follows, the class of
bounded everywhere defined operators on $\sH$ is denoted by
$\boldsymbol{B}(\sH)$.

Observe that
for any relation $A$ one has
\begin{equation}\label{impor}
 \cdom (\clos A)=\cdom A, \quad  \cran (\clos A)=\cran A.
\end{equation}
Sometimes these identities can be improved. The following result for
bounded operators is standard; an extension for linear relations
will appear later in Corollary~\ref{mulclosed} (see also
Proposition~\ref{Dclosed}). A proof is given here for completeness.

\begin{lemma}\label{first}
Let $A$ be a bounded, not necessarily densely defined,
operator in a Hilbert space $\sH$. Then
\begin{enumerate}
\def\labelenumi{(\roman{enumi})}

\item $A$ is closed if and only if $\dom A$ is closed;

\item $A$ is closable and $\clos A$ is bounded with $\|\clos
A\|=\|A\|$;

\item $\dom (\clos {A})=\cdom A$.
\end{enumerate}
\end{lemma}

\begin{proof}
(i) Assume that $A$ is closed. If the sequence $f_n \in \dom A$
tends to $f \in \sH$, then the inequality $\|A(f_n-f_m)\| \le
\|A\|\|f_n-f_m\|$ shows that $Af_n$ is a Cauchy sequence, so that
$Af_n \to g$ for some $g \in \sH$. Therefore $\{f_n,Af_n\} \to
\{f,g\}$ which implies that $f \in \dom A$ and $g=Af$, since $A$
is closed. In particular, $\dom A$ is closed.

Conversely, assume that $\dom A$ is closed. Let the sequence
$\{f_n,Af_n\} \in A$ converge to $\{f,g\}$. Then $f \in \dom A$
since $\dom A$ is closed. It follows from the inequality
$\|Af_n-Af\|\le \|A\|\|f_n-f\|$ that $Af_n \to Af$, in other
words, $g=Af$ or, equivalently, $\{f,g\} \in A$. Hence, $A$ is
closed.

(ii) In order to show that $A$ is closable, assume that $\{0,g\}
\in \clos A$. Then there is a sequence $\{f_n,Af_n\}\in A$ such
that $\{f_n,Af_n\} \to \{0,g\}$, i.e., $f_n \to 0$ and
$Af_n \to g$. However, $f_n \to 0$ implies that $Af_n \to 0$, so that $g=0$.
Thus, $A$ is closable.

As to boundedness, recall that by definition
\[
 \|A\|=\sup\{\,\|Af\|:\, f\in\dom A,\,\, \|f\|\le 1 \,\}.
\]
Since $\clos A$ is an operator and every $f\in \dom(\clos A)$ can
be approximated by a sequence $f_n\in\dom A$ with $f_n\to f$ and
$Af_n\to (\clos A)f$, the equality $\|\clos A\|=\|A\|$ follows
easily from the above definition of the operator norm.

(iii) It follows from \eqref{impor} that $\dom (\clos {A})\subset
\cdom (\clos {A})=\cdom A$.

Conversely, assume that $f \in \cdom A$. Then there exists a
sequence $f_n \in \dom A$ such that $f_n \to f$. Since $Af_n$ is a
Cauchy sequence there exists an element $g$ such that $Af_n \to g$.
Observe that $\{f,g\} \in \clos A$. By (ii) $\clos A$ is an
operator, and hence $f \in \dom (\clos A)$ and $g=(\clos A)f$.
\end{proof}

The following statement concerning closable extensions of bounded
densely defined operators is an immediate consequence of
Lemma~\ref{first}.

\begin{corollary}\label{firstcor}
If $A \subset B$, $B$ is closable, and $A$ is bounded and densely
defined, then $B$ is bounded and, moreover, $B^{**}=A^{**}$.
\end{corollary}

The assumption that $B$ is closable is essential in
Corollary~\ref{firstcor}; cf. Example~\ref{exam1}.

\subsection{Adjoint relations}

Let $A$ be a relation in a Hilbert space $\sH$. The
\textit{adjoint}\index{relation!adjoint $A^*$} $A^*$ of $A$ is the
closed (automatically linear) relation defined by
\begin{equation*}\label{adjo}
 A^*\okr \{\,\{f,f'\} \in \sH\times\sH \,;\; \langle \{f,f'\},
 \{h,h'\}\rangle=0 \mbox{ for all } \{h,h'\} \in A\,\},
\end{equation*}
where the form $\langle \cdot,\cdot\rangle$ is defined by
\[
 \langle \{f,f'\}, \{h,h'\}\rangle=(f',h)-(f,h'), \quad \{f,f'\},
\{h,h'\} \in \sH\times\sH.
\]
Note that the adjoint $A^*$ is given by
\begin{equation}\label{FF}
A^*=JA^\perp=(JA)^\perp,
\end{equation}
where the operator $J$, defined by
\begin{equation}\label{J}
J\{f,f'\}=\{f',-f\}, \quad \{f,f'\} \in \sH \times \sH,
\end{equation}
is unitary in $\sH \times \sH$. If $A$ is a relation, then
$A^{**}=(A^*)^*$ gives the closure of $A$, i.e., $A^{**}=\clos A$,
due to \eqref{FF}.
Note that for two relations $A$ and $B$ one has
\begin{equation}\label{AB*}
  A\subset B \quad \Longrightarrow \quad B^* \subset A^*.
\end{equation}
Furthermore, it follows directly from the definition
that
\[
    (A^{-1})^*=(A^*)^{-1}.
\]

\begin{lemma}\label{lemrelate}
Let $A$ be a relation in a Hilbert space $\sH$. Then
\begin{equation}\label{eeeen}
(\dom A)^\perp=\mul A^*, \quad (\ran A)^\perp=\ker A^*,
\end{equation}
and, likewise,
\begin{equation}\label{tweee}
(\dom A^*)^\perp=\mul A^{**}, \quad (\ran A^*)^\perp =\ker A^{**}.
\end{equation}
\end{lemma}

\begin{proof}
The first identity in \eqref{eeeen} follows from
\[
 \{0,g\} \in A^* \Longleftrightarrow \{0,g\} \in J(A^{\perp})
 \Longleftrightarrow \{g,0\} \in A^{\perp} \Longleftrightarrow g
 \in (\dom A)^{\perp}.
\]
The second identity is obtained by going over to the inverse. The
identities in \eqref{tweee} follow from those in \eqref{eeeen} by
going over to the adjoint.
\end{proof}

In particular, observe that
\begin{equation}\label{neum}
    \mul A^{**}=\{0\} \quad \Longleftrightarrow \quad \dom A^*
\mbox{ dense in } \sH.
\end{equation}

\begin{lemma}\label{easy}
Let $A$ be a relation in a Hilbert space $\sH$. Then the following
equivalences are valid:
\begin{equation}\label{impor0}
 \dom A \subset \cdom A^* \quad \Longleftrightarrow \quad \cdom A
\subset \cdom A^* \quad \Longleftrightarrow \quad \mul A^{**}
\subset \mul A^*,
\end{equation}
and, likewise,
\begin{equation}\label{impor1}
 \dom A^* \subset \cdom A \quad \Longleftrightarrow \quad \cdom
 A^* \subset \cdom A \quad \Longleftrightarrow \quad \mul A^{*}
 \subset \mul A^{**}.
\end{equation}
In particular,
\begin{equation}\label{impor2}
 \cdom A = \cdom A^* \quad \Longleftrightarrow \quad \mul A^{**} =
 \mul A^*,
\end{equation}
\end{lemma}

\begin{proof}
The first equivalence in \eqref{impor0} is valid since the subspace
$\cdom A^*$ of $\sH$ is closed. The second equivalence in
\eqref{impor0} is based on the identity $\mul A^*=(\dom A)^\perp$.
The equivalences in \eqref{impor1} follow if in \eqref{impor0} the
relation  $A$ is replaced by the relation $A^*$ and the identity
\eqref{impor} is used. The identity \eqref{impor2} is now obvious.
\end{proof}

It is a consequence of Lemma \ref{lemrelate}  that the Hilbert space $\sH$ has
the following orthogonal decompositions:
\[
 \sH=\cdom A^{**} \oplus \mul A^*, \quad \sH=\cran A^{**} \oplus \ker A^*.
\]
However, there are also similar, nonorthogonal, decompositions of $\sH$.

\begin{lemma}\label{domranlemma}
Let $A$ be a relation in a Hilbert space $\sH$. Then
\begin{equation}\label{eqq1}
 \sH=\dom A^{**}+\ran A^*,\quad \sH=\ran A^{**}+\dom A^*.
\end{equation}
\end{lemma}

\begin{proof}
Recall from \eqref{FF} that $JA^*=A^\perp$. This implies that
$\sH\times\sH= A^{**} \oplus JA^*$, which leads to \eqref{eqq1}.
\end{proof}

\subsection{Special relations}

A relation $A$ is said to be
\textit{symmetric}\index{relation!symmetric} if $A \subset A^*$; a
relation is symmetric if and only if $(g,f) \in \dR$ for all
$\{f,g\} \in A$. A relation $A$ is said to be \textit{essentially
selfadjoint}\index{relation!essentially selfadjoint} if
$A^{**}=A^*$ and it is said to be
\textit{selfadjoint}\index{relation!selfadjoint} if $A=A^*$. A
relation $A$ in a Hilbert space $\sH$ is said to be
\textit{formally normal}\index{relation!formally normal} if there
exists an isometry $V$ from $A$ into $A^*$ of the form
\[
V\{f,g\}=\{f,h\}, \quad \{f,g\} \in A, \quad \{f,h\} \in A^*,
\]
i.e., $V$ leaves the first component $f$ invariant and
$\|g\|=\|h\|$. A formally normal relation $A$ in a Hilbert space
$\sH$ is said to be \textit{normal}\index{relation!normal} if the
isometry $V$ is from $A$ onto $A^*$. Normal relations and
consequently selfadjoint relations are automatically closed.
Finally, a relation $A$ in $\sH$ is said to be
\textit{subnormal}\index{relation!subnormal} if there exists a
Hilbert space $\sK$ containing $\sH$ isometrically and a normal
relation $B$ in $\sK$ such that $A\subset B$.

\subsection{Sums and products} \index{sum of relations}

Let $A_1$ and $A_2$ be relations in $\sH$. The notation $A_1
\hplus A_2$ denotes the \textit{componentwise sum}\index{sum of
relations!componentwise $\hplus$} of $A_1$ and $A_2$:
\begin{equation}\label{jan00}
 A_1 \hplus A_2\okr\{\,\{f_1+f_2,f'_1+f'_2\}\,;\; \{f_1,f'_1\} \in
 A_1, \, \{f_2,f'_2\} \in A_2\,\}.
\end{equation}
In particular,
\begin{equation*}\label{dommul0}
\dom (A_1 \hplus A_2)=\dom A_1 + \dom A_2, \quad \mul (A_1 \hplus
A_2)=\mul A_1 +\mul A_2.
\end{equation*}

\begin{lemma}\label{basiclem}
The componentwise sum satisfies the identities
\[
(A_1 \hplus A_2)^*=A_1^* \cap A_2^*,
\quad
\clos (A_1 \hplus A_2)=(A_1^* \cap A_2^*)^*.
\]
\end{lemma}

\begin{proof}
Observe that
\[
 (A_1 \hplus A_2)^*=J(A_1 \hplus A_2)^\perp=J(A_1^\perp \cap
 A_2^\perp)=JA_1^\perp \cap
 JA_2^\perp=A_1^*\cap A_2^*,
\]
according to the definition of the adjoint operation.
This gives the first identity and the second identity is obtained by taking
adjoints in the first one.
\end{proof}

The following identities are also clear:
\[
 \clos( A_1 \hplus A_2)=\clos (A_1 \hplus \clos A_2)=\clos (\clos
A_1 \hplus \clos A_2).
\]

The notation $A_1+A_2$ is reserved for the \textit{operatorwise
sum}\index{sum of relations!operatorwise $+$} of $A_1$ and $A_2$:
\begin{equation}\label{jan1}
 A_1+A_2\okr\{\,\{f,f'+f''\}\,;\; \{f,f'\} \in A_1, \, \{f,f''\} \in
A_2\,\}.
\end{equation}
In particular, it follows from the definition in \eqref{jan1} that
\begin{equation}\label{dommul}
\dom (A_1 +A_2)=\dom A_1 \cap \dom A_2, \quad \mul (A_1 +A_2)=\mul
A_1 +\mul A_2.
\end{equation}
In the case when $A_1$ and $A_2$ are operators this sum is the
(graph of the) usual operator sum.

\begin{lemma}
The operatorwise sum satisfies
\begin{equation}\label{jan1+}
 A_1^*+A_2^* \subset (A_1+A_2)^*.
\end{equation}
If $A_1$ or $A_2$ belongs to $\boldsymbol{B}(\sH)$, then
\begin{equation}\label{jan1++}
 A_1^*+A_2^* = (A_1+A_2)^*.
\end{equation}
\end{lemma}

\begin{proof}
Let $\{f,f_1'+f_2'\} \in A_1^*+A_2^*$ with $\{f,f_1'\}\in A_1^*$ and
$\{f,f_2'\}\in A_2^*$. Now assume that $\{h,h_1+h_2\}\in A_1+A_2$
with $\{h,h_1\}\in A_1$ and $\{h,h_2\}\in A_2$. Then
\[
 \langle \{f,f_1'+f_2'\}, \{h,h_1+h_2\}
 \rangle=(f_1',h)-(f,h_1)+(f_2',h)-(f,h_2)=0,
\]
which implies that $\{f,f_1'+f_2'\} \in (A_1+A_2)^*$. This shows
\eqref{jan1+}.

For the converse, let $\{f,f'\} \in (A_1+A_2)^*$, so that for all
$\{h,h_1\}\in A_1$ and $\{h,h_2\}\in A_2$
\[
 0=\langle \{f,f'\}, \{h,h_1+h_2\}\rangle=(f',h)-(f, h_1+h_2)=(f',h)-(f,h_1)-(f,h_2).
\]
Suppose that, for instance, $A_2 \in \boldsymbol{B}(\sH)$, then
$h_2=A_2h$ and the above identity implies that
\[
 (f'-A_2^*f,h)=(f,h_1),
\]
for all $\{h,h_1\} \in A_1$, so that $\{f, f'-A_2^*f\} \in A_1^*$.
Together with $\{f,A_2^*f\} \in A_2^*$, this means that $\{f,f'\}
\in A_1^*+A_2^*$. This shows \eqref{jan1++}.
\end{proof}

The notation $A_1 A_2$ denotes the \textit{product}\index{product
of relations} of $A_1$ and $A_2$:
\begin{equation}\label{jan0}
 A_1  A_2\okr\{\,\{f,f'\}\,;\; \{f,h\} \in
 A_2, \, \{h,f'\} \in A_1\,\}.
\end{equation}
In particular, $\mul A_1 \subset \mul (A_1 A_2)$.
Moreover, if $A_2$ is an operator, then  $\mul A_1 = \mul (A_1 A_2)$.
In the case when $A_1$ and $A_2$ are both  operators the product in \eqref{jan0} is the
(graph of the) usual operator product. The product of relations is
clearly associative. Observe that
\[
 AA^{-1}=I_{\ran A} \hplus (\{0\} \times \mul A), \quad
 A^{-1}A=I_{\dom A} \hplus (\{0\} \times \ker A),
\]
which shows that products of relation require some care. For
$\lambda \in \dC$ the notation $\lambda A$ agrees in this sense with
$(\lambda I)A$.

\begin{lemma}\label{prodlem*}
The product satisfies
\begin{equation}\label{feb1+}
A_2^* A_1^* \subset  (A_1A_2)^*.
\end{equation}
If $A_1$   belongs to $\boldsymbol{B}(\sH)$, then
\begin{equation}\label{feb1++}
 (A_1A_2)^* =A_2^*A_1^*.
\end{equation}
\end{lemma}

\begin{proof}
Let $\{f,f'\} \in A_2^* A_1^*$, so that $\{f,g\} \in A_1^*$ and
$\{g,f'\}\in A_2^*$. Now assume that $\{h,h'\} \in A_1A_2$, so that
$\{h,k\}\in A_2$ and $\{k,h'\}\in A_1$. Then
\[
 \langle \{f,f'\}, \{h,h'\}\rangle=(f',h)-(f,h')=(g,k)-(g,k)=0,
\]
which yields $\{f,f'\} \in (A_1A_2)^*$. This shows \eqref{feb1+}.

Conversely, let $\{f,f'\} \in (A_1A_2)^*$, so that for all
$\{h,h'\}\in A_1A_2$ one has
\[
0=\langle \{f,f'\} , \{h,h'\} \rangle=(f',h)-(f,h').
\]
However, since $A_1 \in \boldsymbol{B}(\sH)$ it is easily seen that
$\{h,h'\}\in A_1A_2$ if and only if  $\{h,k\} \in A_2$ and
$h'=A_1k$. Hence, $\{f,f'\} \in (A_1A_2)^*$ if and only if for all
$\{h,k\} \in A_2$:
\[
 0=(f',h)-(f,A_1k)=(f',h)-(A_1^*f,k).
\]
Therefore $\{f',A_1^*f\} \in A_2^*$, and $\{f,f'\} \in A_2^*A_1^*$.
This shows \eqref{feb1++}.
\end{proof}

Now let $A_1$ and $A_2$ be relations in the Hilbert spaces $\sH_1$
and $\sH_2$, respectively. The notation $A_1 \hoplus A_2$ stands
for the \textit{componentwise orthogonal sum}\index{sum of
relations!componentwise!orthogonal $\hoplus$} of $A_1$ and $A_2$
in $(\sH_1 \oplus \sH_2)\times(\sH_1 \oplus \sH_2)$:
\begin{equation*}\label{jan3}
 A_1 \hoplus A_2 \okr\{\,\{f_1  \oplus f_2,f'_1 \oplus f'_2\}\,;\; \\
  \{f_1,f'_1\} \in A_1, \, \{f_2,f'_2\} \in A_2\}.
\end{equation*}
Hence $(A_1 \hoplus A_2)^*=A_1^* \hoplus A_2^*$,
where the adjoints are taken in the corresponding Hilbert spaces.

It follows from the definition \eqref{jan0}  that for any $R \in \boldsymbol{B}(\sH)$
the product  $AR$ is given by
\[
AR=\{\,\{f,f'\} ;\, \{Rf,f'\} \in A\,\}.
\]
This product can be made more explicit if $R$ or $I-R$ is an
orthogonal projection onto a closed subspace containing $\cdom A$.

\begin{lemma}\label{ap}
Let $A$ be a relation in a Hilbert space $\sH$ and let
$\sX$ and $\sY$ be closed subspaces of $\sH$ such that
 $\mul A^*=\sX \oplus \sY$
and let $R$ be the orthogonal projection onto $\cdom A \oplus \sX$.
Then
\[
 AR=A \hoplus (\sY  \times \{0\}), \quad A(I-R)=(\cdom A \oplus \sX) \times \mul A.
\]
In particular,
\[
 \dom AR=\dom A \oplus \sY, \quad \dom A(I-R)=\cdom A \oplus \sX.
\]
\end{lemma}

\begin{proof}
Since $\dom A\subset \ran R$ the definition of the product $AR$
shows that $A \subset AR$ and since $\sY=\ker R\subset \ker AR$ it
is also clear that $\sY \times \{0\} \subset AR$. Hence
\[
 A \hoplus (\sY \times \{0\}) \subset AR.
\]
For the converse inclusion, let $\{f,f'\} \in AR$. Then
\[
 \{f,f'\}=\{Rf,f'\} +\{(I-R)f,0\} \in A \hoplus (\sY \times \{0\}.
\]
This shows the first identity.

On the other hand, $\ran(I-R)\cap \dom A=\sY\cap \dom A=\{0\}$.
Hence, the definition of the product gives $A(I-R)=\ker(I-R) \times
\mul A$, which yields the second identity.
\end{proof}

\subsection{Some auxiliary results}

Let $\sH$ be a Hilbert space and let $\sM$ and $\sN$ be closed
subspaces of $\sH$. Then $\sM+\sN$ is closed if and only if
$\sM^\perp+\sN^\perp$ is closed; see, for instance,
\cite[IV,Theorem 4.8]{Ka}.

\begin{lemma}\label{krakow}
Let $A$ and $B$ be closed linear relations in a Hilbert space $\sH$.
Then the following statements are equivalent:
\begin{enumerate}\def\labelenumi{\rm (\roman{enumi})}
\item $A\hplus B$ is   closed;
\item $A^*\hplus B^*$ is closed.
\end{enumerate}
\end{lemma}

\begin{proof}
(i) $\Longrightarrow$ (ii) The graphs of $A$ and $B$ are closed
linear subspaces of the Hilbert space $\sH\times \sH$. Hence, the
sum $A\hplus B$ is a closed linear subspace of $\sH\times\sH$ if
and only if the sum of the orthogonal complements
\begin{equation}
\label{csum}
 A^{\perp} \hplus B^{\perp}
\end{equation}
in $\sH \oplus \sH$ is also closed.
Recall that the adjoints of $A$ and $B$ are given by
$A^*=JA^\perp$ and $B^*=JB^\perp$, where the operator $J$ is
defined in \eqref{J}. Hence the sum in \eqref{csum} is closed in
$\sH\times\sH$ if and only if
\[
 J(A^{\perp} \hplus B^{\perp})
 = JA^{\perp} \hplus JB^{\perp}=A^* \hplus B^*
\]
is closed in $\sH\times\sH$.

(ii) $\Longrightarrow$ (i) Since $A$ and $B$ are closed one has
that $A^{**}=A$ and $B^{**}=B$. Hence this implication follows by
symmetry.
\end{proof}

The following observation, based on Lemma \ref{krakow}, goes back
to Yu.L. Shmul'jan \cite{Sm}. A weaker version for so-called range
space relations can be found in \cite{LSSW}.

\begin{theorem}\label{NEW}
Let $A$ be a closed relation in a Hilbert space $\sH$. Then
\begin{enumerate}\def\labelenumi{\rm (\roman{enumi})}
\item $\dom A$ closed $\Longleftrightarrow$ $\dom A^*$ closed;

\item $\ran A$ closed $\Longleftrightarrow$ $\ran A^*$ closed.
\end{enumerate}
\end{theorem}

\begin{proof}
(i) First observe that $A=A^{**}$, since  $A$ is assumed to be
closed. Hence,
\begin{equation}\label{wtT*00}
 \left(A^* \hplus (\{0\} \times \sH )\right)^*
 =A \cap (\{0\} \times \sH)
 =\{0\}\times  \mul A.
\end{equation}
In particular, \eqref{wtT*00} leads to
\begin{equation}\label{wtT*0}
 \left(A^* \hplus (\{0\} \times \sH
)\right)^{**}=(\mul A)^\perp \times \sH.
\end{equation}

Assume that $\dom A$ is closed, so that $A \hplus (\{0\}\times \sH)$
is a closed subspace in $\sH \times \sH$. By Lemma~\ref{krakow} this
implies that $A^* \hplus  (\{0\} \times \sH)$ is a closed subspace
of $\sH\times\sH$, so that with \eqref{wtT*0} it follows that
\begin{equation}\label{wtT*}
  A^* \hplus  (\{0\} \times
  \sH) =(\mul A)^\perp \times \sH,
\end{equation}
or, equivalently, $\dom A^*  =(\mul A)^\perp$. Hence, $\dom A^*$ is
closed.

Now assume that $\dom A^*$ is closed, so that $A^* \hplus  (\{0\}
\times   \sH)$ is closed. By Lemma \ref{krakow}  this implies that
$A \hplus (\{0\}\times \sH)$ is closed, i.e., $\dom A$ is closed.

(ii) This can be seen by going over to the inverse of $A$.
\end{proof}

The next proposition augments the previous theorem by giving
necessary and sufficient conditions for
$\dom A^*$ and $\ran A^*$ to be closed, respectively.

\begin{proposition}\label{Dclosed}
Let $A$ be a relation in a Hilbert space $\sH$. Then the following
statements are equivalent:
\begin{enumerate}\def\labelenumi{\rm (\roman{enumi})}
\item $\dom A^*$ is closed;
\item $\ran PA^{**}\subset \dom A^*$, where
$P$ is the orthogonal projection onto $\cdom A^*$;
\item $\ran QA^{*}\subset \dom A^{**}$,
where $Q$ is the orthogonal projection onto $\cdom A$.
\end{enumerate}
Similarly the following
statements are equivalent:
\begin{enumerate}\def\labelenumi{\rm (\roman{enumi})}
\setcounter{enumi}{3}
\item $\ran A^*$ is closed;
\item $P'(\dom A^{**})\subset \ran A^*$,
where $P'$ is the orthogonal projection onto $\cran A^*$;
\item $Q'(\dom A^{*})\subset \ran A^{**}$,
where $Q'$ is the orthogonal projection onto $\cran A$.
\end{enumerate}
\end{proposition}

\begin{proof}
By Lemma~\eqref{domranlemma} $A$ satisfies the identities
\eqref{eqq1}. The implications (ii) $\Longrightarrow$ (i) and (v)
$\Longrightarrow$ (iv) are obtained by applying $P$ to the second
identity in \eqref{eqq1} and $P'$ to the first identity in
\eqref{eqq1}. The implications (iii) $\Longrightarrow$ (i) and
(vi) $\Longrightarrow$ (iv) follow by first applying $Q$ to the
first identity in \eqref{eqq1} and $Q'$ to the second identity in
\eqref{eqq1} to see that $\dom A^{**}$ and $\ran A^{**}$,
respectively, are closed; then apply Theorem \ref{NEW}.

The implications (i) $\Longrightarrow$ (ii) and (iv)
$\Longrightarrow$ (v) are clear, while the implications (i)
$\Longrightarrow$ (iii) and (iv) $\Longrightarrow$ (vi) follow
from Theorem~\ref{NEW}, because then equivalently $\dom A^{**}$
($\ran A^{**}$, respectively) is closed.
\end{proof}

Observe, that in Proposition~\ref{Dclosed} the statements (iv)--(iv)
are actually connected to statements (iv)--(iv) via the formal
inverse $A^{-1}$ of $A$.

The descriptions of $\dom A^*$ and $\ran A^*$ can be given
by means of certain functionals;
cf. Example \ref{adjex} (see  \cite{HSeS} for further details).

The following result (cf. \cite[Lemma~4.1]{HSS??}) follows easily
from Proposition \ref{Dclosed}.

\begin{corollary}\label{boundcor}
Let $A$ be a relation in a Hilbert space $\sH$. Then the following
statements are equivalent:
\begin{enumerate}\def\labelenumi{\rm (\roman{enumi})}
\item $\dom A^*=\sH$;
\item $\ran A^{**}\subset\dom A^*$;
\item $A$ (and thus also $A^{**}$) is the graph of a bounded operator.
\end{enumerate}
\end{corollary}

\begin{proof}
The equivalence of (i) and (ii) is obtained directly from Proposition \ref{Dclosed}.

(i) $\Longrightarrow$ (iii) If $\dom A^*=\sH$, then $\dom A^{**}$
is closed by Theorem \ref{NEW} and $\mul A^{**}=\{0\}$. Now apply
the closed graph theorem.

(iii) $\Longrightarrow$ (i) The boundedness of $A^{**}$ implies
that $\dom A^{**}$ is closed; see Lemma \ref{first}. Hence, also
$\dom A^*$ is closed by Theorem \ref{NEW}. It follows from $\mul
A^{**}=\{0\}$ that $\dom A^*$ is dense. Therefore $\dom A^*=\sH$.
\end{proof}

   \begin{remark} \label{t54}
   Note that the decomposition $A=B+C$ in Example 1.2 with a
nontrivial singular part $B$ is possible if and only if $\dom A^*
\neq \sH$; according to Corollary 2.13 this is equivalent to the
operator $A$ in Example 1.2 being unbounded.

   \end{remark}

There are similar corollaries characterizing $A^{-1}$, $A^*$,
or $A^{-*}$ to be a bounded
(single-valued) operator. It is also noted that $PA^{**}$
appearing in Proposition \ref{Dclosed}
is in fact the regular part of the closure $A$; see Section \ref{sec3}.
The connection between Proposition \ref{Dclosed}
and Corollary \ref{boundcor} can strengthened by
means of decompositions
in Section \ref{sec3}.

\subsection{Points of regular type and the resolvent set}

Let $A$ be a relation in a Hilbert space $\sH$. Then $\lambda \in
\dC$ is said to be an
\textit{eigenvalue}\index{relation!eigenvalue} of $A$ if
$\{f,\lambda f\} \in A$ for some nonzero $f \in \sH$. The set of
points of \textit{regular type}\index{relation!point of regular
type} of $A$ is denoted by $\gamma(A)$; it consists of those
$\lambda \in \dC$ for which there exists a positive constant
$c(\lambda)>0$ such that
\begin{equation}
\label{poslb}
 \|f'-\lambda f\|\ge c(\lambda) \|f\|, \quad \{f,f'\}\in A.
\end{equation}
In other words, $\lambda \in \dC$ is a point of \textit{regular
type} of $A$ if and only if $(A-\lambda)^{-1}$ is (the graph of) a
bounded linear operator, defined on $\ran (A-\lambda)$. In
particular, the relation $A$ is closed if and only if $\ran
(A-\lambda)$ is closed in $\sH$ for some $\lambda \in \dC$ of
regular type. Furthermore, $\gamma(\clos{A})=\gamma(A)$. It is clear
that $\gamma(A) \subset \gamma(\clos A)$. To see the other
inclusion, let $\lambda \in \gamma(A)$, so that $(A-\lambda)^{-1}$
is a bounded linear operator. From $\clos (A-\lambda)^{-1}=(\clos
A-\lambda)^{-1}$ it follows that $\lambda \in \gamma(\clos A)$.

It is well known that $\gamma(A)$ is an open set for operators, this remains true
also for relations; see \cite{Tay1}, \cite{Tay3}, cf. also \cite{DS1,DS2}.

\begin{theorem}\label{gamma}
Let $A$ be a relation in a Hilbert space $\sH$. Then $\gamma(A)$
is an open set. In particular, if $\mu \in \gamma(A)$ and
$|\lambda-\mu|\|(A-\mu)^{-1}\|<1$, then $\lambda \in \gamma(A)$
and
\begin{equation}
\label{est} \|(A-\lambda)^{-1}\| \leq \frac{\|(A-\mu)^{-1}\|}
{1-|\lambda-\mu|\, \|(A-\mu)^{-1}\|}.
\end{equation}
Moreover, if $\mu \in \gamma(A)$ and
$|\lambda-\mu|\|(A-\mu)^{-1}\|<1$, then $\cran (A-\lambda)$ is not a
proper subset of $\cran (A-\mu)$.
\end{theorem}

\begin{proof}
Let $\mu \in \gamma(A)$ and $\{f,g\} \in A$. Since $(A-\mu)^{-1}$
is a bounded linear operator, it follows from $(A-\mu)^{-1}(g-\mu
f)=f$ that
\[
  \|f\| \leq \|(A-\mu)^{-1}\|\,\|(g-\mu f)\|.
\]
For each $\lambda \in \dC$ one has
\[
   g-\lambda f=g-\mu f-(\lambda-\mu)f,
\]
which implies that
\[
  \|g-\lambda f\| \ge \|g-\mu f\|-|\lambda-\mu| \|f\|.
\]
Hence,
\[
\begin{split}
  \|(A-\mu)^{-1}\|\,\|g-\lambda f\| &\ge \|(A-\mu)^{-1}\|\,\|g-\mu
f\|-|\lambda-\mu|\,\|(A-\mu)^{-1}\|\, \|f\| \\
 & \ge \|f\|-|\lambda-\mu|\,\|(A-\mu)^{-1}\|\, \|f\| \\ &
 =(I-|\lambda-\mu|\,\|(A-\mu)^{-1}\|)\, \|f\| .
\end{split}
\]
With the inclusion $\{g-\lambda f,f\} \in (A-\lambda)^{-1}$ and the assumption
$|\lambda-\mu|\|(A-\mu)^{-1}\|<1$ this
inequality shows that $(A-\lambda)^{-1}$ is a bounded linear
operator, whose norm is estimated by \eqref{est}.

Assume that $\cran (A-\lambda)$ is a proper subset of $\cran
(A-\mu)$. Choose $k\in\cran (A-\mu)\ominus\cran (A-\lambda)$ with $\|k\|=1$.
Then $\|k-g\| \ge 1$ for all $g \in \cran (A-\lambda)$. Let
$k_n \in \ran (A-\mu)$ such that $k_n \to k$. Then there exist
$h_n$ such that $\{h_n,k_n\} \in A-\mu$, so that also
$\{h_n,k_n+(\mu-\lambda)h_n\} \in A-\lambda$. In particular
\[
\begin{split}
      1 &\le \|k-(k_n+(\mu-\lambda)h_n)\| \\
             &\le \|k-k_n\|+|\mu-\lambda|\,\|h_n\| \\
             &\le
             \|k-k_n\|+|\mu-\lambda|\,\|(A-\mu)^{-1}\|\|k_n\|.
\end{split}
\]
Letting $n \to \infty$ leads to
\[
 1 \le |\mu-\lambda|\,\|(A-\mu)^{-1}\|,
\]
a contradiction. Hence $\cran (A-\lambda)$ is not a proper subset
of $\cran (A-\mu)$.
\end{proof}

The \textit{resolvent set} $\rho(A)$ \index{resolvent set
$\rho(A)$ of $A$}of $A$ is the set of all $\lambda \in \dC$ such
that $\lambda\in\gamma(A)$ and for which $\ran (A-\lambda)$ is
dense in $\sH$. Observe that $\rho(\clos{A})=\rho(A)$.

\begin{theorem}
Let $A$ be a relation in a Hilbert space $\sH$. Then $\rho(A)$ is
open. In particular, if $\mu \in \rho(A)$ and
$|\lambda-\mu|\|(A-\mu)^{-1}\|<1$, then $\lambda \in \rho(A)$.
\end{theorem}

\begin{proof}
Since $\mu \in \rho(A)$ one has $\mu \in \gamma(A)$ and $\cran
(A-\mu)=\sH$. Now by Theorem~\ref{gamma} $\lambda \in \gamma(A)$
and $\cran(A-\lambda)$
is not a proper subset of $\cran(A-\mu)$. Therefore
$\cran(A-\lambda)=\sH$, so that $\lambda \in \rho(A)$.
\end{proof}

If $A$ is closed, then $\lambda \in
\rho(A)$ if and only if $(A-\lambda)^{-1} \in \boldsymbol{B}(\sH)$.
Observe that $\rho(\clos{A})=\rho(A)$.

\subsection{Defect numbers}

It is useful to recall the notion of opening between subspaces.
Let $\sL_1$ and $\sL_2$ be linear (not necessarily closed)
subspaces of a Hilbert space $\sH$. Let $P_1$ and $P_2$ be the
orthogonal projections onto the closures $\overline{\sL_1}$ and
$\overline{\sL_2}$ of $\sL_1$ and $\sL_2$, respectively. The
\textit{opening}\index{opening of (two) subspaces}
$\theta(\sL_1,\sL_2)$ is defined by
$\theta(\sL_1,\sL_2)=\|P_1-P_2\|$. It is clear that
$\theta(\sL_1,\sL_2)=\theta(\overline{\sL_1},\overline{\sL_2})
  =\theta(\sL_1^{\perp},\sL_2^{\perp})$. Moreover, $\theta(\sL_1,\sL_2) \leq 1$,
and if $\theta(\sL_1,\sL_2)<1$, then  $\dim \sL_1=\dim \sL_2$. In
order to use the opening the following formula is useful:
\[
  \theta(\sL_1,\sL_2)=\max \left( \sup_{f \in \sL_1}
  \frac{\|(I-P_2)f\|}{\|f\|}, \sup_{f \in \sL_2}
  \frac{\|(I-P_1)f\|}{\|f\|} \right).
\]
The following result is a standard fact for operators, for relations it
appears precisely in the same form.

\begin{theorem}\label{dimen}
Let $A$ be a relation in a Hilbert space $\sH$. Then the defect
\index{relation!defect}
\[
\dim \ran(A-\lambda)^{\perp}
\]
is constant for $\lambda$ in connected components of $\gamma(A)$.
\end{theorem}

\begin{proof}
Let $\lambda, \mu \in \gamma(A)$ and let $P_\lambda$ and $P_{\mu}$
be the orthogonal projections onto the subspaces $\ran
(A-\lambda)^{\perp}$ and $\ran (A-\mu)^{\perp}$, respectively.

{\em Step 1}. For each $h \in \sH$
\[
 \|(I-P_{\mu})h\|= \sup_{\{f,g\} \in A} \frac{|(h, g-\mu
 f)|}{\|g-\mu f\|} =\sup_{\{f,g\} \in A} \frac{|(h, g-\lambda
 f+(\lambda-\mu)f)|} {\|g-\mu f\|}.
\]
In particular, if $h \in \ran (A-\lambda)^{\perp}$, then
\[
 \|(I-P_{\mu})h\|= |\lambda-\mu|\sup_{\{f,g\} \in A}
 \frac{|(h,f)|}{\|g-\mu f\|}.
\]
Since $\|f\| \leq \|(A-\mu)^{-1}\|\, \|g-\mu f\|$, $ \{f,g\} \in
A$, it follows that
\[
 \|(I-P_{\mu})h\| \leq |\lambda-\mu| \, \|(A-\mu)^{-1}\| \, \|h\|.
\]

{\em Step 2}. Completely similar, it follows for $k \in \ran
(A-\mu)^{\perp}$ that
\[
 \|(I-P_{\lambda})k\|= |\lambda-\mu| \sup_{\{f,g\} \in A}
 \frac{|(k,f)|}{\|g -\lambda f\|} \le |\lambda-\mu| \,
 \|(A-\lambda)^{-1}\| \, \|k\|.
\]
Hence, if $|\lambda-\mu|\,\|(A-\mu)^{-1}\|<1$, then
\[
 \|(I-P_{\lambda})k\| \leq \frac{|\lambda-\mu|\, \|(A-\mu)^{-1}\|}
 {1-|\lambda-\mu|\, \|(A-\mu)^{-1}\|} \, \|k\|.
\]

{\em Step 3}. Now let $|\lambda-\mu|\,\|(A-\mu)^{-1}\|<\half$.
Then it follows from Steps 1 and 2 that
\[
 \theta(\ran (A-\lambda)^{\perp}, \ran (A-\mu)^{\perp} ) < 1,
\]
which implies the equality
\[
  \dim \ran (A-\lambda)^{\perp}=\dim \ran (A-\mu)^{\perp},
\]
see \cite{AG}, \cite{Ka}.

{\em Step 4}. For each $\mu \in \gamma(A)$ there exists a positive
number $\delta=\half \|(A-\mu)^{-1}\|^{-1}$, such that $|\lambda -
\mu|< \delta$ implies that $\lambda \in \gamma(A)$ and that at
$\lambda$ there is the same defect as at $\mu$. Now let $\Gamma$
be a connected open component of $\gamma(A)$. Then $\Gamma$ is
arcwise connected and each pair of points in $\Gamma$ can be
connected by a (piecewise) connected curve with compact image.
It remains to use compactness to divide the curve into pieces of length
$\delta/2$ to conclude that $\dim \ker (A^*-\bar{\lambda})$ is constant
in $\Gamma$.
\end{proof}

\subsection{The numerical range}

Let $A$ be a relation in a Hilbert space $\sH$. The
\textit{numerical range}\index{relation!numerical range $\cW (A)$}
$\cW (A)$ of $A$ is defined by
\[
\cW (A)=\{\, (f',f) ;\, \{f,f'\} \in A, \ \|f\|=1 \, \} \subset \dC,
\]
and by $\{ 0 \} \subset \dC$ if $A$ is purely multivalued, i.e. if
$\dom A = \{0\}$. Clearly, all eigenvalues of $A$ belong to the
numerical range $\cW(A)$ of $A$. Observe, that numerical range of the inverse of
$A$ is given by
\[
\cW (A^{-1})=\overline{\cW (A)}=\{\,\lambda\in \dC;\, \overline{\lambda}\in \cW (A)\,\}.
\]
The following result will be
proved along the lines of \cite{Stone}; cf. \cite{Ka}, \cite{Sch}.

\begin{proposition}\label{convex}
Let $A$ be a relation in a Hilbert space $\sH$. Then the numerical
range $\cW(A)$ is a convex set in $\dC$.
\end{proposition}

\begin{proof}
Let $\lambda_1, \lambda_2 \in \cW(A)$ and assume that $\lambda_1
\neq \lambda_2$.
It will be shown that
each point on the segment between $\lambda_1$ and $\lambda_2$
belongs to $\cW(A)$, i.e., it will be shown that for each $u \in
[0,1]$
\[
  u\lambda_1 +(1-u) \lambda_2 \in \cW(A).
\]
For this purpose write $\lambda_i=(g_i,f_i)$, where $\{f_i,g_i\}
\in A$, $\|f_i\|=1$, $i=1,2$, and define for $x_1,x_2 \in \dC$:
\[
F(x_1,x_2)=(x_1 g_1+x_2 g_2,x_1f_1+x_2f_2), \quad
G(x_1,x_2)=\|x_1f_1+x_2f_2\|^2,
\]
and
\[
H(x_1,x_2)=\frac{F(x_1,x_2)-\lambda_2 G(x_1,x_2)} {\lambda_1 -
\lambda_2}.
\]
Note that if $G(x_1,x_2)=1$, then $F(x_1,x_2) \in \cW(A)$, or, in
other words,
\[
 H(x_1,x_2) \lambda_1 +(1-H(x_1,x_2)) \lambda_2 =
 \lambda_2+H(x_1,x_2)(\lambda_1 - \lambda_2) \in \cW(A).
\]
Hence, the proof will be complete if for each $u \in [0,1]$ there
exist numbers $x_1, x_2 \in \dC$ for which
\[
 G(x_1,x_2)=1, \quad H(x_1,x_2)=u.
\]
Observe that $H(x_1,x_2)=x_1\bar{x}_1+c_1\bar{x}_1
x_2+c_2x_1\bar{x}_2$ for some $c_1,c_2 \in \dC$. Define
\[
  \delta=1 \mbox{ if } \bar{c}_1=c_2, \ \
  \delta=\frac{\bar{c}_1-c_2}{|\bar{c}_1-c_2|} \mbox{ if }
  \bar{c}_1 \neq c_2
\]
so that $|\delta|=1$. When $t_1,t_2 \in \dR$ it follows that
\[
  G(t_1,\delta t_2)=t_1^2+2 \beta t_1t_2 + t_2^2, \quad
  H(t_1,\delta t_2)=t_1^2+ \gamma t_1t_2,
\]
where $\beta=\RE (\delta(f_2,f_1))$ and $\gamma=\delta
c_1+\bar{\delta} c_2$. Hence $-1 \le \beta \leq 1$ and $\gamma \in
\dR$. For $t_1 \in [-1,1]$ note that $(1-\beta^2)t_1^2 \le 1$ and
choose
\[
  t_2=-\beta t_1 \pm \sqrt{1-(1-\beta^2)t_1^2},
\]
with the $+$ sign when $\beta \geq 0$ and the $-$ sign when $\beta
<0$. Then
\[
  G\left(t_1,\delta\left(-\beta t_1 \pm
  \sqrt{1-(1-\beta^2)t_1^2}\right)\right)=1,
\]
and
\[
  H\left(t_1,\delta\left(-\beta t_1 \pm
  \sqrt{1-(1-\beta^2)t_1^2}\right)\right)= (1-\beta \gamma)t_1^2
  \pm \gamma t_1 \sqrt{1-(1-\beta^2)t_1^2}.
\]
The last expression is a real continuous function in $t_1$ which
takes the value $0$ at $t_1=0$ and the value $1$ at $t_1=1$. Hence
the segment $[0,1]$ is in the range of values of this function.
\end{proof}

Hence either $\cW(A)=\dC$ or $\cW(A) \ne \dC$, in which case
$\cW(A)$ lies in some halfplane. The first case may actually occur
if, for instance, $\ker A\cap\mul A\neq \{0\}$, so that $A$
contains nontrivial elements $\{0,h\}$ and $\{h,0\}$.
If the relation $A'$ is an extension of $A$, i.e.,
$A \subset A'$, then $\cW(A) \subset\cW(A')$. In particular,
\begin{equation}\label{funda}
\cW(A) \subset \cW(\clos A) \subset \clos \cW (A),
\end{equation}
where the last inclusion is straightforward to verify. All sets in
\eqref{funda} are convex.

\subsection{An extension preserving the numerical range}

Let $A$ be a relation in a Hilbert space $\sH$ and associate with
it the relation $A_\infty$ \index{relation!$A_\infty$}defined by
\begin{equation}\label{+SF}
 A_\infty\okr A \hplus (\{0\} \times \mul A^*);
\end{equation}
the sum in \eqref{+SF} is direct if and only if $\mul A \cap \mul
A^*=\{0\}$. The relation $A_\infty$ is an extension of $A$ and
\begin{equation} \label{++ssf0}
\dom A_\infty=\dom A, \quad \mul A_\infty=\mul A+\mul A^*.
\end{equation}
Clearly, if $\mul A \subset \mul A^*$ then $\mul A_\infty= \mul
A^*$. Moreover, $A_\infty=A$ if and only $\mul A^* \subset \mul A$
(which is the case when, for instance, $A$ is densely defined).
Due to $\mul A^*=(\dom A)^\perp$ it follows from \eqref{+SF} that
\begin{equation}\label{funda++}
 \cW(A_\infty) = \cW(A).
\end{equation}
Constructions in terms of the extension $A_\infty$ can be found in
\cite{CS78} and \cite{HSSz??a}. A key observation is given in the following lemma.

\begin{lemma}\label{queen}
Let $A$ be a relation in a Hilbert space $\sH$. Then
$(A_\infty)^*$ \index{relation!$A_\infty$}can be expressed as a
restriction of $A^*$:
\begin{equation}\label{+ssf}
(A_\infty)^*=\{\,\{f,f'\} \in A^* \,;\; f \in \cdom A \,\}.
\end{equation}
In particular
\begin{equation} \label{++ssf}
\dom (A_\infty)^* =\cdom A \cap \dom A^*, \quad  \mul
(A_\infty)^*=\mul A^*.
\end{equation}
\end{lemma}

\begin{proof}
It follows from \eqref{+SF} and Lemma~\ref{basiclem} that
\[
 (A_\infty)^*= A^* \cap (\cdom A \times \sH),
\]
which leads to the description \eqref{++ssf} and the identities in \eqref{++ssf}.
\end{proof}

\subsection{Formally domain tight and domain tight relations}

A relation $A$ in a Hilbert space $\sH$ is said to be
\textit{formally domain tight}\index{relation!formally domain
tight} if
\begin{equation}\label{eq0}
 \dom A \subset \dom A^*.
\end{equation}
Formally normal and symmetric relations are formally domain tight.
If a relation $A$ is formally domain tight, then \eqref{eq0} shows
that
\begin{equation}\label{eq1+}
(\mul A \subset )\,\, \mul A^{**} \subset \mul A^*.
\end{equation}
A densely defined formally domain tight relation $A$ is (the graph
of) a closable operator, i.e., $\mul A^{**}=\{0\}$. Furthermore,
for a formally domain tight relation $A$ it follows that
\begin{equation}\label{denhaag}
 \mul A^* \subset \mul A \quad \Longrightarrow \quad \mul A =\mul
 A^*=\mul A^{**}.
\end{equation}
A relation $A$ in a Hilbert space $\sH$ is said to be
\textit{domain tight}\index{relation!domain tight} if
\begin{equation}\label{eq}
 \dom A=\dom A^*.
\end{equation}
Normal and selfadjoint relations are domain tight. If a relation
$A$ is domain tight, then
\begin{equation}\label{eq1++}
 \mul A^{**} = \mul A^*.
\end{equation}
A domain tight relation $A$ is densely defined if and only if $A$ is
(the graph of) a closable operator, i.e., $\mul A^{**}=\{0\}$.

\begin{remark}
The notions of formally domain tight and domain tight relations seem
to be new. It is clear that symmetric, formally normal, and
subnormal relations may be viewed as prototypes of formally domain
tight relations and that selfadjoint and normal relations may be
viewed as prototypes of domain tight relations. Densely defined
domain tight symmetric or formally normal operators must necessarily
be selfadjoint or normal, respectively;  on the other hand, domain
tight symmetric or domain tight formally normal relations are
selfadjoint or normal when extra information about the multivalued
parts is provided; cf. Corollary \ref{sysa}.
For subnormal operators the situation is different: in principle
they are not domain tight (see \cite{ass} for some discussion) but
even if they are, they may not be normal as their normal extensions
in most cases go beyond the initial space; this is less visible in
the case of relations and Section \ref{uhaha} sheds some more light
on that.
\end{remark}

\begin{remark}
Further examples of formally domain tight and domain tight relations
or operators come  from the $q$-deformation of the above mentioned
classes. This is motivated by the theory of quantum groups; the
relevant Hilbert space operators were introduced
by S. \^Ota \cite{ota1}, \cite{ota2}.
The balanced operators proposed by S.L. Woronowicz \cite{Wo} appear
to be in the same spirit.
\end{remark}

\begin{lemma}\label{haag}
Let $A$ be a relation in a Hilbert space $\sH$. Then
\begin{enumerate}\def\labelenumi{\rm (\roman{enumi})}
\item
$A$ is formally domain tight if and only if
\begin{equation}\label{denhaag1}
 \dom A \subset \cdom A \cap \dom A^*;
\end{equation}

\item
if $A$ is domain tight then
\begin{equation}\label{denhaag2}
 \dom A = \cdom A \cap \dom A^*.
\end{equation}
\end{enumerate}
\end{lemma}

\begin{proof}
(i) The inclusion $\dom A \subset \dom A^*$ is equivalent to the
inclusion in \eqref{denhaag1}.

(ii) If $A$ is formally domain tight, then \eqref{denhaag1} gives
\[
 \dom A \subset \cdom A \cap \dom A^* \subset \dom A^*.
\]
Hence, if $A$ is domain tight, then \eqref{denhaag2} follows.
\end{proof}

If $B$ is a formally domain tight relation, then any restriction
$A$ of $B$, i.e., $A \subset B$, is also formally domain tight; see \eqref{AB*}.
The following lemma contains a kind of converse statement.

\begin{lemma}\label{ciekawe}
Let $A$ and $B$ be relations in a Hilbert space $\sH$ which
satisfy $A \subset B$. If $A$ is domain tight and $B$ is formally
domain tight, then $ B$ is domain tight.
\end{lemma}

\begin{proof}
The inclusion $A \subset B$ implies that $B^* \subset A^*$.
Therefore, it follows that
\begin{equation} \label{nowy}
\dom A \subset \dom B, \quad \dom B^* \subset \dom A^*.
\end{equation}
The assumptions on $A$ and $B$ are
\[
\dom A=\dom A^*, \quad \dom B \subset \dom B^*.
\]
Combining these assumptions with the inclusions in \eqref{nowy}
gives
\[
\dom A \subset \dom B \subset \dom B^* \subset \dom A^*=\dom A,
\]
which leads to $\dom B=\dom B^*$, i.e., $B$ is domain tight.
\end{proof}

\begin{remark}\label{tight}
Let $A$ be a relation in a Hilbert space. Then clearly
\[
A \quad \mbox{domain tight} \quad \Longrightarrow \quad A \quad
\mbox{and} \quad A^* \quad \mbox{formally domain tight}.
\]
Moreover, if $\dom A^{**}=\dom A$, then
\[
  A \quad \mbox{and} \quad A^* \quad \mbox{formally domain tight}
  \quad \Longrightarrow \quad A \quad \mbox{domain tight}.
\]
If $A^{**}$ is formally domain tight, then $A$ is formally domain
tight.
\end{remark}

The relation $A_\infty$ introduced in \eqref{+SF} can be used to obtain
a characterization for $A$ to be domain tight or formally domain tight.

\begin{proposition}\label{codi}
Let $A$ be a relation in a Hilbert space $\sH$ and let the extension
$A_\infty$ of $A$ be defined by \eqref{+SF}. Then
\begin{enumerate}\def\labelenumi{\rm (\roman{enumi})}
\item
$A$ is formally domain tight if and only if
$A_\infty$ is formally domain tight;

\item
$A_\infty$ is domain tight if and only if\/ $\dom
A=\cdom A \cap \dom A^*$;

\item
$A$ is domain tight if and only if $A_\infty$ is domain
tight and $\dom A^*\subset \cdom A$. Furthermore, in this case
$(A_\infty)^*=A^*=(A^*)_\infty$.
\end{enumerate}
\end{proposition}

\begin{proof}
(i) According to \eqref{++ssf0} and \eqref{++ssf} the relation
$A_\infty$ is formally domain tight (i.e., $\dom A_\infty \subset
\dom (A_\infty)^*$) if and only if
\begin{equation*} \label{+++ssf}
   \dom A \subset \cdom A \cap \dom A^*.
\end{equation*}
Hence, the statement follows from Lemma \ref{haag}.

(ii) The assertion follows from \eqref{++ssf0} and \eqref{++ssf}.

(iii) Let $A$ be domain tight. Then $\dom A=\cdom A \cap \dom A^*$
by Lemma \ref{haag}. Hence, $A_\infty$ is domain tight
by (ii). Moreover, $\dom A^*=\dom A\subset \cdom A$.

Conversely, if $A_\infty$ is domain tight and $\dom A^*\subset \cdom A$,
then part (ii) implies that $\dom A=\cdom A \cap \dom A^*=\dom A^*$.
Thus, $A$ is domain tight.

It is clear for a domain tight relation $A$ that
\[
 \{\,\{f,f'\} \in A^* \,;\; f \in \cdom A \,\}=A^*.
\]
Hence, the identity  \eqref{+ssf} implies that $(A_\infty)^*=A^*$.
In general, $(A^*)_\infty$ is an extension of $A^*$, and
$A^*=(A^*)_\infty$ if and only if $\mul A^{**} \subset \mul A^*$.
Therefore, if $A$ is domain tight, the identity \eqref{eq1++}
implies that $A^*=(A^*)_\infty$.
\end{proof}

\begin{lemma}\label{marci}
Let $A$ be a relation in a Hilbert space $\sH$ and let the extension
$A_\infty$ of $A$ be defined by \eqref{+SF}. If $A$ is formally
domain tight, then
\begin{enumerate}\def\labelenumi{\rm (\roman{enumi})}
\item
$\mul A_\infty=\mul A^*$;

\item
$A_\infty=A$ if and only $\mul A^* = \mul A$;

\item
$A \cap (\{0\} \times \mul A^*)=\{0\} \times \mul
A$, and the sum in \eqref{+SF} is direct if and only if $A$ is an
operator;

\item
$A_\infty$ is an operator if and only if $A$ is
densely defined.
\end{enumerate}
Moreover, if $A$ is domain tight and $\mul A^{**}=\mul A$,  then
$A_\infty=A$. In particular, if $A$ is domain tight and closed, then
$A_\infty=A$.
\end{lemma}

\begin{proof}
(i) Since $A$ is formally domain tight, \eqref{eq1+} shows that
$\mul A \subset \mul A^*$. This shows the assertion.

(ii) Note that $A=A_\infty$ if and only if $\mul A^* \subset \mul
A$. If $A$ is formally domain tight, then \eqref{denhaag} implies
that the inclusion  $\mul A^* \subset \mul A$ is equivalent to the
identity $\mul A^*=\mul A$.

(iii) Since $A$ is formally domain tight,  the inclusion $\mul A
\subset \mul A^*$ in \eqref{eq1+} leads to the assertions.

(iv) If $A$ is densely defined, then $\mul A^*=\{0\}$, so that $\mul
A^{**}=\{0\}$ by \eqref{eq1+}, and $A$ is a closable operator.
Hence, $A_\infty$ is an operator. Conversely, if $A_\infty$ is an
operator, then necessarily $\mul A^*=\{0\}$, so that $A$ is densely
defined.

For the last statement, observe that \eqref{eq} implies \eqref{eq1++}.
The assumption $\mul A^{**} = \mul A$ implies that $\mul A^* = \mul
A$. The assertion now follows from (ii).
\end{proof}

\subsection{Selfadjointness of symmetric relations}

Let $A$ be a symmetric relation in a Hilbert space $\sH$. Then its
closure $A^{**}$ is formally domain tight, as the closure is
symmetric. If $A$ is densely defined, then $\mul A^{**}=\{0\}$, so
that, in fact, $A$ is a closable operator.

If $A$ is a selfadjoint relation, then, in particular,  $A$ is
symmetric, domain tight, and $\mul A^* \subset \mul A$.  \marginpar{changed}

\begin{lemma}\label{sysa}
Let $A$ be a symmetric domain tight  relation in a Hilbert space
$\sH$, such that $\mul A^{*} \subset \mul A$. Then $A$ is
selfadjoint. In particular, a closed domain tight symmetric relation
is selfadjoint.
\end{lemma}

\begin{proof}
It suffices to show that $A^* \subset A$.  Let $\{f,g\} \in A^*$, so
that $f \in \dom A^*=\dom A$, which implies that there is an element
$h$ such that $\{f,h\}\in A (\subset A^*)$. Hence, $g-h\in \mul A^*
\subset \mul A$. Therefore
\[
 \{f,g\}=\{f,h\}+\{0,g-h\} \in A.
\]
Hence, $A^*\subset A$, and thus $A$ is selfadjoint.

When $A$ is closed and domain tight,   it follows from  \eqref{eq1++}
that $\mul A^*=\mul A$. Hence, the last observation is clear.
\end{proof}

If $A$ is a symmetric relation, then, clearly, also the extension
$A_\infty$ is symmetric (for instance, see \eqref{funda++}). The
following result goes back to \cite{CS78}.

\begin{lemma}\label{CS}
Let $A$ be a relation in a Hilbert space $\sH$ and let the extension
$A_\infty$ of $A$ be defined by \eqref{+SF}. Then $A_\infty$ is
selfadjoint if and only if $A$ is symmetric and $\dom A=\cdom A \cap
\dom A^*$.
\end{lemma}

\begin{proof}
($\Longrightarrow$) If $A_\infty$ is selfadjoint, then $A$ is
symmetric and $A_\infty$ is domain tight, so that $\dom A=\cdom A
\cap \dom A^*$; cf. Proposition \ref{codi}.

($\Longleftarrow$) If $A$ is symmetric and $\dom A=\cdom A \cap
\dom A^*$, then $A_\infty$ is symmetric and domain tight. Moreover
$\mul (A_\infty)^*=\mul A^*$ by Lemma \ref{queen} and $\mul
A_\infty=\mul A^*$ by Lemma \ref{marci}, so that $A_\infty$ is
selfadjoint; cf. Lemma \ref{sysa}.
\end{proof}

\subsection{Extensions in larger Hilbert spaces}\label{uhaha}

Let $\sH$ and $\sK$ be two Hilbert space with the inclusion
$\sH\subset\sK$ being isometric. Let $A$
be a relation in the Hilbert space $\sH$ and let $B$ be a relation
in the Hilbert space $\sK$. Assume that $B$ is an extension of $A$,
i.e.,
\begin{equation} \label{iwota08}
A\subset B.
\end{equation}
Then it is clear that
\begin{equation} \label{iwota10}
\dom A\subset \dom B \cap \sH, \quad P(\dom B^*) \subset \dom A^*,
\end{equation}
where $P$ is the orthogonal projection of $\sK$ onto $\sH$. The
assumption $A \subset B$ implies the first inclusion in
\eqref{iwota10} trivially and it implies the second inclusion in
\eqref{iwota10} since $\{Pf,Pg\} \in A^*$ for all $\{f,g\} \in
B^*$. The relation $B$ is said to be a {\em tight
extension}\,\index{extension!tight} of $A$ if
\begin{equation*}
\dom A=\dom B\cap\sH,
\end{equation*}
and, likewise, $B$ is said to be a {\em $*$-tight
extension}\index{extension!$*$-tight} of $A$ if
\begin{equation*}
P(\dom B^*)=\dom A^*.
\end{equation*}
Tight and $*$-tight extensions will be discussed only in this
subsection. If the relation $B$ is formally domain tight in $\sK$, then
\begin{equation}\label{mary}
\dom B \cap \sH \subset P(\dom B^*).
\end{equation}
Hence, if $B$ is a tight and $*$-tight extension of $A$, and if $B$
is formally domain tight in $\sK$, then \eqref{mary} shows that $A$
is formally domain tight in $\sH$. The next result is a counterpart
to Lemma \ref{ciekawe}.

\begin{lemma}\label{ciekawe1}
Let $A$ be a relation in the Hilbert space $\sH$ and let $B$ be a
relation in the Hilbert space $\sK$ which satisfy \eqref{iwota08}.
\begin{enumerate}\def\labelenumi{\rm (\roman{enumi})}
\item
If $A$ is domain tight in $\sH$ and  $B$  is
formally domain tight in $\sK$, then
\begin{equation}\label{iw}
\dom B\cap\sH=P(\dom B^*),
\end{equation}
and $B$ is a tight and $*$-tight extension of $A$.

\item
If the identity \eqref{iw} holds and if $B$ is a tight
and $*$-tight extension of $A$, then $A$ is domain tight in $\sH$.
\end{enumerate}
\end{lemma}

\begin{proof}
(i) If the extension $B$ of $A$ is formally domain tight in $\sK$,
then
\begin{equation}\label{william}
\dom A \subset \dom B \cap \sH \subset P (\dom B^*) \subset \dom
A^*.
\end{equation}
The second inclusion follows from
$\dom B \subset \dom B^*$.
The other inclusions follow from \eqref{iwota10}.
The assumption that $A$ is domain tight
in $\sH$ and the inclusions in \eqref{william} imply the identity in
\eqref{iw}. In particular,
$B$ is a tight and $*$-tight extension of $A$.

(ii) Assume that the identity \eqref{iw} holds and that
$B$ is a tight and $*$-tight extension of $A$.
By the definitions of tight and $*$-tight
extensions it follows that $\dom A=\dom A^*$.
\end{proof}

If $B$ is a tight extension of $A$, then any tight extension of $B$
is again a tight extension of $A$. There is a similar statement for $*$-tight
extensions of $A$.

A densely defined symmetric operator always has a tight selfadjoint
extension; a detailed argument is given in \cite{ass}, which in turn
implements the suggestion made in \cite{AG}, where a tight extension
is called an extension of the second kind. A densely defined
subnormal operator need not have any tight normal extensions; an
example of \^Ota \cite{ota} gives a negative answer to the question
in \cite{ass}.

Tight and $*$-tight extensions as discussed in \cite{ccr} are
essential in identifying solutions of the commutation relation of
the $q$-harmonic oscillator as $q$-creation operators when $q>1$, in
which case nonuniqueness of normal extensions occurs, see
\cite[Theorem 21]{qccr}.

\subsection{Range tight relations}

Let $A$ be a relation in a Hilbert space $\sH$. The notions of
formally domain tight and domain tight refer to properties
relative to the domains $\dom A$ and $\dom A^*$. Similar notions
exist relative to the ranges $\ran A$ and $\ran A^*$.
A relation $A$ in a Hilbert space $\sH$ is said to be
\textit{formally range tight}\index{relation!formally range
tight} if
\[
 \ran A \subset \ran A^*,
\]
and it is said to be
\textit{range tight}\index{relation!range tight} if \[
 \ran A=\ran A^*.
\]
Clearly, a relation $A$ is (formally) range tight if and only if the relation $A^{-1}$ is (formally)
domain tight.  Hence, all earlier statements for
(formally) domain tight relations have their counterparts for (formally)
range tight relations. As an example consider the following consequence of
Lemma \ref{sysa}.

Let $A$ be a symmetric range tight relation in a Hilbert space
$\sH$, such that $\ker A^{*} \subset \ker A$. Then $A$ is
selfadjoint. In particular, a closed range tight symmetric
relation is selfadjoint. The same result for densely defined
closed range tight symmetric operators was obtained independently
by Z. Sebestyen and Z. Tarcsay (personal comunication).

\subsection{Maximality with respect to the numerical range}

The following results are included for completeness. In some form or
other they go back to R.~McKelvey (unpublished lecture notes) and
F.S. Rofe-Beketov \cite{RB}; see also \cite{HSSW??}.

\begin{lemma}\label{numRange2}
Let $A$ be a relation in a Hilbert space $\sH$ with $\cW (A) \ne
\dC$. Let $\lambda \notin \clos \cW (A)$, i.e.,  $d(\lambda) = \dist
(\lambda,\clos \cW (A))>0$, then
\begin{enumerate}\def\labelenumi{\rm (\roman{enumi})}

\item $(A-\lambda)^{-1}$ is a bounded linear operator with
\begin{equation}\label{nR-1}
\|(A-\lambda)^{-1}\| \le 1/d(\lambda) ;
\end{equation}

\item $\mul A \subset \mul A^*$.

\end{enumerate}
\end{lemma}

\begin{proof}
(i) Let $\lambda \notin \clos \cW (A)$ and let $\{f,f'\} \in A$ with
$\|f\|=1$. Then
\[
 (f',f)-\lambda=(f',f)-\lambda (f,f)=(f'-\lambda f,f),
\]
so that
\[
 d(\lambda) \le |(f',f)-\lambda| \le \|f'-\lambda f\|,
\quad \{f,f'-\lambda f\} \in  A-\lambda.
\]
Since $\lambda$ is not an eigenvalue of $A$, the inequality in
\eqref{nR-1} follows from the above inequality.

(ii) Let $\varphi\ \in \mul A$, so that $\{f,f'+c \varphi\} \in A$
for all $\{f,f'\}\in A$ and all $c \in \dC$. Since $\cW (A) \ne
\dC$, the identity
\[
 (f'+c \varphi,f)=(f',f)+c(\varphi,f),
\]
shows that $(\varphi,f)=0$. Hence $\mul A \subset (\dom
A)^\perp=\mul A^*$.
\end{proof}

Let $A$ be a relation in a Hilbert space $\sH$ with $\cW (A) \ne
\dC$. According to Lemma \ref{numRange2} the complement
$\Delta(A)=\dC \setminus \clos \cW(A)$ is a subset of the set of
regular points of $A$. Hence $\ran (A-\lambda)$ is closed for some
$\lambda \notin \clos \cW (A)$ if and only if $A$ is closed.
Since $\clos \cW(A)$ is a closed convex set
(see Proposition \ref{convex} and \eqref{funda}),
it follows that $\Delta(A)$ is an open connected set or $\Delta(A)$ consists of two
open connected components (if $\cW(A)$ is a strip bounded by two
parallel straight lines). Furthermore, by Theorem~\ref{dimen}
$\dim \ker (A^*-\bar{\lambda})$ is constant for
$\lambda \in \Delta(A)$ or for $\lambda$ in each of the connected
components of $\Delta(A)$. If $\ker (A^*-\bar{\lambda})=\{0\}$ for
some $\lambda \in \dC \setminus \clos \cW(A)$ then $\Delta(A)$ or
the corresponding component (to which $\lambda$ belongs) is a subset
of $\rho(A)$.

Note that in the statements (i) and (ii) of Lemma \ref{numRange2}
the relation $A$ may be replaced by the closure $A^{**}$.  In
particular, this shows that a densely defined relation $A$ with $\cW
(A) \ne \dC$ satisfies $\mul A^{**}=\{0\}$; in other words, $A$ is a
closable operator. Furthermore, it follows that $\ran (A^{**}
-\lambda)$ is closed. These observations lead to the following
useful result.

\begin{corollary}
Let $A$ be a relation in a Hilbert space $\sH$ with $\cW (A) \ne
\dC$. Let $\lambda \notin \clos \cW (A)$, then
\[
\ran (A^*-\bar{\lambda})=\sH.
\]
\end{corollary}

\begin{proof}
In general $\sH=\cran(A^*-\bar{\lambda}) \oplus \ker
(A^{**}-\lambda)$. By Lemma \ref{numRange2} and the above remarks,
it follows that $\sH=\cran(A^*-\bar{\lambda})$ and that $\ran(A^{**}
-\lambda)$ is closed. Then also $\ran (A^*-\bar{\lambda})$ is
closed by Theorem \ref{NEW}, so that $\ran(A^*-\bar{\lambda})=\sH$.
\end{proof}

A relation $A$ in a Hilbert space $\sH$ with $\cW (A) \ne \dC$ is
said to be \textit{maximal} with respect to the numerical range
$\cW(A)$ if $\ran (A - \lambda) = \sH$ for some $\lambda \notin
\clos \cW(A)$. Then, clearly, $\lambda \in \rho(A)$ and $A$ is
closed. In fact, $A$ is maximal if and only if some open connected
component of $\Delta(A)$ belongs to the resolvent set of $A$.

\begin{lemma}\label{max2}
Let $A$ be a relation in a Hilbert space $\sH$ with $\cW (A) \ne
\dC$. Assume that $A$ is maximal with respect to $\cW(A)$. Then
\begin{equation}\label{maxEq3}
\mul A=\mul A^{*}.
\end{equation}
\end{lemma}

\begin{proof}
In order to prove the identity \eqref{maxEq3} it suffices to show
that $\mul A^*\subset \mul A$; cf. Lemma \ref{numRange2}. Let
$A_\infty$ be the extension of $A$ defined in \eqref{+SF}. Then
$\cW(A_\infty)=\cW(A)$ according to \eqref{funda++}. Hence, if
$\lambda \notin \clos \cW(A)$, then $\lambda$ is not an eigenvalue
of $A_\infty$. Moreover, since $A_\infty$ is an extension  of $A$ it
follows that $\ran (A-\lambda)\subset \ran (A_\infty-\lambda)$. It
follows from $\cW(A_\infty)=\cW(A)$ and $\ran
(A_\infty-\lambda)=\sH$ that $A_\infty$ is closed.  Therefore
$(A_\infty-\lambda)^{-1} \in \boldsymbol{B}(\sH)$, so that $\lambda
\in \rho(A_\infty)$. It follows from $(A-\lambda)^{-1} \subset
(A_\infty-\lambda)^{-1}$ that
 $(A-\lambda)^{-1} =(A_\infty-\lambda)^{-1}$, in other words
 $A_\infty=A$. This shows that $\mul A^* \subset \mul A$.
\end{proof}

\section{Componentwise decompositions of relations}\label{sec3}

In this section the canonical operatorwise decomposition of a
relation in a Hilbert space is used to characterize  componentwise
decompositions by means of an operator part.
Again, for simplicity, the statements are formulated for linear relations in a Hilbert space,
instead of linear relations acting from one Hilbert space to another Hilbert space.

\subsection{Canonical decompositions of relations}

A relation $A$ in a Hilbert space $\sH$
(or a relation from a Hilbert space $\sH$ to an other Hilbert
space $\sK$) is said to be \textit{singular} if
\begin{equation}
\label{sing01} \ran A \subset \mul A^{**} \quad \mbox{or
equivalently} \quad \cran A \subset \mul A^{**}.
\end{equation}
The equivalence here is due to the closedness of $\mul
A^{**}$. Furthermore, the inclusion
\begin{equation}
\label{sing0} \mul A^{**} \subset \cran A,
\end{equation}
follows from \eqref{impor} as $\mul A^{**} \subset \ran
A^{**}$. Therefore, a linear relation $A$ is singular if and
only if
\begin{equation}
\label{sing00} \cran A = \mul A^{**},
\end{equation}
which follows from \eqref{sing01} and \eqref{sing0}.
There is also an alternative characterization in terms of sequences
which goes back to \^Ota \cite{Ot87} in the case of densely defined
operators; cf. \cite{HSS??}.

\begin{proposition}\label{zoli5}
Let $A$ be a relation in a Hilbert space $\sH$. Then the
following statements are equivalent:
\begin{enumerate}\def\labelenumi{\rm (\roman{enumi})}
\item $A$ is singular;
\item for each $\varphi' \in \ran A$ there exists
a sequence $\{h_n,h_n'\} \in A$
such that $h_n \to 0$ and $h_n' \to \varphi'$.
\end{enumerate}
\end{proposition}

\begin{proof}
The equivalence is obtained by rewriting the condition $\ran A\subset \mul A^{**}$
elementwise using the definition of the closure $A^{**}$ of $A$.
\end{proof}

In what follows a relation
$A$ in a Hilbert space $\sH$ (or a relation from a Hilbert space
$\sH$ to an other Hilbert space $\sK$) is said to be
\textit{regular} if its closure $A^{**}$ is an operator. Thus a
regular relation is automatically an operator.

Let $A$ be a not necessarily closed relation in $A$ in the Hilbert space $\sH$
and define the subspace space $\sH_{A}$ by
\begin{equation}\label{oppa1}
 \sH_A \okr \cdom A^*=\sH \ominus \mul A^{**}.
\end{equation}
Since $\mul A \subset \mul A^{**}$, it follows that
\begin{equation}\label{oppa2}
 \sH_A \subset \sH \ominus \mul A.
\end{equation}
Let $P$ be the orthogonal projection from $\sH$ onto $\sH_A$.
Introduce the following relations:
\begin{equation}\label{reg}
 A_{\rm reg}\okr PA=\{\,\{f,P g\};\, \{f,g\}\in A \,\},
\end{equation}
called the \textit{regular part} of $A$, and
\begin{equation}\label{sing}
 A_{\rm sing}\okr (I-P)A=\{\,\{f,(I-P) g\};\, \{f,g\}\in A \,\},
\end{equation}
called the \textit{singular part} of $A$. Observe that $\dom A_{\rm
reg}=\dom A_{\rm sing}=\dom A$. The following operatorwise sum
decomposition for linear relations acting from one Hilbert space to
another Hilbert space was proved in \cite[Theorem~4.1]{HSSS07}; in
the case that $A$ is an operator it can be found from \cite{Ot87},
\cite{J80}. A short proof of this result can be given by means of
Lemma~\ref{prodlem*} and Lemma \ref{ap}.

\begin{theorem}\label{HSS}
Let $A$ be a relation in a Hilbert space $\sH$. Then $A$ admits
the canonical operatorwise sum decomposition
\begin{equation}\label{HSSSdec}
 A= A_{\rm reg} +  A_{\rm sing},
\end{equation}
where $ A_{\rm reg}$ is a regular operator in $\sH$ and $ A_{\rm
sing}$ is a singular relation in $\sH$ with
\begin{equation}\label{HSSSdec3}
({A}_{\rm{reg}})^{**}=({A^{**}})_{\rm{reg}},
\quad
 ({A}_{\rm{sing}})^{**} =(({A^{**}})_{\rm{sing}})^{**},
\quad
 \mul A_{\rm{sing}} =\mul A.
\end{equation}
\end{theorem}

\begin{proof}
Let $P$ be the orthogonal projection from $\sH$ onto $\sH_A=\cdom A^*$.
The decomposition \eqref{HSSSdec} is clear.

By definition $A_{\rm{reg}}=PA$ and hence by Lemma~\ref{prodlem*} and Lemma \ref{ap}
\[
 (A_{\rm{reg}})^*=(PA)^*=A^*P=A^* \hoplus (\mul A^{**} \times \{0\}).
\]
In particular, $\dom (A_{\rm{reg}})^*= \dom A^* \oplus \mul A^{**}$, so that
$\cdom (A_{\rm{reg}})^*=\sH$, which is equivalent to $\mul (A_{\rm{reg}})^{**}=\{0\}$; cf.
Lemma~\ref{lemrelate}. Thus, the relation $A_{\rm{reg}}$ in \eqref{reg} is
regular.

Again, by definition $A_{\rm{sing}}=(I-P)A$ and hence by Lemma~\ref{prodlem*}
and Lemma \ref{ap}
\[
 (A_{\rm{sing}})^*=((I-P)A)^*=A^*(I-P)=\cdom A^* \times \mul A^*.
\]
Since  $\dom (A_{\rm{sing}})^*=\cdom A^*$, it follows that
$\mul (A_{\rm{sing}})^{**}=\mul A^{**}$; cf.
Lemma~\ref{lemrelate}. Therefore,
$\ran A_{\rm{sing}}\subset \mul A^{**}=\mul (A_{\rm{sing}})^{**}$ and
$A_{\rm{sing}}$ is singular.

It remains to prove the identities in \eqref{HSSSdec3}. The identities
$(PA)^*=A^*P=(PA^{**})^*$ show that
\[
 (A_{\rm{reg}})^{*}=A^*P=((A^{**})_{\rm{reg}})^*
\]
and hence $(A_{\rm{reg}})^{**}=((A^{**})_{\rm{reg}})^{**}$. Since
$\ran (I-P)=\mul A^{**}$ it follows that $(A^{**})_{\rm{reg}}\subset
A^{**}$. This implies that $(A^{**})_{\rm{reg}}$ is closed: indeed,
if $\{f_n,f_n'\}\in A^{**}$ and $\{f_n,Pf_n'\}\to \{f,f'\}$, then
$\{f,f'\}\in A^{**}$ and $f'=Pf'$, so that $\{f,f'\}=\{f,Pf'\}\in
(A^{**})_{\rm{reg}}$. Therefore,
$((A^{**})_{\rm{reg}})^{**}=(A^{**})_{\rm{reg}}$ yielding the first
identity in \eqref{HSSSdec3}.

Likewise, the
equalities $((I-P)A)^*=A^*(I-P)=((I-P)A^{**})^*$ imply that
\[
 (A_{\rm{sing}})^{*}=A^*(I-P)=((A^{**})_{\rm{sing}})^*.
\]
Hence $(A_{\rm{sing}})^{**}=((A^{**})_{\rm{sing}})^{**}$, and the
second indentity in \eqref{HSSSdec3} is proved.

Finally, since $\mul A\subset \mul A^{**}$, one obtains
\[
  \mul A_{\rm{sing}} =\{\,(I-P) f' :\, \{0,f'\}\in A\,\}= \{\,f' :\,
  \{0,f'\}\in A\,\}=\mul A.
\]
This completes the proof.
\end{proof}

Several illustrations of Theorem~\ref{HSS} can be found in \cite{HSS??},
\cite{HSSS07}. Canonical decompositions of
relations have their counterparts in the canonical decomposition of
pairs of nonnegative sesquilinear forms (see \cite{HSS??}).

It is clear from the definitions that $A$ is regular if and only if
in \eqref{HSSSdec} $A_{\rm{sing}}$ is the zero operator on $\dom A$,
and similarly, $A$ is singular if and only if in \eqref{HSSSdec}
$A_{\rm{reg}}$ is the zero operator on $\dom A$. The condition that
$A$ is singular
can be characterized also as follows; cf. \cite{HSSS07}.

\begin{proposition}
\label{sin} Let $A$ be a relation in a Hilbert space $\sH$. Then the
following statements are equivalent:
\begin{enumerate}
\def\labelenumi{\rm (\roman{enumi})}

\item $A$ is singular;

\item $\dom A^* \subset \ker A^*$ or, equivalently, $\dom A^*=\ker
A^*$;

\item $A^*=\dom A^* \times \mul A^*$;

\item $A^{**}=\cdom A \times \mul A^{**}$.
\end{enumerate}
In particular, if one of the relations $A$, $A^{-1}$, $A^*$, or
$A^{**}$ is singular, then all of them are singular.
\end{proposition}

\begin{proof}
(i) $\Longrightarrow$ (ii) The identity in \eqref{sing00} implies
that $(\cran A)^\perp=(\mul A^{**})^\perp$, which is equivalent to
$\ker A^*=\cdom A^*$ by Lemma \ref{lemrelate}. In particular,
$\dom A^* \subset \ker A^*$.

(ii) $\Longrightarrow$ (iii) Let $\{f,g\} \in A^*$. Now $f \in
\dom A^*$ implies that $f \in \ker A^*$. Therefore $\{f,0\} \in
A^*$ and then also $\{0,g\} \in A^*$, or $g \in \mul A^*$. This
shows that $\{f,g\} \in \dom A^* \times \mul A^*$. Conversely, let
$\{f,g\} \in \dom A^* \times \mul A^*$. Then $\{0,g\} \in A^*$.
Moreover, $f \in \dom A^*$ and by (ii) $f \in \ker A^*$, i.e.,
$\{f,0\} \in A^*$. Thus $\{f,g\} \in A^*$.

(iii) $\Longrightarrow$ (iv) Taking adjoints in (iii) yields
$A^{**}=(\mul A^*)^\perp \times (\dom A^*)^\perp$, which gives
(iv) by means of Lemma \ref{lemrelate}.

(iv) $\Longrightarrow$ (i) Now $\ran A^{**} = \mul A^{**}$ gives
$\ran A \subset \mul A^{**}$. Thus $A$ is singular.

The last statement is clear from the equivalence of (i)--(iv).
\end{proof}

The following characterizations for regularity of $A$ are immediate
from definitions. Further characterizations for regularity are given
after componentwise decompositions have been introduced; see
Proposition~\ref{operdec}.

\begin{proposition}
\label{regchar} Let $A$ be a relation in a Hilbert space $\sH$. Then
the following statements are equivalent:
\begin{enumerate}
\def\labelenumi{\rm (\roman{enumi})}

\item $A$ is regular, i.e., a closable operator;

\item $\mul A^{**}=\{0\}$;

\item $A^*$ is densely defined.
\end{enumerate}
\end{proposition}
\begin{proof}
The equivalence of (i) and (ii) holds by definition of closability.
The equivalence of (ii) and (iii) is obtained from
Lemma~\ref{lemrelate}.
\end{proof}

Boundedness of the regular and singular part of $A$ in
Theorem~\ref{HSS} can be characterized as follows.

\begin{proposition}\label{RSbounded}
Let $A$ be a relation in a Hilbert space $\sH$. Then:
\begin{enumerate}\def\labelenumi{\rm (\roman{enumi})}
\item $A_{\rm reg}$ is a bounded operator if and only if $\dom A^*$ is closed;
\item $A_{\rm sing}$ is a bounded operator if and only if it is the zero operator on
$\dom A$, i.e., $A_{\rm sing}=\dom A\times \{0\}$.
\end{enumerate}
In particular, if $\ran A_{\rm sing}\neq \{0\}$ then $A_{\rm sing}$
is either an unbounded operator or it is a multivalued relation with
$\mul A_{\rm sing}=\mul A$.
\end{proposition}
\begin{proof}
(i) According to Theorem~\ref{HSS} ${A}_{\rm{reg}}$ is regular (i.e. closable) and
$({A}_{\rm{reg}})^{**}=({A^{**}})_{\rm{reg}}$. Hence by Lemma~\ref{first}
${A}_{\rm{reg}}$ is bounded if and only if $({A^{**}})_{\rm{reg}}$ is bounded,
or equivalently, $\dom ({A^{**}})_{\rm{reg}}=\dom A^{**}$ is closed.
Then, equivalently, $\dom A^*$ is closed by Theorem \ref{NEW}.

(ii)
Assume that $A_{\rm sing}$ is a bounded operator, so that also
$({A}_{\rm{sing}})^{**}$ is a bounded operator.
According to Theorem~\ref{HSS} ${A}_{\rm{sing}}$ is singular, so that
\[
\ran {A}_{\rm{sing}}\subset \mul ({A}_{\rm{sing}})^{**}=\{0\}.
\]
Therefore, ${A}_{\rm{sing}}=\dom A\times\{0\}$. Conversely, if
$A_{\rm sing}=\dom A\times \{0\}$ then $\ran {A}_{\rm{sing}}=\{0\}$,
and $A_{\rm{sing}}$ is bounded and singular.

The last statement is immediate from (ii) and \eqref{HSSSdec3} in
Theorem~\ref{HSS}.
\end{proof}

Note that by Proposition~\ref{RSbounded} $\dom A^*$ is closed if and
only if $A_{\rm reg}$ is bounded, which by Corollary~\ref{boundcor}
is equivalent to $\dom ({A}_{\rm{reg}})^*=\sH$. Thus, $\dom A^*$ is
closed if and only if $\dom ({A}_{\rm{reg}})^*=\sH$, which  is also
clear from the identity
\[
 \dom ({A}_{\rm{reg}})^*=\dom A^* \oplus \mul A^{**}.
\]
From Proposition~\ref{Dclosed} one obtains
for part (i) in Proposition~\ref{RSbounded} the following formally
weaker, but equivalent, criterion for boundedness of $A_{\rm reg}$.

\begin{corollary}
$A_{\rm reg}$ is a bounded operator if and only if
$\ran (A^{**})_{\rm reg}\subset \dom A^*$,
or equivalently, $\ran (A_{\rm reg})^{**}\subset \dom A^*$.
\end{corollary}

\begin{proof}
By Theorem~\ref{HSS} $({A}_{\rm{reg}})^{**}=({A^{**}})_{\rm{reg}}$ and hence
the assertions follows from Proposition~\ref{RSbounded} (i) and
the equivalence of items (i) and (ii) in Proposition~\ref{Dclosed}.
\end{proof}

\begin{corollary}\label{rud}
Let $A$ be a relation in a Hilbert space $\sH$. Then
\[
 A \subset A_{\rm reg} \hplus (A^{**})_{\rm mul}
 \subset (A^{**})_{\rm reg} \hplus (A^{**})_{\rm mul}.
\]
\end{corollary}

\begin{proof}
Let $\{f,f'\} \in A$ and consider $f'=P f'+ (I-P) f'$. This leads to
\[
 \{f,f'\}=\{f, P f'\} +\{0, (I-P)f'\}.
\]
Hence, the first inclusion is clear. Furthermore, the second
inclusion follows from
\[
  A_{\rm reg}\subset ( A_{\rm reg})^{**}=(A^{**})_{\rm reg}.
\]
where the identity holds by Theorem~\ref{HSS}.
\end{proof}

\begin{remark}\label{numi}
Let $A$ be a relation in a Hilbert space $\sH$, which satisfies
$\mul A^{**} \subset \mul A^*$. Then $\cW(A)=\cW(A_{\rm reg})$. To
see this, observe that
\[
(A_{\rm reg} f,f)=(P f',f)=(f',f), \quad \{f,f'\} \in A,
\]
cf. Lemma \ref{easy}.
\end{remark}

\subsection{Componentwise decompositions of relations via operator part}

By means of the Hilbert space $\sH_{A}$ the restriction $A_{\rm
op}$ \index{part of relation!minimal operator $A_{\rm op}$} of $A$
is defined by
\begin{equation}\label{oppa}
 A_{\rm op} \okr \{\,\{f,g\}\in A;\, g\in\sH_{A} \,\}.
\end{equation}
Equivalently, $A_{\rm op}$ can be written in the following way:
\begin{equation}\label{oppaa}
 A_{\rm op}=A \cap (\sH \times \sH_A).
\end{equation}
By definition $A_{\rm op}$ is (the graph of) an operator in $\sH$
(see \eqref{impor}) and clearly
\begin{equation}\label{reg00}
 A_{\rm op}\subset A_{\rm reg},
\end{equation}
where $A_{\rm reg}$ is as in \eqref{reg}. Since $A_{\rm reg}$ is
closable in $\sH$, the operator $A_{\rm op}$ is also closable in
$\sH$. By means of the multivalued part of $A$ the restriction
$A_{\rm mul}$ of $A$ is defined by
\begin{equation}\label{mul}
 A_{\rm mul}\okr\{0\} \times \mul A.
\end{equation}
In particular, the relation $A_{\text{mul}}$ is closed in $\sH\times
\sH$ if and only if the subspace $\mul A$ is closed in $\sH$. By
taking adjoints in \eqref{mul} one gets
\begin{equation}\label{mulA*}
 (A_{\rm mul})^*=(\mul A)^\perp \times \sH,
\end{equation}
so that $A_{\rm mul}$ is a symmetric relation in $\sH$.
By taking adjoints in \eqref{mulA*} one gets
\begin{equation}\label{mulA**}
 (A_{\rm mul})^{**}=\{0\} \times \cmul A.
\end{equation}
The following theorem is concerned with the decomposition of a,
not necessarily closed, relation $A$ in the graph sense via its
multivalued part.

\begin{theorem}\label{unidec}
Let $A$ be a relation in a Hilbert space $\sH$. If there exists a
relation $B$ in $\sH$, such that
\begin{equation}\label{graphdecB}
 A=B \hplus A_{\rm mul}, \quad \ran B\subset \sH_{A},
\end{equation}
then the sum in \eqref{graphdecB} is direct and $B$ is a closable
operator which coincides with $A_{\rm op}$. In particular, the
decomposition of $A$ in \eqref{graphdecB} is unique.
\end{theorem}

\begin{proof}
It follows from \eqref{impor} that the sum in \eqref{graphdecB} is
direct. The equality in \eqref{graphdecB} implies that $B\subset
A$ and $\dom B=\dom A$. Since $\ran B\subset \sH_{A}$, it follows
from \eqref{oppa} that $B\subset A_{\rm op}$; in particular, it
follows that $B$ is a closable operator in $\sH$. Furthermore, the
inclusion $B \subset A_{\rm op}s$ implies that $\dom A=\dom
B\subset \dom A_{\rm op}$, and thus $\dom A_{\rm op}=\dom A$.
Since $A_{\rm op}$ and $B\subset A_{\rm op}$ are (closable)
operators with $\dom B=\dom A_{\rm op}$, the equality $B=A_{\rm
op}$ follows.
\end{proof}

Hence if $A$ admits a componentwise sum decomposition of the form
\eqref{graphdecB}, then it follows that
\begin{equation}\label{graphdec}
 A=A_{\rm op} \hplus A_{\rm mul},
\end{equation}
and $A_{\rm op}$ in \eqref{oppa} can be viewed as the
\textit{minimal operator part}\index{part of relation!minimal
operator $A_{\rm op}$} of $A$ which together with $A_{\rm mul}$
decomposes $A$ as a componentwise sum, cf. \eqref{jan00}. Clearly,
by \eqref{oppa2} the condition $\ran B\subset \sH_{A}=\sH \ominus
\mul A^{**}$ implies that $\ran B\subset \sH \ominus \mul A$. It
is precisely in the case that $\mul A$ is dense in $\mul A^{**}$
(recall that $A$ is not necessarily closed) where the condition
$\ran B\subset \sH_{A}$ in \eqref{graphdecB} is equivalent to the
condition $\ran B\subset \sH \ominus \mul A$.

A relation $A$ in a Hilbert space $\sH$ is said to be
\textit{decomposable}\index{relation!decomposable} if the
componentwise decomposition \eqref{graphdecB}, or equivalently,
\eqref{graphdec} is valid; cf. Subsection \ref{decc}. The next
theorem gives necessary and sufficient conditions for $A$ to be
decomposable and, furthermore, relates the decomposition of the
relation $A$ in \eqref{graphdec} to the the operatorwise sum
decomposition of $A$ in \eqref{HSSSdec}.

\begin{theorem}\label{reglem}
Let $A$ be a relation in a Hilbert space $\sH$, let $P$ be the
orthogonal projection from $\sH$ onto $\sH_A=\cdom A^*$, and let
the relations $A_{\rm reg}$, $A_{\rm mul}$, and $A_{\rm op}$ be
defined as above. Then the following statements are equivalent:
\begin{enumerate}\def\labelenumi{\rm (\roman{enumi})}

\item $A$ is decomposable;

\item $\dom A_{\rm op}=\dom A$;

\item $A_{\rm reg} = A_{\rm op}$;

\item $A_{\rm reg}\subset A$;

\item $\ran(I-P)A\subset \mul A$;

\item $A=A_{\rm reg} \hplus A_{\rm mul}$.
\end{enumerate}
\end{theorem}

\begin{proof}
(i) $\Longrightarrow$ (ii) This implication is clear, since $\dom
A_{\rm mul}=\{0\}$.

(ii) $\Longrightarrow$ (iii) The assumption gives $\dom A_{\rm
op}=\dom A=\dom A_{\rm reg}$. Now \eqref{reg00} implies that
$A_{\rm op}=A_{\rm reg}$, since $A_{\rm op}$ and $A_{\rm reg}$ are
operators.

(iii) $\Longrightarrow$ (iv) This implication is clear, since
$A_{\rm op}\subset A$ by definition.

(iv) $\Longleftrightarrow$ (v) Let $\{f,g\}\in A$ and write
$\{f,g\}=\{f,P g\}\hplus\{0,(I-P)g\}$. Here $\{f,P g\}\in A_{\rm
reg}$ and the condition $\{f,P g\}\in A$ is equivalent to
$\{0,(I-P)g\} \in A$. This shows that $A_{\rm reg}\subset A$ if
and only if $(I-P)(\ran A)\subset \mul A$, which proves the claim.

(iv), (v) $\Longrightarrow$ (vi) By decomposing $\{f,g\}\in A$ as
$\{f,g\}=\{f,P g\}\hplus\{0,(I-P)g\}$ one concludes that $A\subset
A_{\rm reg} \hplus A_{\rm mul}$. The reverse inclusion is clear,
and thus (vi) follows.

(vi) $\Longrightarrow$ (i) It suffices to prove that $A_{\rm reg}
= A_{\rm op}$. The equality in (vi) implies that $A_{\rm reg}
\subset A$. Hence, if $\{f,g\}\in A_{\rm reg}$ then $\{f,g\} \in
A$, $g\in\sH_{A}$, and thus $\{f,g\} \in A_{\rm op}$. Therefore,
$A_{\rm reg} \subset A_{\rm op}$, while the reverse inclusion is
always true; cf. \eqref{reg00}.

This completes the proof.
\end{proof}

Recall that $A$ is a bounded operator if and only if $\ran
A^{**}\subset \dom A^*$; see Corollary \ref{boundcor}. From Theorem
\ref{reglem} one gets the following characterization for the
essentially weaker condition $\ran A\subset \cdom A^*$.

\begin{proposition}\label{operdec}
Let $A$ be a relation in a Hilbert space $\sH$.
Then the following statements are equivalent:
\begin{enumerate}\def\labelenumi{\rm (\roman{enumi})}

\item $\ran A\subset \cdom A^*\,(=\sH_A)$;

\item $A_{\rm op}=A$;

\item $A$ is regular, i.e., a closable operator;

\item $\sH_A=\sH$;

\item $A$ is a decomposable operator.

\end{enumerate}
\end{proposition}

\begin{proof}
(i) $\Longleftrightarrow$ (ii) This is clear from the definition
of $A_{\rm op}$ in \eqref{oppa}.

(ii) $\Longrightarrow$ (iii) If $A_{\rm op}=A$ then, together with
$A_{\rm op}$, $A$ is closable.

(iii) $\Longrightarrow$ (iv) If $A$ is closable, then $\mul
A^{**}=\{0\}$ and hence $\sH_A=\sH$.

(iv) $\Longrightarrow$ (v) If $\sH_A=\sH$ then $A_{\rm reg}=A$ and
hence $A$ is decomposable by Theorem~\ref{reglem} (iv).

(v) $\Longrightarrow$ (i) If $A$ is a decomposable operator, then
$A_{\rm mul}=\{0\}\times\{0\}$ and hence $(I-P)A=0$ by
Theorem~\ref{reglem} (v). This means that $\ran A\subset
\ker(I-P)=\sH_A$.
\end{proof}

The next result is clear from Proposition~\ref{operdec}.

\begin{corollary}\label{vasa+}
An operator $A$ in a Hilbert space $\sH$ is decomposable if and only
if it is regular, i.e., $A_{\rm sing}=0$.
\end{corollary}

Hence, an operator $A$ is decomposable in the sense of Theorem
\ref{unidec} if and only if it is closable; in this case $A_{\rm
mul}=\{0\}\times\{0\}$ and $A=A_{\rm op}$. In this sense the
decomposability property introduced via Theorem~\ref{unidec} can
be seen as an extension of the notion of closability of operators
for linear relations.

Singular operators and relations are not in general decomposable;
for them the following result holds.

\begin{proposition}\label{singu}
Let $A$ be a relation in a Hilbert space $\sH$. Then
\begin{enumerate}\def\labelenumi{\rm (\roman{enumi})}

\item $A$ is singular and decomposable if and only if
$A=\dom A \times \mul A$, or equivalently, $\dom A=\ker A$.

\item A singular operator $A$ is  decomposable if and only if it is
bounded, or equivalently, $A$ is the zero operator on its domain,
i.e., $A=\dom A \times \{0\}$.
\end{enumerate}
\end{proposition}

\begin{proof}
(i) The relation $A$ is singular if $\ran A \subset \mul A^{**}$.
This is equivalent to $A_{\rm reg}=\dom A \times \{0\}$. By Theorem
\ref{reglem}, $A$ is decomposable if and only if $A=A_{\rm reg}
\hplus A_{\rm mul}$. Hence, if $A$ is singular and decomposable,
then $A=\dom A \times \mul A$. Conversely, if $A$ is of the form
$A=\dom A \times \mul A$, then clearly $A$ is singular and
decomposable. Furthermore, it is easy to check that $A=\dom A \times
\mul A$ is equivalent to $\dom A=\ker A$.

(ii) This is clear from part (i) and
Proposition~\ref{RSbounded}~(ii).
\end{proof}

Next some sufficient conditions for decomposability of relations are given.

\begin{corollary}\label{muldec}
If the relation $A$ satisfies $\mul A=\mul A^{**}$, then $A$ is
decomposable and the relation $A_{\rm mul}$ is closed.
\end{corollary}

\begin{proof}
Note that $I-P$ is the orthogonal projection onto $\mul A^{**}$.
Therefore, in this case $\ran(I-P)A\subset \mul A^{**}=\mul A$, and hence
$A$ is decomposable by Theorem~\ref{reglem} (v). Since $A^{**}$ is closed,
also $\mul A=\mul A^{**}$ and $A_{\rm mul}$ are closed.
\end{proof}

\begin{corollary}\label{closeddec}
If the relation $A$ is a closed, then $A$ is decomposable and the
relations $A_{\rm op}=A_{\rm reg}$ and $A_{\rm mul}$ are closed.
\end{corollary}

\begin{proof}
Since $A$ is closed, $\mul A=\mul A^{**}$ and the first statement
is obtained from Corollary~\ref{muldec}. Moreover, it is clear
from \eqref{oppaa} that $A_{\rm op}$ is closed.
\end{proof}

Later, in Proposition~\ref{muldec2}, it is shown that if $\mul A$ is
closed then the sufficient condition $\mul A=\mul A^{**}$ for
decomposability becomes also necessary.

Let $A$ be a relation in the Hilbert space $\sH$ which is not
necessarily closed. Then the closure of $A$ is given by $A^{**}$;
recall that $(A^{**})_{\rm mul}=\{0\} \times \mul A^{**}$.  It is
useful to observe that
\[
\mul A \subset \cmul A \subset \mul A^{**},
\]
and, furthermore,  that
\begin{equation}\label{mmuull}
(A_{\rm mul})^{**}=(A^{**})_{\rm mul} \quad \Longleftrightarrow
\quad \cmul A=\mul A^{**},
\end{equation}
cf. \eqref{mulA**}. Observe that $\sH_{A^{**}}=\sH_{A}$. Therefore
the operator $(A^{**})_{\rm op}$ is given by
\begin{equation}\label{burb}
 (A^{**})_{\rm op}=A^{**} \cap (\sH \times \sH_A).
\end{equation}
It is clear from \eqref{oppaa}, \eqref{mul}, and \eqref{burb} that
$A_{\rm op}\subset(A^{**})_{\rm op}$ and $A_{\rm mul}\subset
(A^{**})_{\rm mul}$. Therefore, Corollary \ref{closeddec}, applied
to $A^{**}$, implies that
\[
(A_{\rm op})^{**} \subset(A^{**})_{\rm op}, \quad (A_{\rm
mul})^{**} \subset (A^{**})_{\rm mul}.
\]
The following result is a direct consequence of Theorem
\ref{reglem}.

\begin{proposition}\label{vasa}
Let $A$ be a relation in a Hilbert space $\sH$. Then $A^{**}$ is
decomposable and has the following componentwise sum
decomposition:
\begin{equation}\label{grdec**}
 A^{**}=(A^{**})_{\rm op} \hplus (A^{**})_{\rm mul}.
\end{equation}
Moreover, if the relation $A$ is decomposable, then
\begin{equation}\label{rud0}
 (A_{\rm op})^{**}=(A^{**})_{\rm op}, \quad (A_{\rm
mul})^{**}=(A^{**})_{\rm mul}.
\end{equation}
\end{proposition}

\begin{proof}
Let $A$ be any relation in $\sH$. Then $A^{**}$ is closed, and by
Corollary \ref{closeddec}, $A^{**}$ is decomposable, which leads to
the decomposition \eqref{grdec**}.

Now assume that the relation $A$ is decomposable with
decomposition $A=A_{\rm op} \hplus A_{\mul}$. Then it follows from
$\ran A_{\rm op} \subset \sH_A$, that
\begin{equation}\label{rud00}
 A^{**}=(A_{\rm op})^{**} \hplus (A_{\mul})^{**}.
\end{equation}
The identities \eqref{rud00} and   \eqref{mulA**} lead to the
following decomposition
\begin{equation}\label{rud1}
 A^{**}=(A_{\rm op})^{**} \hplus (\{0\} \times \cmul A).
\end{equation}
The operator $A_{\rm op}$ is closable and $\ran A_{\rm op} \subset
\sH_A$. Hence, $(A_{\rm op})^{** }$ is an operator and $\ran
(A_{\rm op})^{**} \subset \sH_A$. Because $(A_{\rm op})^{**}$ is
an operator, it follows from \eqref{rud1} that $\mul A^{**}=\cmul
A$; thus \eqref{rud1} reads as
\begin{equation*}\label{rud2}
 A^{**}=(A_{\rm op})^{**} \hplus (A^{**})_{\rm mul}.
\end{equation*}
An application of Theorem \ref{unidec}, applied to $A^{**}$, shows
that $(A_{\rm op})^{**}=(A^{**})_{\rm op}$. This completes the
proof.
\end{proof}

If a relation $A$ is closed, then it is decomposable by
Corollary~\ref{closeddec}, and Proposition~\ref{vasa} is a
refinement of earlier results. Observe, that in
Proposition~\ref{vasa} one has
\begin{equation}\label{eqadd0}
 (A^{**})_{\rm op}=(A^{**})_{\rm reg}=(A_{\rm reg})^{**}
\end{equation}
by Theorem \ref{reglem} and Theorem \ref{HSS}. For a
relation $A$ which is not necessarily decomposable, it follows from
$A \subset A^{**}$ and \eqref{grdec**} that
\begin{equation*}
 A \subset (A^{**})_{\rm op} \hplus (A^{**})_{\rm mul}.
\end{equation*}
This inclusion also can be seen from Corollary \ref{rud}. If $A$ is
a relation and one of the identities in \eqref{rud0} is not
satisfied, then $A$ is not decomposable. Although the conditions in
\eqref{rud0} are necessary for $A$ to be decomposable, they are not
sufficient. In fact, it is possible that both identities in
\eqref{rud0} are satisfied, while $A$ is not decomposable; see
Example~\ref{exam2}.

A relation $A$, whose regular part is bounded need not be
decomposable; see e.g. Example \ref{exam1}. Decomposability of such
relations is characterized in the next result.

\begin{proposition}\label{Dom*cl}
Let $A$ be a relation in a Hilbert space $\sH$.
Then the following statements are equivalent:
\begin{enumerate}\def\labelenumi{\rm (\roman{enumi})}

\item $A$ is decomposable with a bounded operator part $A_{\rm
op}$;

\item $A_{\rm reg}=A_{\rm op}$ is bounded;

\item $\dom A^*$ is closed and $A_{\rm reg}=A_{\rm op}$.

\end{enumerate}
Furthermore, the following weaker statements are equivalent:
\begin{enumerate}\def\labelenumi{\rm (\roman{enumi})}
\setcounter{enumi}{3}

\item $A_{\rm op}$ is bounded, densely defined in $\cdom A$, and
$(A_{\rm mul})^{**}=(A^{**})_{\rm mul}$;

\item $A_{\rm reg}$ is bounded and the conditions in \eqref{rud0} are satisfied;

\item $\dom A^*$ is closed and the conditions in \eqref{rud0} are satisfied.

\end{enumerate}
If, in addition, $\ran(I-P)A\subset \mul A$ or $\mul A$ is closed,
then the conditions {\rm (iv)--(vi)} are also equivalent to the
conditions {\rm (i)--(iii)}.
\end{proposition}

\begin{proof}
(i) $\Longleftrightarrow$ (ii) This is clear from
Theorem~\ref{reglem}; see items (i) and (iii).

(ii) $\Longleftrightarrow$ (iii) This is an immediate consequence
of Proposition \ref{RSbounded}.

(iv) $\Longrightarrow$ (v) Since $A_{\rm op} \subset A_{\rm reg}$
and the operator $A_{\rm reg}$ is closable, the assumption that
$A_{\rm op}$ is densely defined and bounded in $\cdom A$ leads to
the equality
\[
 (A_{\rm op})^{**}=(A_{\rm reg})^{**},
\]
cf. Corollary~\ref{firstcor}. Hence $(A_{\rm reg})^{**}$ and, in
particular, $A_{\rm reg}$ is bounded. Moreover, $(A_{\rm
op})^{**}=(A^{**})_{\rm op}$ is now obtained from \eqref{eqadd0}.

(v) $\Longrightarrow$ (iv) If $A_{\rm reg}$ is bounded then
$A_{\rm op}\subset A_{\rm reg}$ is bounded, too. By
Proposition~\ref{vasa} $(A^{**})_{\rm reg}=(A^{**})_{\rm op}$.
Hence, if $(A_{\rm op})^{**}=(A^{**})_{\rm op}$ then $\cdom A_{\rm
op}=\dom (A^{**})_{\rm op}=\cdom (A^{**})_{\rm reg}=\cdom A$ (cf.
Lemma \ref{first}), i.e., $A_{\rm op}$ is densely defined in
$\cdom A$.

(v) $\Longleftrightarrow$ (vi) Again this holds by Proposition
\ref{RSbounded}.

To prove the last statement note that (i) implies (v) by Proposition \ref{vasa}.
On the other hand, if $\ran(I-P)A\subset \mul A$ then $A$
is decomposable by Theorem~\ref{reglem}
and thus (iv) implies (i). Similarly, the assumption $\mul A$ is closed together with
$(A_{\rm mul})^{**}=(A^{**})_{\rm mul}$ implies that $\mul A=\mul A^{**}$, so that
$A$ is decomposable by Corollary \ref{muldec}. Hence, again (iv) implies (i).
\end{proof}

Proposition \ref{Dom*cl} indicates that even in the case where
$A_{\rm reg}$ is a bounded operator, the equalities in \eqref{rud0}
are not sufficient for the decomposability of $A$. In fact, this may
happen also in the case where $A_{\rm reg}$ is closed and bounded;
see Example~\ref{exam0a}. However, if $\mul A$ is closed then the
situation is different; see Corollary~\ref{mulclosed}.

\subsection{Componentwise decompositions for relations via the multivalued part}

Theorem \ref{reglem} shows that a relation $A$ in a Hilbert $\sH$
is decomposable in the sense of Theorem \ref{unidec} if and only
if $A_{\rm op}=A_{\rm reg}$, where $A_{\rm reg} =PA$ with $P$ the
orthogonal projection from $\sH$ onto $\sH_A=\sH\ominus \mul
A^{**}$. Closely related to the regular part $A_{\rm reg}$ is the
relation
\begin{equation}\label{Pm}
 A_{\rm m} \okr P_{\rm m} A=\{\, \{f,P_{\rm m} f'\};\, \{f,f'\}\in
A \},
\end{equation}
where $P_{\rm m}$ is the orthogonal projection from $\sH$ onto
$\sH\ominus \cmul A$. $A_{\rm m}$ can be thought of as the {\em
maximal operator part}\index{part of relation!maximal operator
$A_{\rm m}$} of $A$, cf. Theorem \ref{muldecthm} below. Observe
that
\[
 \mul A_{\rm m} =\{\, P_{\rm m} f';\, \{0,f'\}\in A \,\}=\{0\},
\]
i.e., $A_{\rm m}$ is an operator. Note that
\begin{equation}\label{muu}
\sH=\cdom A \oplus \mul A^{**}=\cdom A \oplus (\mul A^{**} \ominus \cmul A) \oplus \cmul A,
\end{equation}
so that $\ran P \subset \ran P_{\rm m}$. Therefore $A_{\rm reg}=P
A_{m}$ and, in addition,
\begin{equation}\label{Am}
 A_{\rm op} \subset A_{m}.
\end{equation}
The operator $A_{\rm m}$ can be used to give one further
equivalent condition for $A$ to be decomposable, which is stated
as item (ii) in the next theorem.

\begin{theorem}\label{muldecthm}
Let $A$ be a relation in a Hilbert space $\sH$. Then $A_{\rm m}$
is an operator and the following statements are equivalent:
\begin{enumerate}\def\labelenumi{\rm (\roman{enumi})}

\item $A$ is decomposable;

\item $A_{\rm m} =A_{\rm op}$.
\end{enumerate}
Furthermore, the following weaker statements are equivalent:
\begin{enumerate}\def\labelenumi{\rm (\roman{enumi})}
\setcounter{enumi}{2}

\item $A_{\rm m} =A_{\rm reg}$;

\item $\ran A_{\rm m} \subset \sH_A$;

\item $\mul A^{**}=\cmul A$;

\item $A_{\rm m}$ is a closable operator.

\end{enumerate}
If, in addition, $\mul A$ is closed, then the conditions
{\rm{(iii)--(vi)}} are also equivalent to the conditions
\rm{(i)--(ii)}.
\end{theorem}
\begin{proof}
(i) $\Longrightarrow$ (ii) It suffices to show that $A_{\rm m}
\subset A_s$. Let $A$ be decomposed as in \eqref{graphdec}. If
$\{f,f'\} \in A$, then $\{f,f'\}=\{f,A_{\rm op}f\} +\{0,
\varphi\}$ with $\varphi \in \mul A$. In particular, $P_{\rm m}
f'=A_{\rm op}f$ for $\{f,f'\} \in A$; i.e. $A_{\rm m} \subset
A_{\rm op}$.

(ii) $\Longrightarrow$ (i) If $A_{\rm m} =A_{\rm op}$, then $\dom
A_{\rm op}=\dom A$ and by Theorem~\ref{reglem} this is equivalent
to $A$ being decomposable.

Next the equivalence of (iii)--(vi) will be proved.

(iii) $\Longrightarrow$ (iv) This is clear since $\ran A_{\rm
reg}\subset \sH_A$.

(iv) $\Longrightarrow$ (v) Observe, that $(A_{\rm m})^*=(P_{\rm m}
A)^*=A^*P_{\rm m}$ and $(A_{\rm m})^{**}=(A^*P_{\rm m})^*\supset
P_{\rm m} A^{**}$ by Lemma~\ref{prodlem*}. The assumption in (iv)
implies that $\ran (A_{\rm m})^{**}\subset \sH_A$; cf.
\eqref{impor}. Then also $\ran P_{\rm m} A^{**}\subset \sH_A$ and,
in particular, $P_{\rm m} (\mul A^{**})\subset \sH_A$, which means
that $\mul A^{**}=\cmul A$; cf. \eqref{muu}.

(v) $\Longleftrightarrow$ (vi) It follows from Lemma \ref{ap},
that $(A_{\rm m})^*=A^*P_{\rm m}=A^* \hoplus (\cmul A \times
\{0\})$, so that
\[
 \dom (A_{\rm m})^*=\dom A^*\oplus\cmul A.
\]
Now the operator $A_{\rm m}$ is closable if and only if $(A_{\rm
m})^*$ is densely defined (cf. Proposition \ref{regchar}), which
is equivalent to $\mul A^{**}=\cmul A$.

(v) $\Longrightarrow$ (iii) If $\mul A^{**}=\cmul A$ then $P_{\rm
m}=P$ and, therefore, $A_{\rm m} =PA=A_{\rm reg}$.

Next it is shown that the conditions (iii)--(vi) follow from the
conditions (i) and (ii). Namely, if $A_{\rm m} =A_{\rm op}$ or,
equivalently, $A$ is decomposable, then $A_{\rm op}=A_{\rm reg}$
holds by Theorem~\ref{reglem}, and hence $A_{\rm m} =A_{\rm
op}=A_{\rm reg}$ follows.

As to the last statement of the theorem observe, that if $\mul A$ is
closed then the condition (v) implies (i) by Corollary \ref{muldec}.
\end{proof}

It is emphasized that the conditions (iii)--(vi) in
Theorem~\ref{muldecthm} do not in general imply decomposability of
$A$; see for instance Example~\ref{exam2}.

Next decompositions of linear relations $A$ whose multivalued part
$\mul A$ is closed in the Hilbert space $\sH$ will be briefly
treated.

\begin{lemma}\label{Pmdecom}
Let $A$ be a relation in a Hilbert space $\sH$ with $\mul A$  closed.
Then $A$ admits the decomposition
\begin{equation}\label{graphdecM}
 A=A_{\rm m} \hplus A_{\rm mul},
\end{equation}
where $A_{\rm m} $ is an operator with $\dom A_{\rm m} =\dom A$.
\end{lemma}

\begin{proof}
It has been shown that $A_{\rm m} $ is an operator and that $\dom
A_{\rm m} =\dom A$. Now, rewrite $A$ as follows
\[
 A=P_{\rm m} A + (I-P_{\rm m})A =\{\,\{f,P_{\rm m}
 g\}\hplus\{0,(I-P_{\rm m})g\};\,\{f,g\}\in A\,\}.
\]
This implies that $A \subset A_{\rm m} \hplus A_{\rm mul}$.

Conversely, since $\mul A$ is closed, one has $A_{\rm mul}\subset
A$ and thus also $P_{\rm m} A\subset A$. Therefore, $P_{\rm m} A +
(I-P_{\rm m})A\subset A$.
\end{proof}

The decomposition of $A$ in Lemma~\ref{Pmdecom} for relations $A$
with $\mul A$ closed is not of the type as introduced via
Theorem~\ref{unidec}, since the condition $\ran A_{\rm m} \subset
\sH_A\,(=\cdom A^*)$ need not be satisfied. This implies that the
decomposition given in Lemma~\ref{Pmdecom} does not behave well
for instance under closures: in particular, the operator $A_{\rm
m} $ in \eqref{graphdecM} is not in general closable. In fact,
when $\mul A$ is closed, Theorem~\ref{muldecthm} shows that the
operator $A_{\rm m} $ is closable precisely when $A$ is
decomposable in the sense of Theorem \ref{unidec}.

One can reformulate the situation also by means of the
decompositions of the form \eqref{graphdec} alone.

\begin{corollary}\label{Pmcor}
Let $A$ be a relation in a Hilbert space $\sH$ with $\mul A$ closed.
Then
\[
 A \text{ is decomposable }\quad \Longleftrightarrow \quad A_{\rm
 m} \text{ is a decomposable operator.}
\]
In this case, the decomposition of $A_{\rm m} $ is trivial, i.e.,
$A_{\rm m} =(A_{\rm m} )_{\rm op}$, and the decompositions in
\eqref{graphdec} and \eqref{graphdecM} coincide:
\begin{equation}\label{graphdecM2}
 A=A_{\rm m} \hplus A_{\rm mul}=A_{\rm op}\hplus A_{\rm mul}.
\end{equation}
\end{corollary}

\begin{proof}
According to Proposition~\ref{operdec} the operator $A_{\rm m} $
is decomposable if and only if $(A_{\rm m} )_{\rm op}=A_{\rm m} $,
or equivalently, $A_{\rm m} $ is closable. This means that $\mul
A=\mul A^{**}$. By Theorem~\ref{muldecthm} this last condition is
equivalent to $A$ being decomposable. In this case $A_{\rm
m}=A_{\rm op}$ and \eqref{graphdecM2} follows.
\end{proof}

The characterizations of decomposability of $A$ in the case that
$\mul A$ is closed are  collected in the next result. It shows
that decomposability of $A$ (with $\mul A$ closed) is a natural
counterpart and extension of the notion of closability of operators;
see also Proposition~\ref{operdec}.

\begin{proposition}\label{muldec2}
Let $A$ be a relation in a Hilbert space $\sH$ with $\mul A$ closed.
Then the following statements are equivalent:
\begin{enumerate}\def\labelenumi{\rm (\roman{enumi})}

\item $A$ is decomposable;

\item $A_{\rm m} =A_{\rm op}$;

\item $A_{\rm m} =A_{\rm reg}$;

\item $\ran A_{\rm m} \subset \sH_A$;

\item $\mul A^{**}=\mul A$;

\item $A_{\rm m} $ is a closable operator.
\end{enumerate}
\end{proposition}

\begin{proof}
Since $\mul A$ is closed, the result is obtained from
Theorem~\ref{muldecthm}.
\end{proof}

The next result augments Proposition~\ref{Dom*cl}.

\begin{corollary}\label{mulclosed}
Let $A$ be a relation in a Hilbert space $\sH$ with $\mul A$ closed.
Then the following statements are equivalent:
\begin{enumerate}\def\labelenumi{\rm (\roman{enumi})}

\item $A$ is decomposable with a bounded operator part $A_{\rm
op}$;

\item the operator $A_{\rm m} $ in \eqref{Pm} is bounded.
\end{enumerate}
\end{corollary}

\begin{proof}
(i) $\Longrightarrow$ (ii) Since $A$ is decomposable,
Proposition~\ref{muldec2} shows that $A_{\rm m} =A_{\rm op}$, and
hence (ii) follows.

(ii) $\Longrightarrow$ (i) If $A_{\rm m} $ is bounded, then it is
closable (see Lemma~\ref{first}). Hence, by
Proposition~\ref{muldec2} $A$ is decomposable and $A_{\rm
op}=A_{\rm m} $ is bounded.
\end{proof}

Observe that the condition (ii) in Corollary~\ref{mulclosed} is
essentially weaker than, for instance, the condition (ii) (or (v))
in Proposition~\ref{Dom*cl}; in particular, no equality $A_{\rm
m}=A_{\rm op}$ is assumed in part (ii) of
Corollary~\ref{mulclosed}. In fact, Corollary~\ref{mulclosed} is a
natural extension of the basic Lemma~\ref{first} stating that a
bounded operator is closable.

\subsection{Componentwise decompositions of adjoint relations}

Let $A$ be a relation in a Hilbert space $\sH$, then its adjoint
$A^*$ is automatically a closed (linear) relation. Let
$\sH_{A^*}=\sH \ominus \mul A^*$ and let $P_*$ be the orthogonal
projection from $\sH$ onto $\sH_{A^*}$. Recall the definitions of
the regular part   of $A^*$:
\[
 (A^*)_{\rm reg}\okr \{\,\{f,P_{*} g\};\, \{f,g\}\in A^* \,\},
\]
and of the singular part of $A^*$:
\[
(A^*)_{\rm sing}\okr \{\,\{f,(I-P_{*}) g\};\, \{f,g\}\in A^* \,\}.
\]
Observe that $\dom (A^*)_{\rm reg}=\dom (A^*)_{\rm sing}=\dom
A^*$. The relation $A^*$ admits the canonical operatorwise sum
decomposition
\begin{equation*}\label{**}
  A^*=(A^*)_{\rm reg}+ (A^*)_{\rm sing},
\end{equation*}
where $ (A^*)_{\rm reg}$ is a regular operator in $\sH$ and $
(A^*)_{\rm sing}$ is a singular relation in $\sH$; cf. Theorem
\ref{HSS}. By means of the Hilbert space $\sH_{A^*}$ the following
restriction $(A^*)_{\rm op}$ of $A^*$ is defined by
\begin{equation*}\label{oppa*}
 (A^*)_{\rm op} \okr \{\,\{f,g\}\in A^*;\, g\in\sH_{A^*} \,\}.
\end{equation*}
Observe that $(A^*)_{\rm op}$ can be rewritten in the following
way:
\begin{equation*}\label{oppaa*}
 (A^*)_{\rm op}=A^* \cap (\sH \times \sH_{A^*}).
\end{equation*}
The next decomposition result follows from Theorem \ref{unidec},
Theorem \ref{reglem},  and Corollary \ref{closeddec}.

\begin{theorem}\label{adjoin1}
Let $A$ be a relation in a Hilbert space $\sH$. Then   $(A^*)_s=
(A^*)_{\rm reg}$ is a closed operator and $A^*$ has the following
componentwise decomposition
\begin{equation}\label{einz}
 A^*=(A^*)_{\rm op} \hplus (A^*)_{\rm mul}.
\end{equation}
If there exists a relation $B$ in $\sH$, such that
\begin{equation}\label{einzz}
 A^*=B \hplus (A^*)_{\rm mul}, \quad \ran B\subset \sH_{A^*},
\end{equation}
then the sum in \eqref{einzz} is direct and $B=(A^*)_{\rm op}$ is
a closed operator. In particular, the decomposition of $A^*$ in
\eqref{einzz} is unique.
\end{theorem}

\subsection{Some examples of operators or relations
which are not decomposable}
  Let $A$ be a relation in a Hilbert space $\sH$. If $A$ is
decomposable then Proposition \ref{vasa} shows that both
identities in \eqref{rud0} are satisfied. By Corollary \ref{vasa+}
any operator which is not regular or, equivalently, closable is
not decomposable, as it violates the second identity in
\eqref{rud0} in Proposition \ref{vasa}. This subsection contains
examples which illustrate the absence of decomposability.

The first example provides
a singular operator, which is not regular, but for which the first
identity in \eqref{rud0}   holds. The second example shows a
relation which is not decomposable as it violates the first identity
in \eqref{rud0}, while the second identity in \eqref{rud0} is
satisfied. In the second example there is also a relation for which
both identities \eqref{rud0}  are satisfied, while the relation is
not decomposable. The third example gives a relation $A$ which is
not decomposable and which does not satisfy either of the identities
\eqref{rud0}. Finally, the fourth example shows that a decomposable
relation $A$ whose operator part is bounded, can become
nondecomposable after one-dimensional perturbation of its operator
part.

\begin{example}\label{exam1}
Let $T=T^*$ be an unbounded selfadjoint operator in a Hilbert space
$\sH$ and let $\sH_+\subset \sH \subset \sH_-$ be the rigged Hilbert
spaces associated with $|T|^\half$; cf. \cite{Be}. Denote the
duality between $\sH_+$ and $\sH_-$ by $(f,\varphi)$, $f\in\sH_+$
and $\varphi\in\sH_-$. With elements $\varphi \in
\sH_-\backslash\sH$ and $y_0 \in \sH$ define the following unbounded
operator $A$ in $\sH$:
\begin{equation}\label{Aexam}
 Af \okr (f,\varphi)y_0, \quad f \in \dom A\okr \sH_+=\dom |T|^{\half}.
\end{equation}
Clearly, the operator $A$ is densely defined.  To determine $A^*$
assume that $\{h,k\}\in \sH \times \sH$ satisfies
\[
 0=(k,f)-(h,Af)=(k,f)-(h,(f,\varphi)y_0)=(k-(h,y_0)\varphi,f)
\]
for all $f\in \dom A$; here $(k,f)$ is written as duality. Since
$\dom A=\dom |T|^\half=\sH_+$, the previous identities imply that
$k-(h,y_0)\varphi=0$. Now $k\in \sH$ and $\varphi \in
\sH_-\backslash\sH$, thus $k=0$ and $(h,y_0)=0$. Conversely, if
$\{h,k\} \in \sH \times \sH$, and $(h,y_0)=0$ and $k=0$, then
$\{h,k\} \in A^*$. Therefore, $A^*$ is given by
\begin{equation*}\label{Aexam*}
 A^*=\{\,\{h,0\} \in \sH\times\sH \,;\; (h,y_0)=0 \,\}.
\end{equation*}
Note that $A^*$ is (the graph of) an operator (since $A$ is densely
defined) and that $\dom A^*$ is not dense. Clearly,
\begin{equation}\label{Aexam**}
 A^{**}=\{\,\{f,g\} \in \sH\times\sH \,;\; g \in \lin \{y_0\}
\,\}=\sH \times \lin \{y_0\},
\end{equation}
so that
\begin{equation}\label{mul1}
\mul A^{**}=\lin \{y_0\}.
\end{equation}
The orthogonal projection $P $ onto $\sH_A=(\lin\{y_0\})^\perp$
satisfies $P y_0=0$. Therefore the canonical decomposition
\eqref{HSSSdec}  of $A$ is trivial:
\begin{equation}\label{cano}
A_{\rm reg}=\{\,\{f,0\};\, f \in \dom A\,\}, \quad A=A_{\rm sing}.
\end{equation}
Next, observe that the operator $A_{\rm op}$ in \eqref{oppa} is
given by
\begin{equation}\label{cano1}
A_{\rm op}=\{\,\{f,0\} ;\, f\in \dom A, \, (f,\varphi)=0\,\}.
\end{equation}
It follows from the identity \eqref{mul1} that the operator $A$ is
not decomposable (cf. Corollary \ref{vasa+}); of course, this also
follows by comparing \eqref{cano} and \eqref{cano1}. Since $A_{\rm
op}$ is densely defined it follows that
\[
(A_{\rm op})^{**}=\sH \times \{0\},
\]
and it follows from \eqref{Aexam**} that
\[
 (A^{**})_{\rm op}=\{\,\{f,g\}\in A^{**};\, g\in\sH_{\rm op} \,\}
 =\sH\times \{0\}.
\]
Hence, the first equality in \eqref{rud0} is satisfied, and the
second equality in \eqref{rud0} is not satisfied. Finally, observe
that while the operator $A$ in \eqref{Aexam} is singular and not
decomposable, its closure $A^{**}$ is singular and decomposable (cf.
\eqref{Aexam**} and Proposition \ref{singu}).
\end{example}

\begin{example}\label{exam2}
Let $\sM$ be a dense subspace of the Hilbert space $\sH$ and let $B$
be a relation in $\sH$. Define the relation $A$ in $\sH$ by
\begin{equation}\label{A2}
 A\okr B\hplus (\{0\}\times \sM),
\end{equation}
so that $\dom A=\dom B$ and  $\mul A=\mul B + \sM$. It follows from
\eqref{A2} that
\[
 A^* =B^*\cap (\sM^\perp \times \sH),
\]
and, since $\sM$ is dense,  one obtains
\begin{equation}\label{A2++}
A^{**}=\clos (B^{**} \hplus (\{0\} \times\sH)).
\end{equation}
Observe that $B^{**} \hplus (\{0\}\times \sH)=\dom B^{**} \times
\sH$. Hence, by means of \eqref{impor} it follows from \eqref{A2++}
that
\begin{equation}\label{A2+}
 A^{**}=  \cdom B^{**} \times \sH=\cdom B \times \sH.
\end{equation}
It is clear that
\begin{equation}\label{A2+++}
\cmul A=\sH, \quad \mul A^{**}=\sH.
\end{equation}
In particular, $\sH_{A}=\{0\}$ (see \eqref{oppa1}), so that the
orthogonal projection $P$ is trivial: $P=0$. Therefore the canonical
decomposition \eqref{HSSSdec}  of $A$ is trivial:
\begin{equation}\label{B10}
 A_{\rm reg}=\dom B \times \{0\}, \quad A=A_{\rm sing}.
\end{equation}
Next, observe that $A_{\rm op}$ in \eqref{oppa} is given by
\begin{equation}\label{B1}
 A_{\rm op} =A\cap (\sH \times \{0\})=\ker A\times \{0\}.
\end{equation}
It follows from \eqref{B10} and \eqref{B1} that
\begin{equation}\label{D}
 \text{$A$ decomposable} \quad \Longleftrightarrow \quad \ker
 A=\dom B;
\end{equation}
cf. Proposition~\ref{singu}. The identities \eqref{B1} and
\eqref{A2+} give
\begin{equation}\label{B11}
 (A_{\rm op})^{**} =\cker A\times \{0\},
\end{equation}
and
\begin{equation}\label{B111}
 (A^{**})_{\rm op} =\cdom B \times \{0\}.
\end{equation}
Hence, as to the first equality in \eqref{rud0} of Proposition
\ref{vasa}, a comparison of \eqref{B11} and \eqref{B111} leads to:
\begin{equation}\label{B111+}
 (A_{\rm op})^{**}=(A^{**})_{\rm op} \quad \Longleftrightarrow
\quad \cker A=\cdom B.
\end{equation}
It follows from \eqref{A2+++} that  the second equality  $(A_{\rm
mul})^{**}=(A^{**})_{\rm mul}$ in \eqref{rud0} is satisfied. The
conditions \eqref{D} and \eqref{B111+} will now be reformulated in a
special case.

\begin{lemma}
Let $\sM$ be a dense subspace of the Hilbert space $\sH$ and let $B$
be a relation in $\sH$. Define the relation $A$ by \eqref{A2} and
assume that $\ran B\cap\sM=\{0\}$. Then
\begin{equation}\label{D-}
 \text{$A$ decomposable} \quad \Longleftrightarrow \quad \ker
 B=\dom B \quad \Longleftrightarrow \quad \text{$B$ singular and
 decomposable},
\end{equation}
and
\begin{equation}\label{B-}
 (A_{\rm op})^{**}=(A^{**})_{\rm op} \quad \Longleftrightarrow
 \quad \cker B=\cdom B \quad \Longrightarrow \quad \text{$B$
 singular}.
\end{equation}
\end{lemma}

\begin{proof}
It follows from the definition \eqref{A2} that $\ker B \subset \ker
A$. To show the converse inclusion, let $\{f,0\} \in A$, so that
$\{f,0\}=\{f,g\} +\{0,\varphi\}$ with $\{f,g\} \in B$ and $\varphi
\in \sM$. The condition $\ran B\cap\sM=\{0\}$ implies that $g=0$ and
$\varphi=0$. In particular, $\{f,0\} \in B$. Hence $\ker B=\ker A$.
The first equivalences in \eqref{D-} and \eqref{B-} now follow from
\eqref{D}, \eqref{B111+}. The second equivalence in \eqref{D-} holds
by Proposition~\ref{singu}. Finally, to see the implication in
\eqref{B-} observe that $\cdom B=\cker B\subset \ker B^{**}$, so
that $B^{-1}$ and, thus also, $B$ is singular by
Proposition~\ref{sin}.
\end{proof}

Let $B$ be a nontrivial injective operator which satisfies $\ran B
\cap \sM=\{0\}$. Then $\cdom B \neq \cker B=\{0\}$ and the first
equality in \eqref{rud0} is not satisfied (and $A$ is not
decomposable). For instance, take $B=\span\{h,h\}$ where $h \in \sH$
is nontrivial and $h\not\in \sM$.

Let $B$ be the densely defined operator in Example \ref{exam1}
(see \eqref{Aexam}),
where $y_0 \not\in \sM$ is a nontrivial vector, so that $\ran B \cap
\sM=\{0\}$. Then $B$ satisfies $\cdom B=\cker B$ and the first
equality in \eqref{rud0} is satisfied. Clearly, $B$  does not
satisfy $\ker B=\dom B$, so that $A$ is not decomposable.
\end{example}

\begin{example}
Let $\sM$ be a nonclosed subspace and let $y_0 \in \sH$, and assume
that $\sH=\clos \sM \oplus \lin \{y_0\}$. Let $B$ be a densely
defined singular operator with $\mul B^{**}=\lin \{y_0\}$; cf. e.g.
Example~\ref{exam1}.
Let the bounded operator $C \in \boldsymbol{B}(\sH)$, $C\neq 0$,
have the property that $\ran C \subset \clos \sM \setminus \sM$. The
operator $B+C$ is densely defined with $\dom (B+C)=\dom B$, and
according to \eqref{jan1+}
\[
 (B+C)^*=B^*+C^*, \quad (B+C)^{**}=B^{**}+C,
\]
so that
\[
 \dom (B+C)^{**}=\dom B^{**}, \quad \mul (B+C)^{**}=\lin \{y_0\},
\]
cf. \eqref{dommul}.
Define the relation $A$ in $\sH$ by
\[
 A \okr (B+C) \hplus (\{0\}\times \sM),
\]
so that $A^*=(B+C)^* \cap (\sM^\perp \times \sH)$, which leads to
\begin{equation}\label{one}
 A^{**} =\cspan \left\{ (B+C)^{**} \hplus (\{0\}\times \clos \sM)
 \right\}.
\end{equation}
Observe that $(B+C)^{**}=(B+C)^{**}\hplus (\{0\}\times \lin\{y_0\})$
and since $\clos \sM \oplus \lin \{y_0\}=\sH$, one concludes that
\begin{equation}\label{two}
 (B+C)^{**} \hplus (\{0\}\times \clos \sM)
 =(B+C)^{**}\hplus (\{0\}\times \sH)=\dom B^{**} \times \sH.
\end{equation}
 A combination of \eqref{one} and \eqref{two} leads to
\begin{equation}\label{three}
 A^{**}=\cdom B^{**} \times \sH=\sH \times \sH,
\end{equation}
since $\dom B$ is dense in $\sH$. In particular, $\mul A^{**}=\sH$,
so that $\sH_{A}=\{0\}$ (see \eqref{oppa1}) and the orthogonal
projection $P$ is trivial: $P=0$. Therefore the canonical
decomposition \eqref{HSSSdec}  of $A$ is trivial:
\begin{equation}\label{B10-}
 A_{\rm reg}=\dom B \times \{0\}, \quad A=A_{\rm sing}.
\end{equation}
Next, observe that $A_{\rm op}$ in \eqref{oppa} is given by
$A_{\rm op}=A \cap (\sH \times \{0\})$, so that
\begin{equation}\label{B10+}
 A_{\rm op}=(\ker B \cap \ker C)\times \{0\},
\end{equation}
since $\ran (B+C)\cap\sM=\{0\}$ and $\ran B\cap\ran C=\{0\}$.
Therefore, a comparison of \eqref{B10-} and \eqref{B10+} shows that
the relation $A$ is not decomposable; already $\dom B\neq \ker B$
since by construction $B$ is an operator with $\ran B=\lin \{y_0\}$.
Furthermore, note that \eqref{B10+} implies that
\begin{equation}\label{three1}
 (A_{\rm op})^{**}=\clos (\ker B \cap \ker C)\times \{0\},
\end{equation}
while it follows from \eqref{three} that
\begin{equation}\label{three2}
 (A^{**})_{\rm op}=\sH \times \{0\}.
\end{equation}
A comparison of \eqref{three1} and \eqref{three2} shows that the
first identity of \eqref{rud0} is not satisfied, since $\ker C\neq
\sH$ by the assumption $C\neq 0$. Finally, the identities $\mul
A^{**}=\sH$ and $\mul A=\sM$ and $\cmul A=\clos \sM$ imply that the
second identity of \eqref{rud0} is not satisfied, cf.
\eqref{mmuull}.

Hence, the relation $A$ in this example is not decomposable and, moreover,
the two identities \eqref{rud0} in Proposition \ref{vasa} are not satisfed.
Another way to construct such an example is to take the orthogonal
sum of the relations in Example \ref{exam1} and Example
\ref{exam2}.
\end{example}

The next example shows that a decomposable relation $A$ whose operator part is bounded,
can become nondecomposable after one-dimensional perturbation of its operator part.

\begin{example}\label{exam0a}
Let $B$ be a bounded operator in $\sH$ and let $\sM\subset
\sH\ominus\cran B$ be a nonclosed subspace. Define the relation $A$ by $A=B\hplus
(\{0\}\times\sM)$, so that
\[
 A^{**}=B^{**} \hplus (\{0\} \times \clos \sM).
\]
The relation $A$ is decomposable with $A_{\rm reg}=B$
and $A_{\rm mul}=\{0\}\times\sM$. Let $f_0\in\cdom B$ and let $e\in
(\clos \sM)\setminus\sM$. Define $B_e f=Bf+(f,f_0)e$, $f\in\dom B$
and define the relation $A_e$ by $A_e=B_e\hplus (\{0\}\times\sM)$, so that
\[
 A_e^{**}=B_e^{**} \hplus (\{0\} \times \clos \sM).
\]
Observe that  $\mul A_e=\mul A=\sM$
and $\mul A_e^{**}=\mul A^{**}=\clos\sM$. However,
$\ran(I-P)A_e=\span\{e\}+\sM$ so that $\ran(I-P)A_e\not\subset\mul
A_e$, and thus $A_e$ is not decomposable by Theorem~\ref{reglem}. In
this case $A_e$ still satisfies the equalities in \eqref{rud0}:
\[
 ((A_e)_{\rm
op})^{**}=(B\upharpoonright_{f_0^\perp})^{**}=(B_e^{**}\hplus
(\{0\}\times\clos\sM))_{\rm op}=((A_e)^{**})_{\rm op}
\]
and
\[
 ((A_e)_{\rm mul})^{**}=\{0\}\times\clos \sM=((A_e)^{**})_{\rm mul}.
\]
\end{example}

\section{Orthogonal componentwise decompositions of relations}

Let $A$ be a decomposable relation in a Hilbert space $\sH$, so that
it has a componentwise sum decomposition as in \eqref{graphdec}.
Furthermore, the adjoint $A^*$, being closed, has a componentwise
decomposition as in \eqref{einz}.  Necessary and sufficient
conditions for these componentwise decompositions to be orthogonal
will be given.

\subsection{Orthogonality for componentwise sum decompositions
of relations}

For any relation $A$ in a Hilbert space $\sH$ the identities
\[
(A_{\rm mul})^*=(\mul A)^\perp \times \sH, \quad (A_{\rm
mul})^{**}=\{0\} \times \cmul A,
\]
are valid, where the adjoint is, as usual, with respect to the
Hilbert space $\sH$. The last identity is concerned with taking
closures, which are automatically with respect to the Hilbert
space $\mul A^{**}$. It is also useful to consider the adjoint of
$A_{\rm mul}$ as a relation in the Hilbert space $\mul A^{**}$.
The proof of the following lemma is straightforward.

\begin{lemma}\label{ADj}
Let $A$ be a relation in a Hilbert space $\sH$. The adjoint of the
relation $A_{\rm mul}=\{0\}\times \mul A$ in the Hilbert space
$\mul A^{**}$ is given by
\[
 (A_{\rm mul})^*=(\mul A^{**} \ominus \mul A) \times \mul A^{**}.
\]
In particular,
\[
 \cmul A=\mul A^{**}\quad \Longleftrightarrow \quad (A_{\rm
mul})^*=(A_{\rm mul})^{**},
\]
and
\[
 \mul A=\mul A^{**} \quad \Longleftrightarrow \quad (A_{\rm
mul})^*= A_{\rm mul}.
\]
\end{lemma}

Hence, the relation $A_{\rm mul}$ is essentially selfadjoint in the
Hilbert space $\mul A^{**}$ if and only if $\cmul A=\mul A^{**}$,
and the relation $A_{\rm mul}$ is selfadjoint in the Hilbert space
$\mul A$ if and only if $\mul A=\mul A^{**}$. The following
proposition is a further specification of the results in Proposition
\ref{vasa}. Recall that $\sH_A=\sH_{A^{**}}$; it will be shown that
the decompositions \eqref{graphdec} and \eqref{grdec**} are
orthogonal with respect to the splitting $\sH=\sH_A \oplus \mul
A^{**}$, simultaneously.

\begin{proposition}\label{orthchar}
Let $A$ be a decomposable relation in a Hilbert space $\sH$.   Then
the componentwise sum decomposition \eqref{graphdec} of $A$ is
orthogonal
\begin{equation}\label{deco00}
 A=A_{\rm op} \hoplus A_{\rm mul}
\end{equation}
if and only if
\begin{equation}\label{14.4.1}
\dom A \subset \cdom A^* \quad \mbox{or, equivalently,} \quad \mul
A^{**} \subset \mul A^*.
\end{equation}
In this case $A_{\rm mul}$ is essentially selfadjoint in $\mul
A^{**}$. Moreover, in this case the componentwise sum
decomposition \eqref{grdec**} of $A^{**}$ is automatically
orthogonal
\begin{equation}\label{deco00a}
 A^{**}=(A^{**})_{\rm op} \hoplus (A^{**})_{\rm mul},
\end{equation}
and $(A^{**})_{\rm mul}$ is selfadjoint in $\mul A^{**}$.
\end{proposition}

\begin{proof}
It is assumed that $A$ is decomposable, i.e., $A=A_{\rm op} \hplus
A_{\rm mul}$. Clearly, the subspaces $\ran A_{\rm op}$ and $\mul
A$ are orthogonal, cf. \eqref{oppaa}. Hence, the componentwise sum
decomposition is orthogonal if and only if the condition $\dom
A_{\rm op} \subset \sH \ominus \mul A^{**}$ is satisfied. Note
that by Theorem \ref{reglem} this last condition is equivalent to
\eqref{14.4.1}, cf. Lemma \ref{easy}. Furthermore, the
decomposability of $A$ implies that $\cmul A=\mul A^{**}$, cf.
Proposition \ref{vasa}. Lemma \ref{ADj} now guarantees that
$A_{\rm mul}$ is essentially selfadjoint in $\mul A^{**}$. It is
clear that $(A^{**})_{\rm mul}$ is selfadjoint in $\mul A^{**}$;
cf. Lemma \ref{ADj}.
\end{proof}

\begin{corollary} \label{william4mary}
Let $A$ be a decomposable relation in a Hilbert space $\sH$. Then
the following statements are equivalent:
\begin{enumerate}\def\labelenumi{\rm (\roman{enumi})}
\item \text{$\mul A^{**}\subset\mul A^*$ and $(A^{**})_{\rm
op}\in\boldsymbol B(\cdom A^*)$}; \item $\dom A^{**}=\cdom A^*$.
\end{enumerate}
\end{corollary}

\begin{proof}
(i) $\Longrightarrow$ (ii) If $\mul A^{**}\subset\mul A^*$, then
\eqref{deco00a} holds by Proposition \ref{orthchar}. Moreover, if
$(A^{**})_{\rm op}\in \boldsymbol B(\cdom A^*)$, then
\begin{equation*}
   \dom A^{**}=\dom (A^{**})_{\rm op}=\cdom A^*.
\end{equation*}

(ii) $\Longrightarrow$ (i) If $\dom A^{**}=\cdom A^*$, then $\mul
A^{**}=\mul A^*$; cf. Lemma \ref{easy}. Hence \eqref{deco00a}
holds by Proposition \ref{orthchar}. Furthermore,
\begin{equation*}
   \dom (A^{**})_{\rm op}=\dom A^{**}=\cdom A^*.
\end{equation*}
Hence the closed operator $(A^{**})_{\rm op}$ is defined on all of
$\dom A^*$, so that it is bounded by the closed graph theorem.
\end{proof}

Let $A$ be a decomposable relation. Then it has already been shown
in Proposition \ref{vasa}   that
\[
 (A_{\rm op})^{**}=(A^{**})_{\rm op}, \quad (A_{\rm
 mul})^{**}=(A^{**})_{\rm mul}.
\]
When $A$ is decomposable and satisfies \eqref{14.4.1}, then these
equalities follow now also from a comparison between
\eqref{deco00} and \eqref{deco00a}. Under the same circumstances,
$A$ is closed if and only if $A_{\rm op}$ and $A_{\rm mul}$ are
closed; and $A_{\rm op}$ is densely defined if and only if $\cdom
A = \cdom A^*$; which is equivalent to $\mul A^{**}=\mul A^*$.

\begin{proposition}\label{orthchar+}
Let $A$ be a  relation in a Hilbert space $\sH$. Assume that there
is a closable operator $B$ (in $\sH_{A}$) such that
\begin{equation}\label{odec0}
 A=B \hoplus A_{\rm mul},
\end{equation}
then $B$ coincides with $A_{\rm op}$. In particular, the relation
$A$ is decomposable and satisfies the condition \eqref{14.4.1}.
\end{proposition}

\begin{proof}
The assumption \eqref{odec0} implies that the condition in Theorem
\ref{unidec} is satisfied, so that $B=A_{\rm op}$. In particular,
it follows that $\dom A_{\rm op}=\dom A$, so that $A$ is
decomposable by Theorem \ref{reglem}. Since $B$ is an operator in
$\sH_{\rm op}$ it is clear that $\dom A=\dom A_{\rm op}=\dom B
\subset \sH_{\rm A}=\cdom A^*$, which leads to \eqref{14.4.1}.
\end{proof}

A combination of Propositions \ref{orthchar} and \ref{orthchar+}
leads to the following corollary.

\begin{corollary}
Let $A$ be a relation in a Hilbert space $\sH$. Then $A$ has an
orthogonal decomposition of the form \eqref{deco00} if and only if
$A$ is decomposable and satisfies \eqref{14.4.1}.
\end{corollary}

\begin{corollary}\label{multiA**}
Let $A$ be a relation in a Hilbert space $\sH$ which satisfies
$\mul A=\mul A^{**}$, so that $A$ is decomposable and $A_{\rm
mul}$ is selfadjoint in $\mul A^{**}$. Then $A$ admits the
orthogonal composition \eqref{deco00} if and only if $\mul
A\subset \mul A^{*}$.
\end{corollary}

\begin{proof}
The condition $\mul A=\mul A^{**}$ implies that $A$ is decomposable;
cf. Corollary \ref{muldec} and Lemma \ref{ADj}. Furthermore, the
condition \eqref{14.4.1} in Proposition \ref{orthchar} is now
equivalent to $\mul A\subset \mul A^*$.
\end{proof}

If $A$ is a closed relation in a Hilbert space $\sH$, then $A$ is
decomposable and $A_{\rm mul}$ is selfadjoint in $\mul A^{**}$; cf.
Corollary \ref{closeddec} and Lemma \ref{ADj}. Hence $A$ admits the
orthogonal composition \eqref{deco00} if and only if $\mul A\subset
\mul A^{*}$ (see Corollary \ref{multiA**}).

\subsection{Orthogonality for componentwise sum decompositions of adjoint
relations}

Let $A$ be a relation in a Hilbert space $\sH$. Since the relation
$A^*$ is closed it is decomposable and has the componentwise
decomposition \eqref{einz}, cf. Theorem \ref{adjoin1}.  The adjoint
of the relation $(A^*)_{\rm mul}$ in the Hilbert space $\mul A^{*}$
is given by
\[
 ((A^*)_{\rm mul})^*=\{0\} \times \mul A^{*}=(A^*)_{\rm mul},
\]
and the relation $(A^*)_{\rm mul}$ is selfadjoint in the Hilbert
space $\mul A^*$, cf. Lemma \ref{ADj}. The following result is
obtained by combining Lemma \ref{easy}, Corollary \ref{closeddec},
Proposition \ref{orthchar}, and Proposition \ref{orthchar+}. The
orthogonal componentwise decomposition is with respect to the
orthogonal splitting $\sH=\sH_{A^*} \oplus \mul A^*$.

\begin{proposition}\label{how}
Let $A$ be a relation in a Hilbert space $\sH$. Then the
componentwise sum decomposition of $A^*$ in \eqref{einz} is
orthogonal
\begin{equation}\label{how1}
 A^*=(A^*)_{\rm op} \hoplus (A^*)_{\rm mul},
\end{equation}
if and only if
\begin{equation}\label{how2}
 \mul A^* \subset \mul A^{**}.
\end{equation}
\end{proposition}

\begin{corollary} \label{mary1}
Let $A$ be a relation in a Hilbert space $\sH$. Then the following
statements are equivalent:
\begin{enumerate}\def\labelenumi{\rm (\roman{enumi})}
\item \text{$\mul A^{*}\subset \mul A^{**}$ and $(A^*)_{\rm op}\in
\boldsymbol B(\cdom A)$;} \item $\cdom A=\dom A^*$.
\end{enumerate}
\end{corollary}

\begin{proof}
(i) $\Longrightarrow$ (ii) If $\mul A^{*}\subset A^{**}$, then
\eqref{how1} holds by Proposition \ref{how}. Moreover, if
$(A^*)_{\rm op}\in \boldsymbol B(\cdom A)$, then
\begin{equation*}
\dom A^*=\dom (A^*)_{\rm op} =\cdom A.
\end{equation*}

(ii) $\Longrightarrow$ (i) If $\cdom A=\dom A^*$, then $\mul
A^*=\mul A^{**}$; cf. Lemma \ref{easy}. Hence \eqref{how1} holds
by Proposition \ref{how}. Furthermore, the closed operator
$(A^*)_{\rm op}$ is defined on all of $\cdom A$, so that it is
bounded by the closed graph theorem.
\end{proof}

Another way to decompose $A^*$ is to assume that $A$ has an
orthogonal componentwise decomposition as in \eqref{deco00}.  Hence,
the following orthogonal componentwise decomposition is with respect
to the orthogonal splitting $\sH=\sH_A \oplus \mul A^{**}$.

\begin{proposition}\label{vaaas}
Let $A$ be a decomposable relation in a Hilbert space $\sH$ which
satisfies \eqref{14.4.1}. Then $A^*$ has the orthogonal
componentwise decomposition
\begin{equation}\label{deco000*}
 A^*=(A_{\rm op})^* \hoplus (A_{\rm mul})^*,
\end{equation}
where $(A_{\rm op})^*$ and $(A_{\rm mul})^*$ stand for the
adjoints of $A_{\rm op}$ and $A_{\rm mul}$ in $\sH_{\rm op}$ and
$\mul A^{**}$, respectively. Moreover, $ (A_{\rm mul})^*=\{0\}
\times \mul A^{**}$ is selfadjoint in $\mul A^{**}$.
\end{proposition}

\begin{proof}
Taking adjoints in \eqref{deco00} gives \eqref{deco000*}. It follows
from  Proposition \ref{orthchar} that $A_{\rm mul}$ is essentially
selfadjoint in $\mul A^{**}$ or, equivalently, that $ (A_{\rm
mul})^*$ is selfadjoint in $\mul A^{**}$, cf. Lemma \ref{ADj}.
\end{proof}

Since the closable operator $A_{\rm op}$ need not be densely
defined in $\sH_A$ its adjoint $(A_{\rm op})^*$ is a relation with
multivalued part $\mul (A_{\rm op})^*$. The following result is a
direct consequence of \eqref{deco000*}.

\begin{corollary}\label{vaaaas}
Let $A$ be a decomposable  relation in a Hilbert space $\sH$ which
satisfies \eqref{14.4.1}. Then
\begin{equation*}\label{deco00*}
 \mul A^*\ominus \mul A^{**}=\mul (A_{\rm op})^*,
\end{equation*}
so that
\begin{equation*}\label{deco00*+}
 (A^*)_{\rm mul}=\{0\}\times (\mul (A_{\rm op})^*\oplus \cmul A).
\end{equation*}
\end{corollary}

A combination of Propositions \ref{orthchar} and \ref{vaaaas}
leads to a decomposition result for formally domain tight
relations.

\begin{proposition}\label{mabel}
Let $A$ be a decomposable relation, which is formally domain
tight. Then $A$ admits the orthogonal decomposition
\eqref{deco00}, where $A_{\rm op}$ is a formally domain tight
operator in $\sH_{\rm op}$ and $A_{\rm mul}$ is essentially
selfadjoint in $\mul A^{**}$.
\end{proposition}

\begin{proof}
Since $A$ is formally domain tight, it follows that $\mul A^{**}
\subset \mul A^*$. Since $A$ is assumed to be also decomposable, the
conditions of Proposition \ref{orthchar} are satisfied. Hence the
orthogonal decomposition of $A$ in \eqref{deco00} and the orthogonal
decomposition of $A^*$ in \eqref{deco000*} are valid.  Recall that
$A_{\rm mul}=\{0\} \times \mul A$ and $(A_{\rm mul})^*=\{0\} \times
\mul A^{**}$, cf. Proposition \ref{vaaas}. Hence, it follows from
\eqref{deco00} and \eqref{deco000*} that
\[
 \dom A_{\rm op}=\dom A \subset \dom A^*=\dom (A_{\rm op})^*.
\]
In other words, the operator $A_{\rm op}$ is formally domain tight
in $\sH_{\rm op}$.
\end{proof}

Let $A$ be a relation in a Hilbert space $\sH$, which satisfies
$\mul A^{**}=\mul A^*$. Then the orthogonal splitting $\sH=\sH_A
\oplus \mul A^{**}$ generated by $\mul A^{**}$ coincides with the
orthogonal splitting $\sH=\sH_{A^*} \oplus \mul A^{*}$ generated
by $\mul A^{*}$. Hence, in this case the orthogonal decompositions
\eqref{deco00}, \eqref{deco000*}, and \eqref{how1} (cf.
\eqref{how2}) are with respect to the same splitting.

\begin{proposition}\label{le0}
Let $A$ be a decomposable relation in a Hilbert space $\sH$, which
satisfies $\mul A^{**}=\mul A^*$. Then $A$ admits the orthogonal
decomposition \eqref{deco00} where $A_{\rm op}$ is a densely
defined operator in $\sH_{\rm op}$ and $A_{\rm mul}$ is an
essentially selfadjoint relation in $\mul A^{**}$. Moreover,
\begin{equation}\label{deco1B*}
 (A_{\rm op})^*=(A^*)_{\rm op}.
\end{equation}
\end{proposition}

\begin{proof}
It follows from the condition $\mul A^{**}=\mul A^*$ that the
identity \eqref{how1} is valid. Since $A$ is assumed to be
decomposable, the condition $\mul A^{**}=\mul A^*$ also implies
that the identity \eqref{deco00} holds. It follows from Corollary
\ref{vaaaas} that $A_{\rm op}$ is a densely defined operator in
$\sH_{\rm op}$. The identity \eqref{deco00} itself shows that the
identity \eqref{deco000*} holds. Furthermore, the condition $\mul
A^{**}=\mul A^*$ implies that both decompositions \eqref{how1} and
\eqref{deco000*} are relative to the same orthogonal splitting of
the Hilbert space $\sH$. Therefore, the identity \eqref{deco1B*}
is immediate.
\end{proof}

A combination of Propositions \ref{orthchar} and \ref{le0} leads
to a decomposition result for domain tight relations.

\begin{proposition}\label{le1}
Let $A$ be a decomposable relation in a Hilbert space $\sH$, which
is domain tight. Then $A$ admits the orthogonal decomposition
\eqref{deco00} where $A_{\rm op}$ is a densely defined domain
tight operator in $\sH_{A}$ and $A_{\rm mul}$ is essentially
selfadjoint in $\mul A^{**}$.
\end{proposition}

\begin{proof}
If $A$ is a domain tight relation, so that $\dom A=\dom A^*$, then
$\mul A^{**}=\mul A^*$ and Proposition \ref{le0} applies. It follows
from the decompositions \eqref{deco00} and \eqref{deco000*}, that
\[
 \dom (A_{\rm op})^*=\dom (A^*)_{\rm op}=\dom A^*=\dom A=\dom
 A_{\rm op},
\]
which shows that $A_{\rm op}$ is domain tight in $\sH_{A}$.
\end{proof}

The relations $A$ which are domain tight, i.e., $\dom A=\dom A^*$,
and which satisfy the additional condition $\mul A=\mul A^*$,
can be characterized in terms of orthogonal decompositions.

\begin{proposition}\label{mulba}
Let $A$ be a relation in a Hilbert space $\sH$. Then $A$ is domain
tight and $\mul A = \mul A^*$ if and only if $A=B \hoplus A_{\rm
mul}$ where $B$ is a densely defined domain tight (closable)
operator in $\sH_A$ and $A_{\mul}$ is selfadjoint in $\mul
A^{**}$. In this case $B =A_{\rm op}$.
\end{proposition}

\begin{proof}
($\Longrightarrow$) Assume that $A$ is domain tight and that $\mul
A = \mul A^*$. Then it follows that $\mul A^{**}=\mul A$. Hence,
$A$ is decomposable by Corollary \ref{muldec} and $A_{\rm mul}$ is
selfadjoint in $\mul A^{**}$ by Lemma \ref{ADj}. By Proposition
\ref{le1} it follows that $A_{\rm op}$ is a densely defined domain
tight operator in $\sH_A$. Furthermore, $A_{\rm op}$ is closable;
which is clear from the fact that $A$ is decomposable, but also
from the fact that $A_{\rm op}$ is domain tight and densely
defined. According to Proposition \ref{le1} the relation $A$
decomposes as $A=A_{\rm op} \hoplus A_{\mul}$.

($\Longleftarrow$) Assume that $A=B \hoplus A_{\rm mul}$ where $B$
is a densely defied domain tight operator in $\sH_A$ and
$A_{\mul}$ is selfadjoint in $\mul A^{**}$. Then $A^*=B^*\hoplus
A_{\rm mul}$, so that $\dom A=\dom B = \dom B^* =\dom A^*$, and
$A$ is domain tight. The condition that $A_{\rm mul}$ is
selfadjoint in $\mul A^{**}$ implies that $\mul A=\mul A^{**}$,
cf. Lemma \ref{ADj}. Since $B$ is densely defined and domain
tight, it follows that $B$ is a closable operator. Hence, by
Proposition \ref{orthchar+}, the identity $B=A_{\rm op}$ is
established.
\end{proof}

\subsection{Some classes of relations with orthogonal componentwise
decompositions}

This subsection describes orthogonal componentwise decompositions
for some classes of relations  described via the numerical range and
for some subclasses of domain tight relations.

Let $A$ be a decomposable relation in a Hilbert space $\sH$ and
assume that $\mul A^{**} \subset \mul A^*$. Then
\begin{equation}\label{nume}
\cW(A)=\cW(A_{\rm op}).
\end{equation}
To see this, note that Theorem \ref{reglem} shows that $A_{\rm
reg}=A_{\rm op}$, and then apply Remark \ref{numi}. Now some
consequences of the assumption $\cW(A) \neq \dC$ are listed.

\begin{proposition}\label{num}
Let $A$ be a decomposable relation in a Hilbert space $\sH$ such
that $\cW(A) \neq \dC$. Then the relation $A$ admits the
orthogonal decomposition \eqref{deco00}, $A_{\rm mul}$ is
essentially selfadjoint in $\mul A^{**}$, and $\cW(A_{\rm
op})=\cW(A)$. Moreover, if $\rho(A) \ne \emptyset$, then $A_{\rm
op}$ is a closed densely defined operator in $\sH_{\rm op}$,
$A_{\rm mul}$ is selfadjoint in $\mul A^{**}$, and $\rho(A_{\rm
op})\ne \emptyset$.
\end{proposition}

\begin{proof}
By Lemma \ref{numRange2} the condition $\cW(A) \neq \dC$ implies
that $\mul A \subset \mul A^*$, and thus also $\cmul A \subset
\mul A^*$. By Proposition \ref{vasa} the condition that $A$ is
decomposable, implies that $\cmul A=\mul A^{**}$. Therefore the
inclusion $\mul A^{**} \subset \mul A^*$ is valid. Since $A$ is
assumed to be decomposable, Proposition \ref{orthchar} may be
applied. The identity $\cW(A_{\rm op})=\cW(A)$ follows from
\eqref{nume}.

If $\rho(A) \ne \emptyset$, then Lemma \ref{max2} shows that $A$
is closed and that $\mul A^*=\mul A$. Hence, Proposition \ref{le0}
applies, so that $A_{\rm op}$ is densely defined closed operator
in $\sH_{A}$ and $\mul A$ is closed. The decomposition $A=A_{\rm
op} \hoplus A_{\rm mul}$, where $A_{\rm mul}$ is selfadjoint in
$\mul A^{**}$, shows that $A$ and $A_{\rm op}$ have the same
resolvent set.
\end{proof}

Let $A$ be a relation in a Hilbert space $\sH$. Then $A$ is
symmetric if and only if $\cW(A) \subset \dR$. A relation $A$ is
said to be \textit{dissipative}\index{relation!dissipative} if
$\cW(A)$ is a subset of the upper halfplane:
\begin{equation*}\label{diss}
 \IM (f',f) \ge 0, \quad \{ f,f'\} \in A,
\end{equation*}
and a relation $A$ is said to be
\textit{accretive}\index{relation!accretive} if $\cW(A)$ is a
subset of the right halfplane:
\begin{equation*}\label{accr}
 \RE (f',f) \ge 0, \quad \{ f,f'\} \in A.
\end{equation*}
A relation $A$ is said to be \textit{sectorial with vertex at the
origin and semiangle} $\alpha$\index{relation!sectorial with
vertex at the origin and semiangle $\alpha$}, $\alpha \in (0,\pi /
2)$, if $\cW(A)$ is a subset of the corresponding sector in the
right halfplane:
\begin{equation}\label{sect}
 (\tan \alpha) \RE (f',f) \ge | \IM (f',f) |, \quad  \{ f,f'\} \in
A,
\end{equation}
cf. \cite{Ar2}, \cite{Yury}, \cite{HSSW??}, \cite{Sand}. A
relation $A$ is said to be
\textit{nonnegative}\index{relation!nonnegative}, if $\cW(A)$ is a
subset of $[0,\infty)$. In each of these cases the closure gives
rise to a similar inequality. Hence, if the relation $A$ belongs
to one of the above classes, it may be assumed in addition that
$A$ is closed. Therefore Proposition \ref{num} may be applied and
the operator part $A_{\rm op}$ in the orthogonal decomposition
\eqref{deco00} belongs to the same class as the original relation
$A$.

In each of these cases the relation $A$ is said to be maximal with
respect to the indicated property if the complement of $\clos
\cW(A)$ (or one of its components) belongs to the resolvent set so that
$\rho(A)$. It can be shown that maximality is equivalent to the
absence of nontrivial (relation) extensions with the same property;
cf. \cite{Ka}, \cite{Phillips}, \cite{DS}, \cite{HSSW??}.

\begin{corollary}\label{cad}
Let $A$ be a maximal symmetric (dissipative, accretive, sectorial,
nonnegative) relation in a Hilbert space $\sH$. Then $A$ admits an
orthogonal decomposition of the form $A=A_{\rm op} \hoplus A_{\rm
mul}$, where $A_{\rm op}$ is a closed, densely defined, maximal
symmetric (dissipative, accretive, sectorial, nonnegative)
operator in the Hilbert space $\sH_{A}$ and $A_{\rm mul}$ is a
selfadjoint relation in $\mul A^{**}$.
\end{corollary}

The result for maximal symmetric relations can also be seen as a
consequence of Proposition \ref{mabel}, since symmetric relations
are formally domain tight. Selfadjoint and normal relations are
domain tight and there is a decomposition result for them
corresponding to Corollary \ref{cad}, as an application of
Proposition \ref{le0}; see \cite{Co73} and \cite{HSSz??} for further
details.

\begin{corollary}\label{cad1}
Let $A$ be a selfadjoint (normal) relation in a Hilbert space
$\sH$. Then $A$ admits an orthogonal decomposition of the form
$A=A_{\rm op} \hoplus A_{\rm mul}$, where $A_{\rm op}$ is a
selfadjoint (normal) operator in the Hilbert space $\sH_{A}$ and
$A_{\rm mul}$ is a selfadjoint relation in $\mul A^{**}$.
\end{corollary}

Recall that selfadjoint and normal operators are automatically
densely defined; cf. \eqref{eq1++}.

\section{Cartesian decompositions of relations}

In this section the notions of real and imaginary parts of a
relation in a Hilbert space are confronted with the notion of a
Cartesian decomposition.

\subsection{Real and imaginary parts of relations}

Let $A$ be a relation in a Hilbert space $\sH$. The {\it real
part\,}\index{part of relation!real $\RE A$} $\RE A$ and the {\it
imaginary part\,}\index{part of relation!imaginary $\IM A$} $\IM
A$ of $A$ are defined by
\begin{equation}\label{re}
\RE A\okr\frac{1}{2}(A+A^*) =\left\{\, \left\{f, \frac{f'+f''}{2}
\right\}\,;\; \{f,f'\} \in A,\, \{f,f''\} \in A^* \, \right\},
\end{equation}
and
\begin{equation}
\label{im} \IM A\okr\frac{1}{2\I}(A-A^*) =\left\{\, \left\{f,
\frac{f'-f''}{2\I} \right\}\,;\; \{f,f'\} \in A,\, \{f,f''\} \in
A^* \, \right\},
\end{equation}
with the operatorwise sums defined as in \eqref{jan1}. It is clear
from the definitions that
\begin{equation}\label{14.4.2}
\left\{ \begin{array}{l}\dom \RE A=\dom \IM A =\dom A \cap \dom A^*,
\\\dom \RE A^*=\dom \IM A^* =\dom A^{**} \cap \dom A^*.
\end{array}\right.
\end{equation}
The real and imaginary parts of $A$ are connected by
\begin{equation}
\label{imre} \RE (\I A)=-\IM A,\quad \IM (\I A)=\RE A.
\end{equation}
In what follows the relations $\RE A \pm \I \IM A$ and their
connections to the original relation $A$ will be studied.

\begin{proposition}\label{reim}
Let $A$ be a relation in a Hilbert space $\sH$. Then
\begin{enumerate}\def\labelenumi{\rm (\roman{enumi})}
\item
$\RE A \subset \RE A^*=\RE A^{**}\subset (\RE
A)^*$ and $\IM A\subset - \IM A^*= \IM A^{**} \subset (\IM A)^*$;

\item
if $A$ is closed, then $\RE A =\RE A^*$ and $\IM A
=-\IM A^*$;

\item
$\mul \RE A=\mul \IM A =\mul A + \mul A^*$ and if, in addition, $A$
is formally domain tight, then  $\mul \RE A=\mul \IM A =\mul A^*$.
\end{enumerate}
\end{proposition}

\begin{proof}
(i) Since $A \subset A^{**}$, it follows from \eqref{jan1+} that
\[
 \half(A+A^*) \subset \half(A^{**}+A^*) \subset
 \left(\half(A+A^*)\right)^*,
\]
and
\[
 \frac{1}{2\I}(A-A^*) \subset -\frac{1}{2\I}(A^{*}-A^{**}) \subset
 \left(\frac{1}{2\I} (A-A^*) \right)^*.
\]
The assertions concerning $\RE A$ and $\IM A$ are now clear.

(ii) Here $A=A^{**}$ and thus the stated equalities are clear from
\eqref{re} and \eqref{im}.

(iii) The first assertion is immediate from \eqref{re} and
\eqref{im}. If $A$ is formally domain tight, then it follows from
\eqref{eq1+} that $\mul A \subset \mul A^*$ and thus $\mul A+\mul
A^*= \mul A^*$, which implies the second assertion.
\end{proof}

The real and the imaginary parts $\RE A$ and $\IM A$ of a relation
$A$ are symmetric relations, due to Proposition \ref{reim}.  They
are defined in terms of operatorwise sums involving $A$ and $A^*$.
There are also connections with the componentwise sum $A \hplus
A^*$.

\begin{proposition}\label{eenv}
Let $A$ be a linear relation in a Hilbert space $\sH$. Then
\begin{enumerate}\def\labelenumi{\rm (\roman{enumi})}
\item
$\RE A \subset A \hplus A^*$ and $\IM A \subset A
\hplus A^*$;

\item
$\ran (\RE A)=\mul (A \hplus A^*)$ and $\ran (\IM
A)=\mul (A \hplus A^*)$;

\item
$\RE A \pm \I \IM A \subset \RE A \hplus (\{0\}
\times \ran \IM A) \subset A \hplus A^*$;

\item
$\IM A \pm \I \RE A \subset \IM A \hplus (\{0\}
\times \ran \RE A) \subset A \hplus A^*$.
\end{enumerate}
\end{proposition}

\begin{proof}
(i) The first inclusion follows from \eqref{re} and
\[
 2 \left\{f, \frac{f'+f''}{2} \right\}=\{f,f'\}+\{f,f''\} \in A
 \hplus A^*, \quad \{f,f'\} \in A, \quad \{f,f''\} \in A^*.
\]
The second inclusion  can be shown similarly.

(ii) The second inclusion will be shown. Let $\{0,g\} \in A \hplus
A^*$, then
\[
 \{0,g\}=\{f,f'\}-\{f,f''\}, \quad \{f,f'\} \in A, \quad \{f,
 f''\} \in A^*,
\]
so that
\[
 \left\{f,\frac{g}{2\I} \right\}=\left\{f, \frac{f'-f''}{2\I}
 \right\} \in \IM A.
\]
Hence $\mul (A \hplus A^*) \subset \ran (\IM A)$. The reverse
inclusion follows immediately from \eqref{im}. This proves the
second identity. The first identity is now obtained as follows:
\[
 \ran (\RE A)=\ran (\IM iA)=\mul (iA \hplus (iA)^*)=\mul (-A
\hplus A^*).
\]

(iii) Let $\{f,\varphi\pm \I \psi\} \in \RE A \pm\I \IM A$ with
$\{f, \varphi\} \in \RE A$ and $\{f,\psi\} \in \IM A$. Then clearly
\[
  \{f,\varphi \pm \I \psi\}=\{f,\varphi\} \hplus \{0,\pm\I \psi\}
  \in \RE A \hplus (\{0\} \times \ran \IM A),
\]
which shows the first inclusion in (iii).  The second inclusion in
(iii)  follows from (i) and (ii).

(iv) This is obtained from (iii)  by means of \eqref{imre}.
\end{proof}

The next result gives necessary and sufficient conditions for a
relation $A$ to be formally domain tight.

\begin{theorem}\label{FDtight}
Let $A$ be a relation in a Hilbert space $\sH$. Then the following
statements are equivalent:
\begin{enumerate}\def\labelenumi{\rm (\roman{enumi})}
\item $A$ is formally domain tight;

\item $A\subset \RE A+\I \IM A$;

\item $(\RE A) \hplus A^* = A \hplus A^*$;

\item there exists a relation $B$ in $\sH$, such that $\dom A=\dom
B$ and $A\subset B^*$;

\item there exists a relation $C$ in $\sH$, such that $A\subset
\RE C+\I \IM C$.
\end{enumerate}
\end{theorem}

\begin{proof}
(i) $\Longrightarrow$ (ii) Let $\{f,g\} \in A$. Since $\dom
A\subset\dom A^*$, there exists $h \in \sH$ such that $\{f,h\} \in
A^*$. Then clearly
\[
 \{f,g\}=\left\{f,\frac{g+h}{2}+\I \frac{g-h}{2\I} \right\} \in
 \RE A +\I \IM A.
\]
Hence $A\subset \RE A+\I \IM A$.

(ii) $\Longrightarrow$ (iii) By Proposition~\ref{eenv} $\RE A
\subset A \hplus A^*$ and hence
\begin{equation*}\label{mult5a}
 (\RE A) \hplus A^* \subset A \hplus A^*.
\end{equation*}
Thus it is enough to prove the reverse inclusion: $A \hplus A^*
\subset (\RE A) \hplus A^*$. It suffices to prove that $A \subset
(\RE A) \hplus A^*$. Therefore, let $\{f,f'\} \in A$. Then by (ii)
$\{f,f'\} \in \RE A +\I \IM A$, so that $f \in \dom A \cap \dom
A^*$ by \eqref{14.4.2} and, in particular, $f \in \dom A^*$.
Hence, there exists an element $f''$ such that $\{f,f''\} \in
A^*$. Then
\[
 \{f,f'\}=\{2f,f'+f''\}-\{f,f''\} \in (\RE A) \hplus A^*.
\]
This completes the proof of the equality in (iii).

(iii) $\Longrightarrow$ (i) Let $f \in \dom A$, then $\{f,f'\} \I
A$ for some $f'\in \sH$. By (iii)$\{f,f'\} \in (\RE A) \hplus
A^*$, so that $f=f_1+f_2$ with $f_ \in \dom \RE A$ and $f_2 \in
\dom A^*$. It follows from \eqref{14.4.2} that $f_1 \in \dom A^*$.
Hence, $f=f_1+f_2 \in \dom A^*$. Hence (i) has been shown.

(i), (ii) $\Longrightarrow$ (iv) Define $B\okr \RE A-\I \IM A$.
Then by Proposition~\ref{reim} and \eqref{jan1+}
\begin{equation*}\label{5eqB}
 B^*\supset (\RE A)^*+\I (\IM A)^*\supset \RE A+\I \IM A \supset
A.
\end{equation*}
Furthermore, it follows from $\dom A \subset \dom A^*$ and
\eqref{14.4.2} that
\[
 \dom B=\dom \RE A=\dom \IM A =\dom A \cap \dom A^*=\dom A.
\]
Hence (iv) has been shown.

(iv) $\Longrightarrow$ (i) By taking adjoints in $\subset B^*$ one
gets $B\subset B^{**}\subset A^*$, so that $\dom A=\dom B\subset
\dom A^*$. Hence $A$ is formally domain tight.

(v) $\Longrightarrow$ (i) Taking adjoints in $A\subset \RE C+\I
\IM C$ one obtains by Proposition~\ref{reim} and \eqref{jan1+}
that
\[
 A^*\supset (\RE C+\I \IM C)^* \supset(\RE C)^* - \I (\IM C)^*
 \supset \RE C-\I \IM C.
\]
Since $\dom A\subset \dom \RE C=\dom \IM C \subset \dom A^*$, this
shows that $A$ is formally domain tight.

(ii) $\Longrightarrow$ (v) This implication is trivial.
\end{proof}

The following lemma contains a result analogous to the equivalence
of (i) and (iii) in Theorem~\ref{FDtight}. Moreover, the
identities $\RE A = \RE A^*$ and $\IM A = -\IM A^*$ will be shown
under different conditions than in Proposition \ref{reim}.

\begin{lemma}\label{rmlem}
Let $A$ be a relation in a Hilbert space $\sH$. Then
\begin{enumerate}\def\labelenumi{\rm (\roman{enumi})}
\item $\dom A^*\subset \dom A$ if and only if
\begin{equation}\label{dom*}
  (\RE A) \hplus A= A \hplus A^*;
\end{equation}

\item if $\dom A^*\subset \dom A \subset \cdom A^*$, then
\begin{equation}\label{reim*}
 \RE A = \RE A^* , \quad \IM A = -\IM A^*.
\end{equation}
\end{enumerate}
\end{lemma}

\begin{proof}
(i) Assume that $A \hplus A^* = (\RE A) \hplus A$, which, in
particular, leads to $ A^* \subset (\RE A) \hplus A$. Since $\dom \RE
A=\dom A\cap\dom A^*$ (see \eqref{14.4.2}), it follows that $\dom
A^*\subset \dom A$.

Now assume $\dom A^*\subset \dom A$. It
suffices to show that $A\hplus A^*\subset (\RE A) \hplus A$, as
the reverse inclusion is always true by Proposition~\ref{eenv}.
Let $\{f,f''\}\in A^*$, then there exists $\{f,f'\}\in A$. Hence,
\[
 \{ f,f''\}=\{2f,f'+f''\}-\{f,f'\}\in (\RE A) \hplus A.
\]
It follows that $A^*\subset (\RE A) \hplus A$, but then also
$A\hplus A^*\subset (\RE A) \hplus A$. Therefore, \eqref{dom*} has
been proved.

(ii) By Lemma \ref{easy}, it follows from $\dom A^*\subset \dom A
\subset \cdom A^*$ that $\mul A^{**}=\mul A^*$. According to
Proposition~\ref{reim} $\RE A\subset \RE A^*$. To prove the reverse
inclusion assume that $\{f,g\}\in \RE A^*$. Then for some
$\{f,g'\}\in A^*$ and $\{f,g''\}\in A^{**}$ one has $2g=g'+g''$.
Here $f\in\dom A^*\cap\dom A^{**}$ and since $\dom A^*\subset \dom
A$, one has $\{f,f'\}\in A$ for some $f'$.  Consequently, it follows
that $\{f,g''\}-\{f,f'\}\in A^{**}$ and
\[
 g''-f'\in\mul A^{**}=\mul A^*\subset \mul \RE A,
\]
where the last inclusion is due to \liczp 3 in
Proposition~\ref{reim}. Therefore,
\[
\{f,g\}=\left\{f,\frac{f'+g'}{2}\right\}
+\left\{0,\frac{g''-f'}{2}\right\}\in \RE A,
\]
and hence $\RE A^*\subset \RE A$. This proves the identity $\RE
A=\RE A^*$. The second identity in \eqref{reim*} is obtained from
the first one by means of the equalities $\RE (iA)=-\IM A$ and
$\RE (iA)^*=\IM A^*$; cf. \eqref{imre}.
\end{proof}

The following characterizations for a relation to be domain tight
are consequences of Lemma \ref{rmlem}, cf. Theorem \ref{FDtight}.

\begin{proposition} \label{reaal}
Let $A$ be a relation in a Hilbert space $\sH$.  The following
conditions are equivalent:
\begin{enumerate}
\def\labelenumi{\rm (\roman{enumi})}
\item $A$ is domain tight;

\item $(\RE A) \hplus A = (\RE A) \hplus A^*$;

\item $\RE A \hplus (\{0\} \times \ran \IM A) = A \hplus A^*$.
\end{enumerate}
In this case,
\begin{equation}\label{mult5}
 \RE A \hplus (\{0\} \times \ran \IM A) =(\RE A) \hplus A=(\RE A)
\hplus A^* = A \hplus A^*.
\end{equation}
\end{proposition}

\begin{proof}
(i) $\Longrightarrow$ (ii) If $\dom A = \dom A^*$ then $(\RE A)
\hplus A^* = A \hplus A^*$ by part (iii) in Theorem~\ref{FDtight},
while $(\RE A) \hplus A= A \hplus A^*$ due to \eqref{dom*} in
Lemma \ref{rmlem}. This gives the identity in (ii).

(ii) $\Longleftarrow$ (i) If $(\RE A) \hplus A = (\RE A) \hplus
A^*$, then, in particular, $A \subset (\RE A) \hplus A^*$. Since,
by \eqref{14.4.2}, $\dom \RE A=\dom A\cap\dom A^*$, it follows
that $\dom A \subset \dom A^*$. The inclusion $\dom A^* \subset
\dom A$ follows in a similar way. Hence, $A$ is domain tight.

(i) $\Longrightarrow$ (iii) In view of the second inclusion in
(iii) of Proposition~\ref{eenv} it suffices to show that the
inclusion $A \hplus A^*\subset \RE A \hplus (\{0\} \times \ran \IM
A)$ when $A$ is domain tight. Since $A$ is domain tight, $A^*$ is
formally domain tight, cf. Remark \ref{tight}. Hence,
Theorem~\ref{FDtight} implies
\[
 A\subset \RE A +\I \IM A,\quad A^*\subset \RE A^* +\I \IM A^*=\RE
 A -\I \IM A,
\]
where the last identity is obtained from Lemma~\ref{rmlem}. It
remains to use (i)  in Proposition~\ref{eenv} to get the claimed
inclusion.

(iii) $\Longrightarrow$ (i) The equality in (iii) implies that
$\dom A\cap\dom A^*=\dom A+\dom A^*$, cf. \eqref{14.4.2}. This
last identity is clearly equivalent to $\dom A=\dom A^*$.

Finally, the equalities stated in \eqref{mult5} are clear from the
above arguments.
\end{proof}

\subsection{Cartesian decompositions of relations}

A relation $A$ in a Hilbert space $\sH$ is said to have a
\textit{Cartesian decomposition}\index{decomposition of
relations!Cartesian} if there are two symmetric relations $A_1$
and $A_2$ in $\sH$ such that
\begin{equation}\label{e3.3}
 A=A_1+\I A_2,
\end{equation}
with the operatorwise sum defined as in \eqref{jan1}, so that $\dom
A=\dom A_1 \cap \dom A_2$ and $\mul (A_1+A_2)=\mul A_1+\mul A_2$,
cf. \eqref{dommul}. In particular, if $A$ is an operator, then
$A_1$ and $A_2$ in \eqref{e3.3} are operators.  The Cartesian
decomposition for operators is extensively considered in \cite{StSz}.

\begin{example}
Let $A$ be a maximal sectorial relation in $\sH$ with vertex at the
origin and semiangle $\alpha$, cf. \eqref{sect}. Then there exist a
nonnegative selfadjoint relation $H$ in $\sH$ and a selfadjoint
operator $B \in \boldsymbol{B}(\sH)$ with $\cran B \subset (\mul
A)^{\perp}$ and $\| B \| \le \tan \alpha$, such that
\[
A = H^{\half} (I + \I B) H^{\half},
\]
cf. \cite{HSSW??}, \cite{Ka}, \cite{Sand}. Clearly, the relations
$H$ and $H^{\half} B H^{\half}$ are symmetric relations, but $H+ \I
H^{\half} B H^{\half}$, the operatorlike sum of these relations,
need not be equal to $A$. In general, the inclusion
\[
 H+ \I H^{\half}B H^{\half} \subset A
\]
holds. There is equality if, for instance,  $\ran B \subset \dom H^\half$.
\end{example}

\begin{proposition} \label{t3.6}
Let $A$ be a relation in a Hilbert space $\sH$, let $A$ have a
Cartesian decomposition \eqref{e3.3}, and define the relation $B$
by $B=A_1-\I A_2$. Then $A$ and $B$ have the same domain $\dom
B=\dom A$, they are formally domain tight, and they form a dual
pair:
\[
  B \subset A^*, \quad A \subset B^*.
\]
Moreover, the symmetric components $A_1$ and $A_2$ satisfy
\begin{equation} \label{incl}
A_1\cap( \dom A \times \sH) \subset \RE A, \quad A_2\cap( \dom
A\times \sH) \subset \IM A,
\end{equation}
and $A_1\pm\I A_2\subset \RE A\pm\I \IM A$.
\end{proposition}

\begin{proof}
If $A$ has a Cartesian decomposition of the form \eqref{e3.3},
then clearly $A$ and $B$ have the same domain. By \eqref{jan1+}
and the symmetry of $A_1$ and $A_2$ it follows that
\begin{equation}\label{e3.4a}
A^* =(A_1+\I A_2)^* \supset A_1^* -\I A_2^* \supset A_1-\I A_2=B.
\end{equation}
Hence, $\dom A=\dom B \subset \dom A^*$, so that $A$ is formally
domain tight. A similar argument shows that $B$ is formally domain
tight. Moreover, \eqref{e3.4a} shows that $B \subset A^*$, which
also leads to $A \subset A^{**} \subset B^*$; hence $A$ and $B$
form a dual pair.

In order to show the first inclusion in \eqref{incl}, let
$\{f,f_1'\}\in A_1$ with $f\in \dom A$. Then there exists $f_2'\in
\sH$ such that $\{f,f_2'\}\in A_2$. Hence, $\{f,f_1'+\I f_2'\}\in A$
due to \eqref{e3.3} and $\{f,f_1'-\I f_2'\}\in A^*$ due to
\eqref{e3.4a}, so that $\{f,f_1'\}\in \RE A$. Thus, $A_1\cap( \dom A
\times \sH) \subset \RE A$  and then in view of \eqref{imre} the
second inclusion in \eqref{incl} follows as well.

The statements $A_1\pm\I A_2\subset \RE A\pm\I \IM A$ follow
directly from the inclusions in \eqref{incl}.
\end{proof}

A formally domain tight relation $A$ satisfies $A\subset \RE A+\I
\IM A$, cf. Theorem~\ref{FDtight}. By means of Cartesian
decompositions this inclusion can be made more precise, yielding
some characterizations for a relation $A$ to be formally domain
tight.

\begin{theorem}\label{cart}
Let $A$ be a relation in a Hilbert space $\sH$ and let the extension
$A_\infty$ of $A$  be as defined in \eqref{+SF}. Then the following
conditions are equivalent:
\begin{enumerate}
\def\labelenumi{\rm (\roman{enumi})}
\item $A$ is formally domain tight;

\item $A$ admits a Cartesian decomposition $A=A_1+\I A_2$ for some
symmetric relations $A_1$ and $A_2$ in $\sH$;

\item $A_\infty$ admits the Cartesian decomposition
\begin{equation}\label{cart1}
 A_\infty = \RE A + \I \IM A.
\end{equation}
\end{enumerate}
\end{theorem}

\begin{proof}
(ii) $\Longrightarrow$ (i) This implication follows from
Proposition ~\ref{t3.6}.

(iii) $\Longrightarrow$ (i) Since $A\subset A_\infty$ this
implication follows from Theorem~\ref{FDtight}. Another approach
is that $A_\infty$ is formally domain tight by Proposition
~\ref{t3.6}, but then $A$ is formally domain tight by Proposition
\ref{codi}.

(i) $\Longrightarrow$ (iii) Let $A$ be formally domain tight. Then
$A \subset \RE A + \I \IM A$ by Theorem~\ref{FDtight}.
Furthermore, (iii) in Proposition~\ref{reim} shows that
$\{0\}\times \mul A^* \subset \RE A + \I \IM A$. Hence, the
inclusion $A_\infty=A \hplus (\{0\} \times \mul A^*) \subset \RE A
+ \I \IM A$ in the identity \eqref{cart1} has been shown. Now the
reverse inclusion will be shown. An arbitrary element $\{f,g\} \in
\RE A+\I \IM A$ is given by
\begin{equation}\label{help}
 \{f,g\}=\left\{f,\frac{f'+f''}{2}+\I \frac{h'-h''}{2\I} \right\},
\end{equation}
where $\{f,f'\}, \{f,h'\} \in A$ and $\{f,f''\}, \{f,h''\} \in
A^*$. Then
   \begin{gather}  \label{helpB}
\begin{split}
 2\{f,g\}&=\{2f, f'+f''+h'-h''\} \\&=\{f,f'\}+\{f,h'\}+
\{0,f''-h''\}
 \\ &\in A \hplus (\{0\} \times \mul A^*).
\end{split}
\end{gather}
Hence the inclusion $\RE A + \I \IM A \subset A_\infty$ in the
identity \eqref{cart1} has been shown.

(i) $\Longrightarrow$ (ii) By Theorem \ref{adjoin1} $A^*$ can be
decomposed as $ A^*=(A^*)_{\rm op} \hplus (A^*)_{\rm mul}$; see
also Corollary~\ref{closeddec}. Here $(A^*)_{\rm op}$ is an
operator with $\dom (A^*)_{\rm op}=\dom A^*$. Now define
\[
 A_1 \okr \frac{1}{2}(A+(A^*)_{\rm op}), \quad A_2 \okr
 \frac{1}{2\I}(A-(A^*)_{\rm op}),
\]
compare \eqref{re}, \eqref{im}. The assumption $\dom A\subset \dom
A^*$ shows that $\dom A_1=\dom A_2=\dom A$; therefore $A_1\subset
\RE A$ and $A_2\subset \IM A$, so that $A_1$ and $A_2$ are
symmetric relations. The inclusion $A\subset A_1+\I A_2$ can be
proved in the same way as the implication (i) $\Longrightarrow$
(ii) in Theorem~\ref{FDtight}, when $\dom A\subset \dom A^*=\dom
(A^*)_{\rm op}$ is used. The reverse inclusion $A_1+\I A_2\subset
A$ can be seen with a similar, but simpler, calculation as used in
\eqref{help}, \eqref{helpB}. Therefore, the equality $A=A_1+\I
A_2$ holds.
\end{proof}

Domain tight relations can now be characterized via Cartesian
decompositions as follows.

\begin{theorem}\label{cart+}
Let $A$ be a relation in a Hilbert space $\sH$ and let the
extension $A_\infty$ be as defined in \eqref{+SF}. Then the
following conditions are equivalent:
\begin{enumerate}
\def\labelenumi{\rm (\roman{enumi})}
\item $A$ is domain tight;

\item $A_\infty$ and $(A^*)_\infty$ admit the Cartesian
decompositions
\begin{equation}
\label{cart11} A_\infty = \RE A + \I \IM A, \quad (A^*)_\infty =
\RE A -\I \IM A;
\end{equation}

\item $A$ and $A^*$ satisfy
\begin{equation}
\label{cart11B} A \subset \RE A + \I \IM A, \quad A^*= \RE A -\I
\IM A;
\end{equation}

\item for some symmetric relations $A_1$ and $A_2$ in $\sH$ one
has
\begin{equation}
\label{cart12} A = A_1 + \I A_2, \quad A^*= \RE A -\I \IM A.
\end{equation}
\end{enumerate}
\end{theorem}

\begin{proof}
(i) $\Longrightarrow$ (ii) Assume that $A$ is domain tight. Then
$A$ and $A^*$ are formally domain tight, cf. Remark \ref{tight}.
The first identity in \eqref{cart11} holds by \eqref{cart1} in
Theorem~\ref{cart}. Since $A$ is domain tight, Lemma~\ref{rmlem}
shows that $\RE A^* =\RE A$ and $\IM A^* =- \IM A$. Now
Theorem~\ref{cart} (applied with $A^*$) gives the second identity
in \eqref{cart11}:
\[
 (A^*)_\infty = \RE A^* +\I \IM A^*=\RE A -\I \IM A^.
\]

(ii) $\Longrightarrow$ (i) It follows from \eqref{14.4.2} and the
Cartesian decompositions in \eqref{cart11} that
\begin{equation*}
  \dom A=\dom A_\infty=\dom A \cap \dom A^*=\dom (A^*)_\infty=\dom
A^*,
\end{equation*}
which shows that $A$ is domain tight.

(i), (ii) $\Longrightarrow$ (iii) The inclusion in \eqref{cart11B}
is clear from \eqref{cart11} as $A\subset A_\infty$. Since $A$ is
domain tight, $\mul A^{**} = \mul A^*$ and therefore
$(A^*)_\infty=A^*\hplus (\{0\} \times \mul A^{**})=A^*$. Thus the
second identity in \eqref{cart11B} is also immediate from
\eqref{cart11}.

(iii) $\Longrightarrow$ (iv) The inclusion in \eqref{cart11B}
implies that $A$ is formally domain tight. Hence, the first
identity in \eqref{cart12} is obtained from part (ii) in
Theorem~\ref{cart}.

(iv) $\Longrightarrow$ (i) The first identity in \eqref{cart12}
shows that $A$ is formally domain tight by Theorem~\ref{cart},
while the second identity in \eqref{cart12} implies that $\dom
A^*\subset \dom \RE A=\dom \IM A=\dom A\cap\dom A^*$, cf.
\eqref{14.4.2}. Hence, $A$ is domain tight.
\end{proof}

In the above characterization some of the conditions do not look
symmetric. By turning to a more special class of domain tight
relations the description will be more symmetric.

\begin{theorem}\label{cart++}
Let $A$ be a relation in a Hilbert space $\sH$. Then the following
conditions are equivalent:
\begin{enumerate}
\def\labelenumi{\rm (\roman{enumi})}
\item $A$ is domain tight and $\mul A = \mul A^*$;

\item $A$ and $A^*$ admit the Cartesian decompositions
\begin{equation}\label{cart13}
A = \RE A + \I \IM A, \quad A^* = \RE A -\I \IM A;
\end{equation}

\item $A$ and $A^*$ admit the Cartesian decompositions
\begin{equation}\label{cart14}
A = A_1 + \I A_2, \quad A^*= A_1 - \I A_2
\end{equation}
for some symmetric relations $A_1$ and $A_2$ in $\sH$.
\end{enumerate}
\end{theorem}

\begin{proof}
(i) $\Longrightarrow$ (ii) The assumption $\mul A = \mul A^*$
implies that $A_\infty=A$. Therefore, the statement follows from
\eqref{cart11} and \eqref{cart12}.

(ii) $\Longrightarrow$ (iii) In \eqref{cart13} the relations $\RE
A$ and $\IM A$ are symmetric. Hence, this implication is trivial.

(iii) $\Longrightarrow$ (i) It is clear from \eqref{cart14} that
$\dom A=\dom A_1\cap \dom A_2=\dom A^*$ and $\mul A=\mul A_1 +
\mul A_2=\mul A^*$.
\end{proof}

The special domain tight relations in Theorem \ref{cart++} can be
also characterized by means of decomposable domain tight
relations; cf. Proposition \ref{mulba}.

\printindex

\end{document}